\documentclass[11pt]{article}
\usepackage[tbtags]{amsmath}
\usepackage{amssymb}
\usepackage{amsthm}
\usepackage[misc]{ifsym}
\usepackage{cases}
\usepackage{mathrsfs}
\usepackage{subfigure}
\usepackage{mathtools}
\usepackage{bbm}
\usepackage{enumerate}
\usepackage{xcolor}
\usepackage{float}
\usepackage{graphicx}
\usepackage{epstopdf}

\numberwithin{equation}{section} 
\normalsize
\title{\bf Stochastic Linear-Quadratic Optimal Control Problems with Markovian Regime Switching and $H_\infty$ Constraint under Partial Information \thanks{This work is supported by National Natural Science Foundations of China (12471419, 12271304), and Shandong Provincial Natural Science Foundation (ZR2024ZD35).}}
\author{\normalsize Na Xiang\thanks{\it School of Mathematics, Shandong University, Jinan 250100, P.R. China, E-mail: 202211967@mail.sdu.edu.cn} , Jingtao Shi\thanks{\it School of Mathematics, Shandong University, Jinan 250100, P.R. China, E-mail: shijingtao@sdu.edu.cn}}
\newtheorem{mypro}{Proposition}[section]
\newtheorem{mythm}{Theorem}[section]
\newtheorem{mydef}{Definition}[section]
\newtheorem{mylem}{Lemma}[section]
\newtheorem{Remark}{Remark}[section]
\newtheorem{mycor}{Corollary}[section]
\begin{document}
	\maketitle
	
\noindent{\bf Abstract:}\quad This paper is concerned with a stochastic linear-quadratic optimal control problem of Markovian regime switching system with model uncertainty and partial information, where the information available to the control is based on a sub-$\sigma$-algebra of the filtration generated by the underlying Brownian motion and the Markov chain. Based on $H_\infty$ control theory, we turn to deal with a soft-constrained zero-sum linear-quadratic stochastic differential game with Markov chain and partial information. By virtue of the filtering technique, the Riccati equation approach, the method of orthogonal decomposition, and the completion-of-squares method, we obtain the closed-loop saddle point of the zero-sum game via the optimal feedback control-strategy pair. Subsequently, we prove that the corresponding outcome of the closed-loop saddle point satisfies the $H_\infty$ performance criterion. Finally, the obtained theoretical results are applied to a stock market investment problem to further illustrate the practical significance and effectiveness.

\vspace{2mm}


\vspace{2mm}


\section{Introduction}

The \emph{stochastic linear-quadratic} (SLQ) optimal control problems constitute a class of extremely important optimal control problems in stochastic optimization theory, since they can model many problems in applications, and more importantly, many nonlinear control problems can be reasonably approximated by the SLQ problems. On the other hand, solutions of SLQ problems exhibit elegant properties due to their simple and nice structures. 
The study of SLQ optimal control problems originated from the works of Kushner \cite{Kushner62} and Wonham \cite{Wonham68} in the 1960s. Bismut \cite{Bismut76} studied SLQ optimal control problems with random coefficients using functional analysis techniques and obtained the optimal control in a random feedback form. Chen et al. \cite{Chen98} pioneered the investigation of SLQ optimal control problems with indefinite control weight costs, which presented new stochastic Riccati equations involving complicated nonlinear terms. The monograph by Yong and Zhou \cite{Yong99} studied the SLQ optimal control problems via stochastic Riccati equations and presented a comprehensive survey on the SLQ optimal control theory. Tang \cite{Tang03} discussed the existence and uniqueness of the associated stochastic Riccati equation for general SLQ optimal control problems with random coefficients and state control dependent noise via the stochastic flows, which solved Bismut and Peng's long-standing open problem. There has been extensive research on SLQ optimal control problems, and significant extensions have been developed based on it, such as open-loop and closed-loop solvabilities, mean-field type, risk-sensitive, time-inconsistent, and jump-diffusion models, etc. For more recent developments on SLQ optimal control problems, readers may refer to \cite{Li18,Li16,Li25,Oh25,Sun16,Sun20,Sun24,Tang15,Xiong25}. 

In practice, model uncertainty is widespread in the control systems. One extension of the SLQ optimal control problem involves the presence of external unknown disturbances in the state equation. Along this line, the stochatic $H_2/H_\infty$ control problem has been extensively studied via various approaches, including the \emph{linear matrix inequality} (LMI) technique, convex optimization approach, the Nash game approach.  Among these different methods to this class of worst-case design problems, the one that uses the framework of dynamic differential game theory seems to be the most natural. The Nash game approach is a classical technique to deal with the $H_2/H_\infty$ control problem. By constructing two performances associated with $H_\infty$ robustness and $H_2$ optimization, respectively, the $H_2/H_\infty$ control can be converted into finding the Nash equilibrium point. 
Limebeer et al. \cite{Limebeer94} was the first pioneering work to solve the deterministic mix $H_2/H_\infty$ control problem by means of the theory of nonzero-sum games. The resulting $H_2/H_\infty$ controller admits a linear state feedback representation characterized by the solution to a pair of cross-coupled Riccati equations. Chen and Zhang \cite{Chen04} generalized the work of \cite{Limebeer94} to the stochastic scenario with state-dependent noise, including both finite and infinite horizon cases and shown that finite (infinite) horizon stochastic $H_2/H_\infty$ control has close relation to a pair of coupled differential Riccati equations (a pair of coupled algebraic Riccati equations, correspondingly). Zhang et al. \cite{Zhang05} extended the result of  \cite{Chen04} to the system with all state, control and external disturbance dependent noise, where four cross-coupled matrix-valued differential and algebraic equations were introduced to express the finite and infinite horizon stochastic $H_2/H_\infty$ control. Then, Zhang et al. \cite{Zhang06-1} considered a general finite horizon nonlinear stochastic $H_2/H_\infty$ control with state, control and external disturbance dependent noise, and proved that the mixed $H_2/H_\infty$ control is associated with the four cross-coupled Hamilton–Jacobi equations. Subsequently, the stochastic $H_2/H_\infty$ control theory has experienced rapid development. For more details on the stochastic $H_2/H_\infty$ control problem, readers may refer to \cite{Wang21,Wang23,Xiang25-1,Zhang24,Zhang17}. 
On the other hand, $H_\infty$ optimal control problem has also been extensively investigated. In fact, $H_\infty$ optimal control problem is a minimax optimization problem, and hence a zero-sum game, where the controller can be viewed as the minimizing player and disturbance as the maximizing player. It is closely associated with the zero-sum game, and compared with the $H_2/H_\infty$ control problem, it can avoid solving the cross-coupled Riccati equations. The monograph by Ba\c{s}ar and Bernhard \cite{Basar95}  systematically elaborates on the well-known relationship between $H_\infty$ optimization and LQ zero-sum differential games. Pan and Ba\c{s}ar \cite{Pan95} studied the $H_\infty$ optimal control problem with a Markov chain in both finite and infinite horizon cases, where the controller has access to perfect or imperfect continuous state measurements. Ugrinovskii \cite{Ugrinovskii98} considered the stochastic $H_\infty$ control problem with state-dependent noise and obtained the state feedback controller that guarantees a prescribed level of disturbance attenuation for all admissible stochastic uncertainties. Hinrichsen and Pritchard \cite{Hinrichsen98} investigated the stochastic $H_\infty$ control problem with state- and control-dependent noise and proved a bounded real lemma for stochastic systems with deterministic and stochastic perturbations. van den Broek et al. \cite{van03} studied the existence of Nash equilibria in LQ differential games on an infinite planning horizon, where the system is disturbed by deterministic noise and the strategy spaces are of the static linear feedback type and considered the soft-constrained and hard-bounded cases. Zhang and Chen \cite{Zhang06-2} discussed the stochastic $H_\infty$ control for nonlinear systems with both state- and disturbance-dependent noise, including finite and infinite horizon cases. By means of two kinds of Hamilton–Jacobi equations, both infinite and finite horizon nonlinear stochastic $H_\infty$ control designs were developed. There are various works on the stochastic $H_\infty$ control problem, one can refer to \cite{Huang17,Liang22,Mukaidani22,Xiang25-2} and the references cited therein.

Markov regime switching models have been widely used in stochastic optimal control problems in recent years. In financial engineering, bank interest rates, stock appreciation and volatility are modulated by Markov processes (e.g., bear market and bull market), which can more directly characterize the factors and events that do not change frequently but exert a significant impact on the long-term trends of the system. There is a vast literature on SLQ optimal control problems involving Markov chains, see, for examples, \cite{Chen25,Li02,Li03,Wen23,Wu25,Zhang21}. For the $H_2/H_\infty$ control problems with regime switching jumps, one can refer to \cite{Sheng14,Wang21,Zhang25}. For the $H_\infty$ control problems with regime switching jumps, one can refer to \cite{Lin09,Pan95,Wang17}. In addition, regime switching model has also been widely applied in finance, such as the mean-variance portfolio selection and investment-consumption problem, see, for examples, \cite{Yin04,Zhang12,Zhou03}.

We note that in the vast majority of the aforementioned literature, player can only make decisions with complete information. However, in practical problems, we often lack complete information and have to make decisions based on partial information. For a stochastic control problem, if the available information comes from an observation equation, the corresponding control problem is referred to as a partially observed optimal control problem, see,  for examples, \cite{Li20,Oh25,Si25,Sun23,Wang15}. On the other hand, if the available information is directly given by an abstract sub-filtration that does not depend on the control, the corresponding control problem is termed a stochastic optimal control problem with partial information, see \cite{Huang09,Li25,Moon25,Shi16,Zheng21}.

Motivated by the aforementioned literature, we consider an SLQ optimal control problem with Markovian regime switching and model uncertainty under partial information, where both the drift and diffusion terms of the state equation and the cost functional contain the control and the external unknown disturbance. It is required that the information available to the control is based on a sub-$\sigma$-algebra of the filtration generated by the underlying Brownian motion and the Markov chain. Based on $H_\infty$ control theory and by virtue of the close relationship between $H_\infty$ control problems and LQ zero-sum stochastic differential games, we transform this problem into a soft-constrained zero-sum LQ stochastic differential game with Markov chain and partial information, where the control can be viewed as the minimizing player with partial information $\mathbb{G}$ and the disturbance as the maximizing player with complete information $\mathbb{F}$. As with other existing literature, we present the definition of the closed-loop saddle point of the LQ zero-sum game with Markov chain and partial information (see Definition \ref{Def-CLSP}). To preliminarily guarantee the existence of the $H_\infty$ optimal control, we analyze the disturbance attenuation parameter $\gamma$, discuss the properties of the corresponding open-loop upper value $\overline{V}_\gamma^0(0,0,i)$ with respect to the disturbance attenuation parameter $\gamma$ (see Proposition \ref{prop-uv} and Corollary \ref{cor1}), and derive a threshold level $\gamma^*$ for the $H_\infty$ control problem (see Proposition \ref{prop2.1}). Inspired by the work of Yu \cite{Yu15}, we formulate two auxiliary problems to obtain the closed-loop saddle point of the zero-sum game by seeking the optimal feedback control-strategy pair in a closed-loop form. Firstly, using the filtering technique, we obtain the corresponding filtering equation and the equation that the difference satisfies. Further, by means of Riccati equation approach and orthogonal decomposition of the state process and the disturbance, the soft-constrained cost functional is expressed as the sum of two parts: one part is a functional of the control $u$, the filtering state and disturbance processes (i.e., $\hat{x}$ and $\hat{v}$), which lies in the common information $\mathbb{G}$, and the other part is independent of the choice of the control and only related to the  differences of the state process and the disturbance (i.e., $\tilde{x}$ and $\tilde{v}$), which lies in the private information of the disturbance $\tilde{\mathbb{G}}:=\mathbb{F}\setminus\mathbb{G}$. Due to the two equivalent forms of the Riccati equation (see Lemma \ref{lem2}), we perform the completion-of-squares method for the two cost functionals separately, derive the feedback form of $(u^*,\hat{\alpha}_2^*(u))$ and $(\alpha_1^*(\hat{v}),\hat{v}^*)$ with respect to $\hat{x}$ and the feedback form of $\tilde{v}^*$ with respect to $\tilde{x}$ (see Lemma \ref{lem3}, Theorem \ref{Thm1}, and Theorem \ref{Thm2}). Then, by proving the equivalence of the two {\it backward ordinary differential equations} (BODEs) and introducing an algebraic equations system, we give the conditions of the existence of the optimal control-strategy pair $(u^*,v^*)$ (see Proposition \ref{equivalence} and Theorem \ref{Thm-OCS}). We then verify that the optimal control-strategy pair is the outcome of the closed-loop saddle point to the zero-sum game (see Theorem \ref{Thm-CL}), and prove that it satisfies the $H_\infty$ performance criterion (see Theorem \ref{Thm-R}). To conclude, the effectiveness of the theoretical results is confirmed via numerical simulations.

The main contributions of this paper can be summarized as follows.

(1) Firstly, we study a class of SLQ optimal control problems with Markovian regime switching and model uncertainty under partial information, where the control variable and the external unknown disturbance enter the drift term and diffusion terms of the state equation, respectively, and the information available to the control is based on a sub-$\sigma$-algebra of the filtration generated by the underlying Brownian motion and the Markov chain.

(2) Secondly, we discuss the properties of the open-loop upper value for the homogeneous system with zero initial time and zero initial state with respect to the disturbance attenuation level $\gamma$, and further derive the threshold level for the $H_\infty$ control problem from two distinct aspects to preliminarily guarantee the existence of the robust $H_\infty$ optimal control.

(3) Thirdly, by virtue of $H_\infty$ control theory, the closed-loop solvability of a soft-constrained zero-sum LQ stochastic differential game with Markov chain and partial information is investigated. We present two equivalent forms of the Riccati equation. With the help of the filtering technique, the method of orthogonal decomposition, and the completion-of-squares method, we obtain two optimal control-strategy pairs for the control and the disturbance. By proving the equivalence of the two BODEs and introducing an algebraic equations system, we obtain an optimal feedback control-strategy pair for the zero-sum game in a closed-loop form, and then derive the desired closed-loop saddle point of the zero-sum game. Subsequently, we prove that the corresponding outcome of the closed-loop saddle point satisfies the $H_\infty$-performance.

(4) Finally, to further demonstrate the effectiveness and practical applicability of the theoretical results, we solve a stock market investment problem, characterize the bear and bull market states in the stock market by a two-state Markov chain, consider the impacts of bear and bull markets in the stock market on the investment strategies of individual investor and the worst-case disturbance, and then conduct a sensitivity analysis with respect to the disturbance attenuation level.

The rest of the paper is organized as follows. In Section 2, we introduce some preliminary notations and formulate the soft-constrained zero-sum LQ stochastic differential game with Markov chain and partial information. Section 3 discusses the  disturbance attenuation parameter. The robust $H_\infty$ optimal control and the worst-case disturbance of Problem (R-SCG) are obtained in Section 4. A numerical example is given in Section 5. Section 6 concludes this paper.

\section{Preliminaries and problem formulation}

Let $T\in\left(0,\infty\right)$ be a fixed and deterministic time horizon. Let ($\Omega,\mathcal{F},\mathbb{F},\mathbb{P}$) be a complete filtered probability space on which two independent standard one-dimensional Brownian motions $W$, $\overline{W}$ and a continuous-time finite state space Markov chain $\alpha\equiv\{\alpha(t);0\leq t<\infty\}$ are defined. We assume that Brownian motion ($W,\overline{W}$) and Markov chain $\alpha$ are mutually independent. The filtration $\mathbb{F}\equiv\{\mathcal{F}_t\}_{t\geq 0}$ is generated jointly by the Brownian motion ($W,\overline{W}$) and the Markov chain $\alpha$ as follows:
\begin{equation*}
	\mathcal{F}_t:=\sigma\{W(s),\overline{W}(s);0\leq s\leq t\}\vee\sigma\{\alpha(s);0\leq s\leq t\}\vee\mathcal{N}(\mathbb{P}),\\
\end{equation*}
where $\mathcal{N}(\mathbb{P})$ denotes the collection of all $\mathbb{P}$-null sets in $\mathcal{F}$. Let $\mathbb{G}\equiv\{\mathcal{G}_t\}_{t\geq 0}$ be the natural filtration of $W$ and $\alpha$ augmented by all $\mathbb{P}$-null sets in $\mathcal{G}$.
We assume that the Markov chain $\alpha$ takes values in a finite state space $\mathcal{S}\equiv\{1,2,\dots,D\}$ with some positive integer $D$, which is also homogeneous and irreducible. The generator of $\alpha$ is a $D\times D$ matrix $\Lambda:=\left
(\lambda_{ij}\right)_{i,j=1}^D$. For each $i,j\in\mathcal{S}$, $\lambda_{ij}$ is the constant transition intensity of the chain from state $i$ to state $j$ at time $t$. Note that $\lambda_{ij}\ge0$ for $i\neq j$ and $\sum_{j=1}^{D}\lambda_{ij}=0$, so $\lambda_{ii}\leq0$. In what follows, for each $i,j\in\mathcal{S}$ with $i\neq j$, we further suppose that $\lambda_{ij}>0$, so $\lambda_{ii}<0$. Set 
\begin{equation*}
	N_{ij}(t):=\sum_{0\leq s\leq t} \mathbbm{1}_{\{\alpha(s-)=i\}}\mathbbm{1}_{\{\alpha(s)=j\}},\\
\end{equation*}
the counting process $N_{ij}(t)$ counts the number of jumps from state $i$ to state $j$ up to time $t$. The process
\begin{equation*}
	\widetilde{N}_{ij}(t):=N_{ij}(t)-\tilde{\lambda}_{ij}(t)=N_{ij}(t)-\int_{0}^{t}\lambda_{ij}\mathbbm{1}_{\{\alpha(s-)=i\}}ds
\end{equation*}
is a purely discontinuous, square-integrable martingale (compensated measure). 

Throughout the paper, let $\mathbb{R}^n$ denote the \emph{n}-dimensional Euclidean space with standard Euclidean norm $\vert\cdot\vert $ and standard Euclidean inner product $\left\langle\cdot,\cdot\right\rangle $. The transpose of a vector (or matrix) $x$ is denoted by $\mathbf{\emph{x}}^\top$. $\mbox{Tr}(A)$ denotes the trace of a square matrix $A$. Let $\mathbb{R}^{n\times m}$ be the Hilbert space consisting of all $n\times m$-matrices with the inner product $\left\langle A,B\right\rangle := \mbox{Tr}(AB^\top$) and the Frobenius norm $\vert A \vert:=\langle A,A\rangle^\frac{1}{2}$. Denote the set of symmetric $n\times n$ matrices with real elements by $\mathbb{S}^n$. If $M\in\mathbb{S}^n$ is positive (semi-) definite, we write $M > (\geq) 0$. If there exists a constant $\delta>0$ such that $M\geq\delta I$, we write $M\gg0$. 
For a given Hilbert space $\mathbb{H}$, if $\xi$: $\Omega\to\mathbb{H}$ is an $\mathcal{F}_T$-measurable, square-integrable random variable, we denote $\xi\in L_{\mathcal{F}_T}^2(\Omega;\mathbb{H})$; 
if $\phi$: $[t,T]\times\Omega\to\mathbb{H}$ is $\mathbb{F}$-progressively measurable s.t. $\mathbb{E}\int_t^T |\phi(s)|^2 ds < \infty$, we denote $\phi\in L_{\mathbb{F}}^2(t,T;\mathbb{H})$; 
if $\phi$: $[t,T]\times\Omega\to\mathbb{H}$ is $\mathbb{F}$-adapted, continuous s.t. $\mathbb{E}\big[\sup_{s\in[t,T]}|\phi(s)|^2\big]  < \infty$, we denote $\phi\in L_{\mathbb{F}}^2(\Omega;C([t,T];\mathbb{H}))$.

For $t\in[0,T)$, we consider the following controlled Markovian regime switching linear \emph{stochastic differential equation} (SDE) on the finite horizon $[t,T]$:
\begin{equation}\label{state eq}
\left\{\begin{aligned}
	dx(s)=&\big[A(s,\alpha(s))x(s)+B_1(s,\alpha(s))u(s)+B_2(s,\alpha(s))v(s)+b(s,\alpha(s))\big]ds\\
	&+\big[C(s,\alpha(s))x(s)+D_1(s,\alpha(s))u(s)+D_2(s,\alpha(s))v(s)+\sigma(s,\alpha(s))\big]dW(s)\\
	&+\big[\bar{C}(s,\alpha(s))x(s)+\bar{D}_1(s,\alpha(s))u(s)+\bar{D}_2(s,\alpha(s))v(s)+\bar{\sigma}(s,\alpha(s))\big]d\overline{W}(s),\\
	x(t)=&\ \xi,\quad\alpha(t)=i,\\
\end{aligned}\right.
\end{equation}
where $\xi\in L_{\mathcal{F}_t}^2(\Omega;\mathbb{R}^n)$ and the coefficients $A(\cdot,j)$, $B_1(\cdot,j)$, $B_2(\cdot,j)$, $C(\cdot,j)$, $D_1(\cdot,j)$, $D_2(\cdot,j)$, $\bar{C}(\cdot,j)$, $\bar{D}_1(\cdot,j)$, $\bar{D}_2(\cdot,j)$ are given deterministic matrix-valued functions of proper dimensions, and $b(\cdot,j)$, $\sigma(\cdot,j)$, $\bar{\sigma}(\cdot,j)$ are deterministic $\mathbb{R}^n$-valued functions, $j\in\mathcal{S}$. Moreover, $x(\cdot)\in\mathbb{R}^n$ is the state process, $u(\cdot)\in\mathbb{R}^m$ is the control process, and $v(\cdot)\in\mathbb{R}^{n_v}$ is the external unknown disturbance to characterize the model uncertainty and represents the influence of the external environment on the decision-maker.

Moreover, we consider the cost functional:
\begin{equation}\label{cost}
\begin{aligned}
	&J(t,\xi,i;u(\cdot),v(\cdot))=\mathbb{E}\bigg\{\langle G(T,\alpha(T))x(T),x(T)\rangle+2\langle g(T,\alpha(T)),x(T)\rangle\\
	&\quad+\int_t^T\Big[  \langle Q(s,\alpha(s))x(s),x(s)\rangle+\langle R_1(s,\alpha(s))u(s),u(s)\rangle+\langle R_2(s,\alpha(s))v(s),v(s)\rangle\\
	&\qquad\qquad+2\langle S_1(s,\alpha(s))x(s),u(s)\rangle+2\langle S_2(s,\alpha(s))x(s),v(s)\rangle+2\langle q(s,\alpha(s)),x(s)\rangle \\
	&\qquad\qquad+2\langle \rho_1(s,\alpha(s)),u(s)\rangle+2\langle \rho_2(s,\alpha(s)),v(s)\rangle \Big]ds\bigg\},
\end{aligned}
\end{equation}
where $G(T,j)\in\mathbb{S}^n$, $g(T,j)\in\mathbb{R}^n$, and the coefficients $Q(\cdot,j)$, $R_1(\cdot,j)$, $R_2(\cdot,j)$, $S_1(\cdot,j)$, $S_2(\cdot,j)$ are given deterministic matrix-valued functions of proper dimensions, and $q(\cdot,j)$, $\rho_1(\cdot,j)$, $\rho_2(\cdot,j)$ are deterministic vector-valued functions of proper dimensions, $j\in\mathcal{S}$.

According to $H_\infty$ control theory, we introduce the parameterized (in the disturbance attenuation level $\gamma>0$) family of cost functionals:
\begin{equation}\label{cost1}
\hspace{-3mm}\begin{aligned}
	&J_\gamma(t,\xi,i;u(\cdot),v(\cdot))=\mathbb{E}\bigg\{\langle G(T,\alpha(T))x(T),x(T)\rangle+2\langle g(T,\alpha(T)),x(T)\rangle\\
	&\quad+\int_t^T\Big[  \langle Q(s,\alpha(s))x(s),x(s)\rangle+\langle R_1(s,\alpha(s))u(s),u(s)\rangle+\big\langle (R_2(s,\alpha(s))-\gamma^2I)v(s),v(s)\big\rangle\\
	&\qquad\qquad+2\langle S_1(s,\alpha(s))x(s),u(s)\rangle+2\langle S_2(s,\alpha(s))x(s),v(s)\rangle+2\langle q(s,\alpha(s)),x(s)\rangle \\
	&\qquad\qquad+2\langle \rho_1(s,\alpha(s)),u(s)\rangle+2\langle \rho_2(s,\alpha(s)),v(s)\rangle \Big]ds\bigg\},\\
\end{aligned}
\end{equation}
which is a so-called parameterized {\it soft-constrained} cost functional associated with $H_\infty$ optimal control problem (see \cite{Basar95}), with the scalar parameter $\gamma$ standing for ``level of disturbance attenuation".  Note that the functional $J$ of (\ref{cost}) is in fact $J_\gamma$ of (\ref{cost1}) evaluated at $\gamma=0$. For notational simplicity, we denote
\begin{equation*}
\begin{aligned}
	&B(\cdot,\cdot):=\begin{pmatrix}
		B_1(\cdot,\cdot) & B_2(\cdot,\cdot)
	\end{pmatrix},\quad D(\cdot,\cdot):=\begin{pmatrix}
	D_1(\cdot,\cdot) & D_2(\cdot,\cdot)
	\end{pmatrix},\quad \bar{D}(\cdot,\cdot):=\begin{pmatrix}
	\bar{D}_1(\cdot,\cdot) & \bar{D}_2(\cdot,\cdot)
	\end{pmatrix},\\
	&R(\cdot,\cdot):=\begin{pmatrix}
		R_1(\cdot,\cdot) & 0\\
		0 & R_2(\cdot,\cdot)
	\end{pmatrix},\quad R_\gamma(\cdot,\cdot):=\begin{pmatrix}
		R_1(\cdot,\cdot) & 0\\
		0 & R_2(\cdot,\cdot)-\gamma^2I
	\end{pmatrix},\\
	&S(\cdot,\cdot):=\begin{pmatrix}
		S_1(\cdot,\cdot)\\
		S_2(\cdot,\cdot)
	\end{pmatrix},\quad\rho(\cdot,\cdot):=\begin{pmatrix}
		\rho_1(\cdot,\cdot)\\
		\rho_2(\cdot,\cdot)
	\end{pmatrix}.
\end{aligned}
\end{equation*}

With the above notations, some assumptions will be in force throughout this paper.

\textbf{(H1)} The coefficients of the state equation satisfy the following: for each $j\in\mathcal{S}$,
\begin{equation*}
\left\{\begin{aligned}
	&A(\cdot,j)\in L^1(0,T;\mathbb{R}^{n\times n}),\quad B(\cdot,j)\in L^2(0,T;\mathbb{R}^{n\times (m+n_v)}),\quad b(\cdot,j)\in L^1(0,T;\mathbb{R}^{n}),\\
	&C(\cdot,j)\in L^2(0,T;\mathbb{R}^{n\times n}),\quad D(\cdot,j)\in L^\infty(0,T;\mathbb{R}^{n\times (m+n_v)}),\quad \sigma(\cdot,j)\in L^2(0,T;\mathbb{R}^{n}),\\
	&\bar{C}(\cdot,j)\in L^2(0,T;\mathbb{R}^{n\times n}),\quad \bar{D}(\cdot,j)\in L^\infty(0,T;\mathbb{R}^{n\times (m+n_v)}),\quad \bar{\sigma}(\cdot,j)\in L^2(0,T;\mathbb{R}^{n}).\\                     
\end{aligned}\right.
\end{equation*} 

\textbf{(H2)} The weighting coefficients in the cost functional satisfy the following: for each $j\in\mathcal{S}$,
\begin{equation*}
\left\{\begin{aligned}
	&Q(\cdot,j)\in L^1(0,T;\mathbb{S}^n),\quad S(\cdot,j)\in L^2(0,T;\mathbb{R}^{(m+n_v)\times n}),\quad R(\cdot,j)\in L^\infty(0,T;\mathbb{S}^{m+n_v}),\\
	&q(\cdot,j)\in L^1(0,T;\mathbb{R}^n),\quad \rho(\cdot,j)\in L^2(0,T;\mathbb{R}^{m+n_v}),\quad G(T,j)\in\mathbb{S}^n, \quad g(T,j)\in\mathbb{R}^n.
\end{aligned}\right.
\end{equation*}

\textbf{(H3)} The following standard conditions hold:
\begin{equation*}
	G(T,j)\geq0,\quad R(s,j)\gg0,\quad Q(s,j)-S(s,j)^\top R(s,j)^{-1}S(s,j)\geq0,\quad j\in\mathcal{S},\quad a.e. \;s\in[0,T].
\end{equation*}

Under (H1), for any initial pair $(t,\xi,i)\in[0,T)\times L_{\mathcal{F}_t}^2(\Omega;\mathbb{R}^n)\times\mathcal{S}$, $u(\cdot)\in L_{\mathbb{G}}^2(t,T;\mathbb{R}^{m})$ and $v(\cdot)\in L_{\mathbb{F}}^2(t,T;\mathbb{R}^{n_v})$, the state equation (\ref{state eq}) admits a unique strong solution (Mao \cite{Mao07})
\begin{equation*}
	x(\cdot)\equiv x(\cdot;t,\xi,i,u(\cdot),v(\cdot))\in L_{\mathbb{F}}^2(\Omega;C([t,T];\mathbb{R}^n)).
\end{equation*} 
$u(\cdot)\in L_{\mathbb{G}}^2(t,T;\mathbb{R}^{m})$ is called an admissible control and $x(\cdot)$ is called the corresponding admissible state process.
Therefore, under (H1)-(H2), the quadratic performance functionals $J(t,\xi,i;u,v)$ and  $J_\gamma(t,\xi,i;u,v)$ are well-defined for all $(t,\xi,i)\in[0,T)\times L_{\mathcal{F}_t}^2(\Omega;\mathbb{R}^n)\times\mathcal{S}$ and $(u(\cdot),v(\cdot))\in L_{\mathbb{G}}^2(t,T;\mathbb{R}^{m})\times L_{\mathbb{F}}^2(t,T;\mathbb{R}^{n_v})$. If $b(\cdot,\cdot)=\sigma(\cdot,\cdot)=\bar{\sigma}(\cdot,\cdot)=q(\cdot,\cdot)=\rho_1(\cdot,\cdot)=\rho_2(\cdot,\cdot)=g(T,\cdot)=0$, the solution of system (\ref{state eq}) is denoted by $x^0(\cdot)$ and the corresponding cost functionals are denoted by $J^0(t,\xi,i;u(\cdot),v(\cdot))$ and $J_\gamma^0(t,\xi,i;u(\cdot),v(\cdot))$. 

For any $t\in[0,T)$, we define the closed-loop admissible strategy sets of $u(\cdot)$ and $v(\cdot)$:
\begin{equation*}
\begin{aligned}
	\mathcal{U}[t,T]&=\Big\{u:[t,T]\times\Omega\to\mathbb{R}^m\Big|u(s)\in\sigma\left\{\hat{x}(r),W(r),\alpha(r);r\leq s\right\}, \mathbb{E}\int_t^T |u(s)|^2 ds < \infty\Big\},\\
	\mathcal{V}[t,T]&=\Big\{v:[t,T]\times\Omega\to\mathbb{R}^{n_v}\Big|v(s)\in\sigma\left\{x(r),W(r),\overline{W}(r),\alpha(r);r\leq s\right\}, \mathbb{E}\int_t^T |v(s)|^2 ds < \infty\Big\},\\
\end{aligned}
\end{equation*}
where $\hat{x}(r)$ is the optimal filtering estimate of $x(r)$ with respect to $\mathcal{G}_r$ in sense of Xiong \cite{Xiong08}, and the equation it satisfies is given in Section 4.

Due to the existence of the external unknown disturbance $v(\cdot)$ in the system (\ref{state eq}), the decision-maker is required to take into account the $H_\infty$ performance to guarantee robust stability in this case, that is the $H_\infty$ norm is less than a fixed level $\gamma$. This paper will consider the synthesis of a closed-loop control strategy for a finite-horizon stochastic $H_\infty$ 
control problem with Markovian regime switching under partial information. We next give the definition about the finite-horizon stochastic robust $H_\infty$ closed-loop control with Markov chain and partial information.

\begin{mydef}[finite-horizon stochastic $H_\infty$  closed-loop optimal control]
Given a disturbance attenuation level $\gamma>0$, find an $H_\infty$ state feedback control strategy $u^*(\cdot)\equiv u^*(\cdot;\hat{x}(\cdot),\alpha(\cdot))\in\mathcal{U}[t,T]$ such that the homogeneous closed-loop system corresponding to the control $u^*(\cdot)$ satisfies a prescribed $H_\infty$-performance
\begin{equation}\label{Hinfty-p}
	\Vert\mathcal{L}_{u^*}\Vert:=\underset{\substack{v\neq0\\v(\cdot)\in L_{\mathbb{F}}^2(t,T;\mathbb{R}^{n_v}) }}\sup\frac{J^0(0,0,i;u^*(\cdot),v(\cdot))^{\frac{1}{2}}}{\big(\mathbb{E}\int_{0}^{T}\vert v(s)\vert^2ds\big)^{\frac{1}{2}}} < \gamma.
\end{equation}
\end{mydef}

\begin{Remark}
(\romannumeral1) If such a $u^*(\cdot)\in\mathcal{U}[t,T]$ exists, the finite-horizon stochastic $H_\infty$ control problem with Markovian regime switching under partial information is said to admit a solution, and the corresponding $u^*(\cdot)$ is referred to as the robust $H_\infty$ (closed-loop optimal) control.

(\romannumeral2) Under (H3), for any initial pair and $u(\cdot)\in L_{\mathbb{G}}^2(t,T;\mathbb{R}^{m})$ and $v(\cdot)\in L_{\mathbb{F}}^2(t,T;\mathbb{R}^{n_v})$, the cost functional $J^0(t,\xi,i;u(\cdot),v(\cdot))$ is non-negative, which can be viewed as a norm on the  output of the system. The effect of the disturbance on the cost functional of system (\ref{state eq})-(\ref{cost}) is described by the perturbation operator $\mathcal{L}_u$, which (for zero initial time and state) maps finite energy disturbance signals $v(\cdot)$ into the corresponding finite energy output signals $J^0(0,0,i;u(\cdot),v(\cdot))$ of the homogeneous closed-loop system. The magnitude of this linear operator is measured by the induced norm ($H_\infty$ norm). The larger that this norm is, the larger is the effect of the unknown disturbance $v(\cdot)$ on the cost functional $J^0(0,0,i;u(\cdot),v(\cdot))$ in the worst case. 
\end{Remark}

Since the $H_\infty$ control problem can be solved by reformulating it into the corresponding zero-sum game model, we consider the associated soft-constrained zero-sum LQ stochastic differential game with Markov chain and partial information, whose quadratic cost functional is given by $J_\gamma(t,\xi,i;u(\cdot),v(\cdot))$ and denote it as {\bf Problem (SCG)}. In this game, the decision-maker $u(\cdot)\in L_{\mathbb{G}}^2(t,T;\mathbb{R}^m)$, which acts as the minimizing player (henceforth called Player 1), can only access partial information, while the disturbance $v(\cdot)\in L_{\mathbb{F}}^2(t,T;\mathbb{R}^{n_v})$ is  the maximizing player (called Player 2). When the non-homogeneous terms in (\ref{state eq}) and (\ref{cost1}) are zero, we denote it as {\bf Problem (SCG)$^0$}. We now introduce the following definition of the closed-loop saddle point for Problem (SCG).

\begin{mydef}\label{Def-CLSP}
For any given $t\in[0,T)$, let $\hat{\Theta}(\cdot,\cdot), \tilde{\Theta}(\cdot,\cdot):[t,T]\times\mathcal{S}\to\mathbb{R}^{(m+n_v)\times n}$ be the deterministic functions and $\bar{v}(\cdot):[t,T]\times\Omega\to\mathbb{R}^{m+n_v}$ be an $\mathbb{F}$-progressively measurable process with $\hat{\Theta}(\cdot,\alpha(\cdot))\equiv(\hat{\Theta}_1(\cdot,\alpha(\cdot))^\top,\hat{\Theta}_2(\cdot,\alpha(\cdot))^\top)^\top$, $\tilde{\Theta}(\cdot,\alpha(\cdot))\equiv(0,\tilde{\Theta}_2(\cdot,\alpha(\cdot))^\top)^\top$and $\bar{v}(\cdot)\equiv(v_1(\cdot)^\top,v_2(\cdot)^\top)^\top$, satisfying
\begin{equation*}
	\mathbb{E}\int_{t}^{T}\vert\hat{\Theta}(s,\alpha(s))\vert^2ds<\infty,\quad \mathbb{E}\int_{t}^{T}\vert\tilde{\Theta}(s,\alpha(s))\vert^2ds<\infty,
\end{equation*}
with $v_1(\cdot)\in L_{\mathbb{G}}^2(t,T;\mathbb{R}^m)$, and $v_2(\cdot)\in L_{\mathbb{F}}^2(t,T;\mathbb{R}^{n_v})$. The set of all closed-loop strategy pairs on  $[t,T]$ is denoted by $\mathcal{Q}[t,T]$. A 3-tuple $(\hat{\Theta}^*(\cdot,\alpha(\cdot)),\tilde{\Theta}^*(\cdot,\alpha(\cdot)),\bar{v}^*(\cdot))\in\mathcal{Q}[t,T]$ is called a closed-loop saddle point of Problem (SCG) on $[t,T]$ if
\begin{equation}\label{ineq-CL}
\begin{aligned}
	&J_\gamma(t,\xi,i;\hat{\Theta}_1^*(\cdot,\alpha(\cdot))\hat{x}(\cdot)+v_1^*(\cdot),v(\cdot))\\
	&\leq J_\gamma(t,\xi,i;\hat{\Theta}_1^*(\cdot,\alpha(\cdot))\hat{x}^*(\cdot)+v_1^*(\cdot),\hat{\Theta}_2^*(\cdot,\alpha(\cdot))\hat{x}^*(\cdot)+\tilde{\Theta}_2^*(\cdot,\alpha(\cdot))\tilde{x}^*(\cdot)+v_2^*(\cdot))\\
	&\leq J_\gamma(t,\xi,i;u(\cdot),\hat{\Theta}_2^*(\cdot,\alpha(\cdot))\hat{x}(\cdot)+\tilde{\Theta}_2^*(\cdot,\alpha(\cdot))\tilde{x}(\cdot)+v_2^*(\cdot)),\\
	&\qquad\qquad\forall (\xi,i)\in L_{\mathcal{F}_t}^2(\Omega;\mathbb{R}^n)\times\mathcal{S},\;(u(\cdot),v(\cdot))\in L_{\mathbb{G}}^2(t,T;\mathbb{R}^{m})\times L_{\mathbb{F}}^2(t,T;\mathbb{R}^{n_v}),\\
\end{aligned}
\end{equation}
where $\hat{x}^*(\cdot)$ and $\tilde{x}^*(\cdot)$ denote, respectively, the filtering state process and the corresponding difference of the closed-loop system under $(\hat{\Theta}^*(\cdot,\alpha(\cdot)),\tilde{\Theta}^*(\cdot,\alpha(\cdot)),\bar{v}^*(\cdot))$.
\end{mydef}

\begin{Remark}\label{Remark1}
Using the similar method of Propositin 3.3 in Sun and Yong \cite{Sun14}, we can prove that the inequality (\ref{ineq-CL}) in Definition \ref{Def-CLSP} is equivalent to the following inequality:
\begin{equation}\label{ineq-CL'}
\begin{aligned}
	&J_\gamma(t,\xi,i;\hat{\Theta}_1^*(\cdot,\alpha(\cdot))\hat{x}(\cdot)+v_1^*(\cdot),\hat{\Theta}_2^*(\cdot,\alpha(\cdot))\hat{x}(\cdot)+\tilde{\Theta}_2^*(\cdot,\alpha(\cdot))\tilde{x}(\cdot)+v_2(\cdot))\\
	&\leq J_\gamma(t,\xi,i;\hat{\Theta}_1^*(\cdot,\alpha(\cdot))\hat{x}^*(\cdot)+v_1^*(\cdot),\hat{\Theta}_2^*(\cdot,\alpha(\cdot))\hat{x}^*(\cdot)+\tilde{\Theta}_2^*(\cdot,\alpha(\cdot))\tilde{x}^*(\cdot)+v_2^*(\cdot))\\
	&\leq J_\gamma(t,\xi,i;\hat{\Theta}_1^*(\cdot,\alpha(\cdot))\hat{x}(\cdot)+v_1(\cdot),\hat{\Theta}_2^*(\cdot,\alpha(\cdot))\hat{x}(\cdot)+\tilde{\Theta}_2^*(\cdot,\alpha(\cdot))\tilde{x}(\cdot)+v_2^*(\cdot)),\\
	&\qquad\qquad\forall (\xi,i)\in L_{\mathcal{F}_t}^2(\Omega;\mathbb{R}^n)\times\mathcal{S},\;(v_1(\cdot),v_2(\cdot))\in L_{\mathbb{G}}^2(t,T;\mathbb{R}^{m})\times L_{\mathbb{F}}^2(t,T;\mathbb{R}^{n_v}),\\
\end{aligned}
\end{equation}
which will be used in subsequent discussions.
\end{Remark}

If the closed-loop saddle point of Problem (SCG) satisfies the $H_\infty$-performance specified by (\ref{Hinfty-p}), then the corresponding outcome of the closed-loop saddle point is the desired robust $H_\infty$ optimal control and the worst-case disturbance, which we state as {\bf Problem (R-SCG)} below, where the letter ``R" represents $H_\infty$ robustness.

\textbf{Problem (R-SCG).} For any given $t\in[0,T)$ and the disturbance attenuation level $\gamma>0$, find a closed-loop saddle point $(\hat{\Theta}^*(\cdot,\alpha(\cdot)),\tilde{\Theta}^*(\cdot,\alpha(\cdot)),\bar{v}^*(\cdot))\in\mathcal{Q}[t,T]$ of Problem (SCG), and the corresponding outcome of the closed-loop saddle point  $u^*(\cdot):=\hat{\Theta}_1^*(\cdot,\alpha(\cdot))\hat{x}(\cdot)+v_1^*(\cdot)\in\mathcal{U}[t,T]$ ensures that the $H_\infty$-performance holds, i.e., $\Vert \mathcal{L}_{u^*}\Vert<\gamma$. 

If the closed-loop saddle point $(\hat{\Theta}^*(\cdot,\alpha(\cdot)),\tilde{\Theta}^*(\cdot,\alpha(\cdot)),\bar{v}^*(\cdot))\in\mathcal{Q}[t,T]$ mentioned above exists, and the corresponding outcome $(u^*(\cdot),v^*(\cdot))$ is given by
\begin{equation}\label{outcome}
\begin{aligned}
	u^*(\cdot)&:=\hat{\Theta}_1^*(\cdot,\alpha(\cdot))\hat{x}(\cdot)+v_1^*(\cdot),\\
	v^*(\cdot)&:=\hat{\Theta}_2^*(\cdot,\alpha(\cdot))\hat{x}(\cdot)+\tilde{\Theta}_2^*(\cdot,\alpha(\cdot))\tilde{x}(\cdot)+v_2^*(\cdot),\\
\end{aligned}
\end{equation}
then the finite-horizon stochastic $H_\infty$ control of Markovian regime switching system with partial information is said to admit a pair of solution $(u^*(\cdot),v^*(\cdot))$, where $u^*(\cdot)\in\mathcal{U}[t,T]$ is robust $H_\infty$ closed-loop optimal control for the system and $v^*(\cdot)\in\mathcal{V}[t,T]$ is the corresponding worst-case disturbance.

\section{The disturbance attenuation parameter}

Since the $H_\infty$ optimal control depends on the disturbance attenuation level $\gamma$, the smaller this level is, the lower the influence of the external disturbance on the system output that the decision-maker can tolerate. However, the selection of the disturbance attenuation level is not arbitrary. If the pre-specified disturbance attenuation level is too small, the corresponding $H_\infty$ optimal control may not exist. Therefore, in this section, we aim to determine a lower bound for the disturbance attenuation parameter, which also serves as the foundation for all subsequent derivations and discussions.

For any disturbance attenuation level $\gamma>0$ and initial pair $(t,\xi,i)\in[0,T)\times L_{\mathcal{F}_t}^2(\Omega;\mathbb{R}^n)\times\mathcal{S}$, the open-loop upper value $\overline{V}_\gamma^0(t,\xi,i)$ and the open-loop lower value $\underline{V}_\gamma^0(t,\xi,i)$ of Problem (SCG)$^0$ are determined by 
\begin{equation*}
\begin{aligned}
	\overline{V}_\gamma^0(t,\xi,i)&:=\underset{u(\cdot)\in L_{\mathbb{G}}^2(t,T;\mathbb{R}^{m})}\inf\underset{v(\cdot)\in L_{\mathbb{F}}^2(t,T;\mathbb{R}^{n_v})}\sup J_\gamma^0(t,\xi,i;u(\cdot),v(\cdot)),\\
	\underline{V}_\gamma^0(t,\xi,i)&:=\underset{v(\cdot)\in L_{\mathbb{F}}^2(t,T;\mathbb{R}^{n_v})}\sup\underset{u(\cdot)\in L_{\mathbb{G}}^2(t,T;\mathbb{R}^{m})}\inf J_\gamma^0(t,\xi,i;u(\cdot),v(\cdot)),\\
\end{aligned}
\end{equation*}  
which automatically satisfying the following:
\begin{equation*}
	\underline{V}_\gamma^0(t,\xi,i)\leq \overline{V}_\gamma^0(t,\xi,i),\qquad \forall(t,\xi,i)\in [0,T)\times L_{\mathcal{F}_t}^2(\Omega;\mathbb{R}^n)\times\mathcal{S}.
\end{equation*}
In the case that
\begin{equation*}
	\underline{V}_\gamma^0(t,\xi,i)=\overline{V}_\gamma^0(t,\xi,i)=V_\gamma^0(t,\xi,i),
\end{equation*}
we say that Problem (SCG)$^0$ admits an open-loop value $V_\gamma^0(t,\xi,i)$ at $(t,\xi,i)$ with respect to $\gamma>0$. The maps $(t,\xi,i)\mapsto\underline{V}_\gamma^0(t,\xi,i)$, $(t,\xi,i)\mapsto\overline{V}_\gamma^0(t,\xi,i)$, and $(t,\xi,i)\mapsto V_\gamma^0(t,\xi,i)$ are called the open-loop upper value function, open-loop lower value function, and open-loop value function with respect to $\gamma$, respectively.

It should be noted that the open-loop upper value  $\overline{V}_\gamma^0(0,0,i)$ is bounded  below by 0, which can be ensured for Player 2 by choosing $v(\cdot)$ to be zero. If $H_\infty$-performance holds, then we have $\overline{V}_\gamma^0(0,0,i)\leq0$. Hence, the disturbance attenuation level $\gamma$ must satisfy $\overline{V}_\gamma^0(0,0,i)=0$. In view of this, let us introduce a set $\bar{\Gamma}$:
\begin{equation}\label{bar_Gamma}
	\bar{\Gamma}:=\{\gamma>0,\;\forall\gamma'>\gamma,\;\overline{V}_{\gamma'}^0(0,0,i)=0\},
\end{equation}
and a threshold $\bar{\gamma}$:
\begin{equation}\label{bar_gamma}
	\bar{\gamma}:=\inf\bar{\Gamma}.
\end{equation}

According to the definition of the above open-loop upper value, we can draw the following  result, which indicates that the open-loop upper value $\overline{V}_\gamma^0(0,0,i)$ at $(0,0,i)$ is monotonically decreasing with respect to the parameter $\gamma$.

\begin{mypro}\label{prop-uv}
Let (H1)-(H3) hold, if $0<\gamma_1\leq\gamma_2$, then $\overline{V}_{\gamma_2}^0(0,0,i)\leq\overline{V}_{\gamma_1}^0(0,0,i)$.
\end{mypro}

\begin{proof}
Since $0<\gamma_1\leq\gamma_2$, for any $u(\cdot)\in L_{\mathbb{G}}^2(t,T;\mathbb{R}^m)$ and $v(\cdot)\in L_{\mathbb{F}}^2(t,T;\mathbb{R}^{n_v})$, we have 
\begin{equation*}
\begin{aligned}
	J_{\gamma_2}^0(0,0,i;u(\cdot),v(\cdot))&=J^0(0,0,i;u(\cdot),v(\cdot))-\gamma_2^2\mathbb{E}\int_0^T\vert v(s)\vert^2ds\\
	&\leq J^0(0,0,i;u(\cdot),v(\cdot))-\gamma_1^2\mathbb{E}\int_{0}^{T}\vert v(s)\vert^2ds=J_{\gamma_1}^0(0,0,i;u(\cdot),v(\cdot)).
\end{aligned}
\end{equation*}
By taking the supremum on both sides of the inequality, it follows that
\begin{equation*}
	\overline{V}_{\gamma_2}^0(0,0,i)\leq\underset{v(\cdot)\in L_{\mathbb{F}}^2(t,T;\mathbb{R}^{n_v})}\sup J_{\gamma_2}^0(0,0,i;u(\cdot),v(\cdot))\leq\underset{v(\cdot)\in L_{\mathbb{F}}^2(t,T;\mathbb{R}^{n_v})}\sup J_{\gamma_1}^0(0,0,i;u(\cdot),v(\cdot)).
\end{equation*}
Therefore,
\begin{equation*}
	\overline{V}_{\gamma_2}^0(0,0,i)\leq\overline{V}_{\gamma_1}^0(0,0,i).\\
\end{equation*}
We complete the proof.
\end{proof}

\begin{mycor}\label{cor1}
Let (H1)-(H3) hold, if $\gamma>0$ satisfies $\overline{V}_{\gamma}^0(0,0,i)=0$, then $\gamma\geq\bar{\gamma}$.
\end{mycor}

\begin{proof}
By Proposition \ref{prop-uv}, we obtain that for any $\gamma'>\gamma$,
\begin{equation*}
	\overline{V}_{\gamma'}^0(0,0,i)\leq \overline{V}_{\gamma}^0(0,0,i)=0.\\
\end{equation*}
Thus
\begin{equation*}
	\overline{V}_{\gamma'}^0(0,0,i)=0,\\
\end{equation*}
which implies $\gamma\in\bar{\Gamma}$. This completes the proof.
\end{proof}

Moreover, set
\begin{equation}\label{gamma*}
	\gamma^*:=\underset{u(\cdot)\in L_{\mathbb{G}}^2(t,T;\mathbb{R}^m)}\inf\underset{\substack{v\neq0\\v(\cdot)\in L_{\mathbb{F}}^2(t,T;\mathbb{R}^{n_v}) }}\sup\frac{J^0(0,0,i;u(\cdot),v(\cdot))^{\frac{1}{2}}}{\big(\mathbb{E}\int_{0}^{T}\vert v(s)\vert^2ds\big)^{\frac{1}{2}}},
\end{equation}
which is the optimal (minimax) disturbance attenuation level. The outer minimization problem in (\ref{gamma*}) is often referred to as the finite-horizon disturbance attenuation problem (see \cite{Basar95}). If there exists a control $u(\cdot)\in L_{\mathbb{G}}^2(t,T;\mathbb{R}^m)$ satisfying the minimax disturbance attenuation bound $\gamma^*$ in (\ref{gamma*}). Then (\ref{gamma*}) becomes equivalent to:\\
(\romannumeral1)
\begin{equation}\label{eqv_def1}
	J^0(0,0,i;u^*(\cdot),v(\cdot))\leq\gamma^{*^2}\mathbb{E}\int_0^T\vert v(s)\vert^2ds,\qquad\forall v(\cdot)\in L_{\mathbb{F}}^2(t,T;\mathbb{R}^{n_v}).
\end{equation}
(\romannumeral2) There is no other $\hat{u}(\cdot)\in L_{\mathbb{G}}^2(t,T;\mathbb{R}^m)$ and a corresponding $\hat{\gamma}<\gamma^*$, such that
\begin{equation}\label{eqv_def2}
	J^0(0,0,i;\hat{u}(\cdot),v(\cdot))\leq\hat{\gamma}^2\mathbb{E}\int_0^T\vert v(s)\vert^2ds,\qquad\forall v(\cdot)\in L_{\mathbb{F}}^2(t,T;\mathbb{R}^{n_v}).
\end{equation}

\begin{mypro}\label{prop2.1}
Let (H1)-(H3) hold, then $\gamma^*=\bar{\gamma}$.
\end{mypro}

\begin{proof}
It is first proven that $\bar{\gamma}\leq\gamma^*$ holds. From (\ref{gamma*}), we have for any $\epsilon>0$, there exists some $u_0(\cdot)\in L_{\mathbb{G}}^2(t,T;\mathbb{R}^m)$ such that
\begin{equation*}
	\underset{\substack{v\neq0\\v(\cdot)\in L_{\mathbb{F}}^2(t,T;\mathbb{R}^{n_v})}}\sup\frac{J^0(0,0,i;u_0(\cdot),v(\cdot))}{\mathbb{E}\int_{0}^{T}\vert v(s)\vert^2ds}\leq\gamma^{*^2}+\epsilon.\\
\end{equation*}
By the arbitrariness of $\epsilon$, for any $v(\cdot)\in L_{\mathbb{F}}^2(t,T;\mathbb{R}^{n_v})$, $v\neq0$, we have
\begin{equation*}
	J_{\gamma^*}^0(0,0,i;u_0(\cdot),v(\cdot))=J^0(0,0,i;u_0(\cdot),v(\cdot))-\gamma^{*^2}\mathbb{E}\int_{0}^{T}\vert v(s)\vert^2ds\leq0.\\
\end{equation*}
Hence, 
\begin{equation*}
	\overline{V}_{\gamma^*}^0(0,0,i)\leq\underset{v(\cdot)\in L_{\mathbb{F}}^2(t,T;\mathbb{R}^{n_v})}\sup J_{\gamma^*}^0(0,0,i;u_0(\cdot),v(\cdot))\leq0.
\end{equation*}
Since 
\begin{equation*}
	\overline{V}_{\gamma^*}^0(0,0,i)\geq0, 
\end{equation*}
we obtain that for any $\gamma>\gamma^*$,
\begin{equation*}
	\overline{V}_\gamma^0(0,0,i)\leq\underset{v(\cdot)\in L_{\mathbb{F}}^2(t,T;\mathbb{R}^{n_v})}\sup J_{\gamma^*}^0(0,0,i;u_0(\cdot),v(\cdot))=0.
\end{equation*}
Therefore, $\gamma^*\in\bar{\Gamma}$, which implies $\bar{\gamma}\leq\gamma^*$.

On the other hand, by the definition of $\bar{\gamma}$, it follows that for sufficiently large positive integer $k$, there exists a monotonic non-increasing sequence $\{\gamma_k\}_{k\geq0}\subseteq\bar{\Gamma}$ such that $\gamma_k\downarrow\bar{\gamma}$ as $k$ tends to infinity. Furthermore, by the definition of the set $\bar{\Gamma}$, we have
\begin{equation*}
	\overline{V}_{\gamma_k+\frac{1}{k}}^0(0,0,i)=0.\\
\end{equation*}
Therefore, for any $\epsilon>0$, there exists $u_k(\cdot)\in L_{\mathbb{G}}^2(t,T;\mathbb{R}^m)$ such that
\begin{equation*}
	J_{\gamma_k+\frac{1}{k}}^0(0,0,i;u_k(\cdot),v(\cdot))\leq\epsilon,\quad\forall v(\cdot)\in L_{\mathbb{F}}^2(t,T;\mathbb{R}^{n_v}).
\end{equation*}
There exists a common convergent subsequence and a $\bar{u}(\cdot)\in L_{\mathbb{G}}^2(t,T;\mathbb{R}^m)$ satisfying
\begin{equation*}
	J_{\bar{\gamma}}^0(0,0,i;\bar{u}(\cdot),v(\cdot))\leq\epsilon,\quad\forall v(\cdot)\in L_{\mathbb{F}}^2(t,T;\mathbb{R}^{n_v}).
\end{equation*}
Using the arbitrariness of $\epsilon$ and the second item (\ref{eqv_def2}) of the equivalent definition of (\ref{gamma*}), we can obtain $\bar{\gamma}\geq\gamma^*$, which is the desired result.
\end{proof}

\begin{Remark}
Since the disturbance attenuation level $\gamma$ should satisfy $\overline{V}_\gamma^0(0,0,i)=0$, it follows from Corollary \ref{cor1} and Proposition \ref{prop2.1} that $\gamma\geq\gamma^*$. Furthermore, according to Problem (R-SCG), $H_\infty$-performance holds only if $\gamma>\gamma^*$, and we take $\gamma>\gamma^*$ throughout the subsequent discussions.
\end{Remark}

\section{Robust $H_\infty$ control of Problem (R-SCG)}

For the purpose of deriving the robust $H_\infty$ optimal control of the stochastic $H_\infty$ control problem with Markovian regime switching under partial information, we reformulate this problem as an equivalent Problem (R-SCG) to obtain the robust $H_\infty$ optimal control and the worst-case disturbance. To this end, in this section, we first consider the soft-constrained zero-sum LQ stochastic differential game with Markov chain and partial information (Problem (SCG)) and present the sufficient conditions for the existence of the closed-loop saddle point, and then verify that the associated outcome of the closed-loop saddle point fulfills the $H_\infty$ performance criterion.

We define the filtering state and disturbance processes as follows: for $t\in[0,T), s\in[t,T]$,
\begin{equation}\label{not1}
	\hat{x}(s):=\mathbb{E}[x(s)|\mathcal{G}_s],\qquad \hat{v}(s):=\mathbb{E}[v(s)|\mathcal{G}_s],\qquad \hat{\xi}:=\mathbb{E}[\xi|\mathcal{G}_t],
\end{equation}
and the differences
\begin{equation}\label{not2}
	\tilde{x}(s):=x(s)-\hat{x}(s),\qquad \tilde{v}(s):=v(s)-\hat{v}(s),\qquad \tilde{\xi}:=\xi-\hat{\xi}.
\end{equation}
Based on the above notations, we present the following lemma, which will be frequently used in subsequent discussions.

\begin{mylem}\label{lem1}
If $X(\cdot)$, $Y(\cdot)$ are square-integrable, $\mathbb{F}$-adapted stochastic processes, for each $s\in[t,T]$ and $Z(\cdot)=X(\cdot), Y(\cdot)$, $\hat{Z}(s):=\mathbb{E}[Z(s)|\mathcal{G}_s]$ is the orthogonal projection of $Z(s)$ onto $L_{\mathcal{G}_s}^2(\Omega;\mathbb{R}^{n_Z})$, then the difference $\tilde{Z}(s):=Z(s)-\hat{Z}(s)$ is independent of $\mathcal{G}_s$ and
\begin{equation*}
\begin{aligned}
	\mathbb{E}\langle H(s,\alpha(s))Z(s),Z(s)\rangle&=\mathbb{E}\langle H(s,\alpha(s))\hat{Z}(s),\hat{Z}(s)\rangle+\mathbb{E}\langle H(s,\alpha(s))\tilde{Z}(s),\tilde{Z}(s)\rangle,\\
	\mathbb{E}\langle I(s,\alpha(s))\hat{X}(s),\tilde{Y}(s)\rangle&=0,\qquad \mathbb{E}\langle h(s,\alpha(s)),\tilde{Z}(s)\rangle=0,\\
\end{aligned}
\end{equation*}
where $H(\cdot,j)$, $I(\cdot,j)$ and $h(\cdot,j)$, $j\in\mathcal{S}$ are deterministic matrix (vector)-valued functions of proper dimensions, respectively.
\end{mylem}

\begin{proof}
Firstly,
\begin{equation*}
	\mathbb{E}[\tilde{Z}(s)|\mathcal{G}_s]=\mathbb{E}[Z(s)-\hat{Z}(s)|\mathcal{G}_s]=\mathbb{E}[Z(s)|\mathcal{G}_s]-\hat{Z}(s)=0.
\end{equation*}
Then,
\begin{equation*}
	\mathbb{E}\langle h(s,\alpha(s)),\tilde{Z}(s)\rangle=\mathbb{E}[ h(s,\alpha(s))^\top\tilde{Z}(s)]=\mathbb{E}\left[  h(s,\alpha(s))^\top\mathbb{E}[ \tilde{Z}(s)|\mathcal{G}_s]\right] =0,
\end{equation*}
which implies the third equation. And
\begin{equation*}
\begin{aligned}
	&\mathbb{E}\langle I(s,\alpha(s))\hat{X}(s),\tilde{Y}(s)\rangle\\
	&=\mathbb{E}\left[\mathbb{E}[\hat{X}(s)^\top I(s,\alpha(s))^\top\tilde{Y}(s)|\mathcal{G}_s]\right]=\mathbb{E}\left[\hat{X}(s)^\top I(s,\alpha(s))^\top\mathbb{E}[\tilde{Y}(s)|\mathcal{G}_s]\right]=0.
\end{aligned}
\end{equation*}
Similarly, we can get
\begin{equation*}
    \mathbb{E}\langle H(s,\alpha(s))\hat{Z}(s),\tilde{Z}(s)\rangle=0.
\end{equation*}
Using $Z(\cdot)=\hat{Z}(\cdot)+\tilde{Z}(\cdot)$, the first equation of the lemma holds.
\end{proof}

By notations (\ref{not1}), (\ref{not2}) and Lemma \ref{lem1}, the filtering process $\hat{x}(\cdot)$ evolves according to the following SDE:
\begin{equation}\label{hat_x}
\left\{\begin{aligned}
	d\hat{x}(s)=&\big[A(s,\alpha(s))\hat{x}(s)+B_1(s,\alpha(s))u(s)+B_2(s,\alpha(s))\hat{v}(s)+b(s,\alpha(s))\big]ds\\
	&+\big[C(s,\alpha(s))\hat{x}(s)+D_1(s,\alpha(s))u(s)+D_2(s,\alpha(s))\hat{v}(s)+\sigma(s,\alpha(s))\big]dW(s),\\
	\hat{x}(t)=&\ \hat{\xi},\quad\alpha(t)=i,
\end{aligned}\right.
\end{equation}
and the difference $\tilde{x}(\cdot)$ satisfies the following SDE:
\begin{equation}\label{tilde_x}
\left\{\begin{aligned}
	d\tilde{x}(s)=&\big[A(s,\alpha(s))\tilde{x}(s)+B_2(s,\alpha(s))\tilde{v}(s)\big]ds+\big[C(s,\alpha(s))\tilde{x}(s)+D_2(s,\alpha(s))\tilde{v}(s)\big]dW(s)\\
	&+\big[\bar{C}(s,\alpha(s))x(s)+\bar{D}_1(s,\alpha(s))u(s)+\bar{D}_2(s,\alpha(s))v(s)+\bar{\sigma}(s,\alpha(s))\big]d\overline{W}(s),\\
	\tilde{x}(t)=&\ \tilde{\xi},\quad\alpha(t)=i.
\end{aligned}\right.
\end{equation}

Moreover, the soft-constrained cost functional $J_\gamma(t,\xi,i;u(\cdot),v(\cdot))$ can be rewritten as
\begin{equation}
\begin{aligned}
	&J_\gamma(t,\xi,i;u(\cdot),v(\cdot))=\mathbb{E}\bigg\{\langle G(T,\alpha(T))\tilde{x}(T),\tilde{x}(T)\rangle+\int_t^T\Big[\langle Q(s,\alpha(s))\tilde{x}(s),\tilde{x}(s)\rangle\\
	&\qquad+\big\langle (R_2(s,\alpha(s))-\gamma^2I)\tilde{v}(s),\tilde{v}(s)\big\rangle+2\langle S_2(s,\alpha(s))\tilde{x}(s),\tilde{v}(s)\rangle\Big]ds\bigg\}\\
	&\qquad+\mathbb{E}\bigg\{\langle G(T,\alpha(T))\hat{x}(T),\hat{x}(T)\rangle+2\langle g(T,\alpha(T)),\hat{x}(T)\rangle+\int_t^T\Big[  \langle Q(s,\alpha(s))\hat{x}(s),\hat{x}(s)\rangle\\
	&\qquad\qquad+\langle R_1(s,\alpha(s))u(s),u(s)\rangle+\big\langle (R_2(s,\alpha(s))-\gamma^2I)\hat{v}(s),\hat{v}(s)\big\rangle\\
	&\qquad\qquad+2\langle S_1(s,\alpha(s))\hat{x}(s),u(s)\rangle+2\langle S_2(s,\alpha(s))\hat{x}(s),\hat{v}(s)\rangle+2\langle q(s,\alpha(s)),\hat{x}(s)\rangle \\
	&\qquad\qquad+2\langle \rho_1(s,\alpha(s)),u(s)\rangle+2\langle \rho_2(s,\alpha(s)),\hat{v}(s)\rangle \Big]ds\bigg\}.
\end{aligned}
\end{equation}

To obtain the closed-loop saddle point of Problem (SCG), inspired by the work of Yu \cite{Yu15} on a zero-sum LQ stochastic differential game with complete information, we employ the technique of completion-of-squares to obtain the optimal feedback control-strategy pair in a closed-loop form based on the solution of Riccati equations, and then prove that the optimal feedback control-strategy pair is an outcome of the closed-loop saddle point of Problem (SCG). To this end, we first introduce the notion of Elliott–Kalton strategies (or \emph{nonanticipative} strategies) for two-person zero-sum games (see \cite{Elliot72,Yu15,Moon25}).

\begin{mydef}\label{Def1}
A nonanticipative strategy for Player 1 is a mapping $\alpha_1: L_{\mathbb{F}}^2(t,T;\mathbb{R}^{n_v})\to L_{\mathbb{G}}^2(t,T;\mathbb{R}^{m})$ such that for any $\mathbb{F}$-stopping time $\tau:\Omega\to[t,T]$ and any $v'(\cdot),v''(\cdot)\in L_{\mathbb{F}}^2(t,T;\mathbb{R}^{n_v})$, with $v'\equiv v''$ on $[t,\tau]$, it holds that $\alpha_1(v')\equiv\alpha_1(v'')$ on $[t,\tau]$. The nonanticipative strategy $\alpha_2: L_{\mathbb{G}}^2(t,T;\mathbb{R}^{m})\to L_{\mathbb{F}}^2(t,T;\mathbb{R}^{n_v})$ for Player 2 are defined in a similar way. The collection of all admissible nonanticipative strategies for Player 1 (resp., Player 2) over $[t,T]$ is denoted by $\mathcal{A}_1[t,T]$ (resp., $\mathcal{A}_2[t,T]$).
\end{mydef}

In light of Definition \ref{Def1}, we formulate the following two auxiliary problems.

\textbf{Problem (SCG-1).} For any $t\in[0,T)$ and the disturbance attenuation level $\gamma>0$, find an admissible control $u^*(\cdot)\in L_{\mathbb{G}}^2(t,T;\mathbb{R}^{m})$ and an admissible nonanticipative strategy $\alpha_2^*(\cdot)\in\mathcal{A}_2[t,T]$ such that
\begin{equation}\label{CS-1}
	J_\gamma(t,\xi,i;u^*(\cdot),\alpha_2^*(u^*)(\cdot))=\underset{\alpha_2(\cdot)\in\mathcal{A}_2[t,T]}\sup\underset{u(\cdot)\in L_{\mathbb{G}}^2(t,T;\mathbb{R}^{m})}\inf J_\gamma(t,\xi,i;u(\cdot),\alpha_2(u)(\cdot)).
\end{equation}
If there exist an admissible control $u^*(\cdot)$ and admissible nonanticipative strategy $\alpha_2^*(\cdot)$ satisfying (\ref{CS-1}), then we call $(u^*(\cdot),\alpha_2^*(\cdot))$ an optimal control-strategy pair of Player 1’s value $J_\gamma(t,\xi,i;u^*(\cdot),\alpha_2^*(u^*)(\cdot))$.

\textbf{Problem (SCG-2).} For any $t\in[0,T)$ and the disturbance attenuation level $\gamma>0$, find an admissible control $v^*(\cdot)\in L_{\mathbb{F}}^2(t,T;\mathbb{R}^{n_v})$ and an admissible nonanticipative strategy $\alpha_1^*(\cdot)\in\mathcal{A}_1[t,T]$ such that
\begin{equation}\label{CS-2}
	J_\gamma(t,\xi,i;\alpha_1^*(v^*)(\cdot),v^*(\cdot))=\underset{\alpha_1(\cdot)\in\mathcal{A}_1[t,T]}\inf\underset{v(\cdot)\in L_{\mathbb{F}}^2(t,T;\mathbb{R}^{n_v})}\sup J_\gamma(t,\xi,i;\alpha_1(v)(\cdot),v(\cdot)).
\end{equation}
If there exist an admissible control $v^*(\cdot)$ and admissible nonanticipative strategy $\alpha_1^*(\cdot)$ satisfying (\ref{CS-2}), then we call $(\alpha_1^*(\cdot),v^*(\cdot))$ an optimal control-strategy pair of Player 2’s value $J_\gamma(t,\xi,i;\alpha_1^*(v^*)(\cdot),v^*(\cdot))$.

In the following, we first introduce the following two sets of coupled backward generalized Riccati equations:

\begin{equation}\label{Pi}
\hspace{-2mm}\left\{\begin{aligned}
	&\dot{\Pi}(s,i)+\Pi(s,i)A(s,i)+A(s,i)^\top\Pi(s,i)+C(s,i)^\top\Pi(s,i)C(s,i)+\bar{C}(s,i)^\top\Pi(s,i)\bar{C}(s,i)\\
	&\quad -\bar{\mathcal{S}}_2(s,i)^\top\bar{\mathcal{R}}_2(s,i)^{-1}\bar{\mathcal{S}}_2(s,i)+Q(s,i)+\sum_{j=1}^{D}\lambda_{ij}\Pi(s,j)=0,\\
	&\Pi(T,i)=G(T,i),\quad\bar{\mathcal{R}}_2(s,i)\ll0,\quad a.e.\;s\in[t,T],\quad i\in\mathcal{S},
\end{aligned}\right.
\end{equation}
\begin{equation}\label{P}
\hspace{-2mm}\left\{\begin{aligned}
	&\dot{P}(s,i)+P(s,i)A(s,i)+A(s,i)^\top P(s,i)+C(s,i)^\top P(s,i)C(s,i)+\bar{C}(s,i)^\top\Pi(s,i)\bar{C}(s,i)\\
	&\quad -\hat{\mathcal{S}}(s,i)^\top\hat{\mathcal{R}}(s,i)^{-1}\hat{\mathcal{S}}(s,i)+Q(s,i)+\sum_{j=1}^{D}\lambda_{ij}P(s,j)=0,\qquad a.e.\;s\in[t,T],\\
	&P(T,i)=G(T,i),\quad i\in\mathcal{S},
\end{aligned}\right.
\end{equation}
where
\begin{equation*}
\begin{aligned}
	\bar{\mathcal{S}}_2(s,i)&:=B_2(s,i)^\top\Pi(s,i)+D_2(s,i)^\top\Pi(s,i)C(s,i)+\bar{D}_2(s,i)^\top\Pi(s,i)\bar{C}(s,i)+S_2(s,i),\\
	\bar{\mathcal{R}}_2(s,i)&:=R_2(s,i)-\gamma^2I+D_2(s,i)^\top\Pi(s,i)D_2(s,i)+\bar{D}_2(s,i)^\top\Pi(s,i)\bar{D}_2(s,i),\\
	\hat{\mathcal{S}}(s,i)&:=\begin{pmatrix}
		\hat{\mathcal{S}}_1(s,i) \\
		\hat{\mathcal{S}}_2(s,i)
	\end{pmatrix}:=B(s,i)^\top P(s,i)+D(s,i)^\top P(s,i)C(s,i)\\
    &\qquad\qquad\qquad\qquad +\bar{D}(s,i)^\top \Pi(s,i)\bar{C}(s,i)+S(s,i),\\
\end{aligned}
\end{equation*}
\begin{equation*}
\begin{aligned}	
        \hat{\mathcal{R}}(s,i)&:=\begin{pmatrix}
		\hat{\mathcal{R}}_{11}(s,i) & \hat{\mathcal{R}}_{12}(s,i) \\
		\hat{\mathcal{R}}_{12}(s,i)^\top & \hat{\mathcal{R}}_{22}(s,i)\\
	\end{pmatrix}:=R_\gamma(s,i)+D(s,i)^\top P(s,i)D(s,i)+\bar{D}(s,i)^\top \Pi(s,i)\bar{D}(s,i).\\
\end{aligned}
\end{equation*}

We next present the following lemma, which essentially shows two equivalent forms of the Riccati equation (\ref{P}), and it plays an important role in using the completion-of-squares method to find the optimal feedback control-strategy pairs for two players in the subsequent analysis.

\begin{mylem}\label{lem2}
For any $(s,i)\in[t,T]\times\mathcal{S}$ and $(\Pi,P)\in\mathbb{S}^n\times\mathbb{S}^n$,\\
(\romannumeral1) If the matrices $\hat{\mathcal{R}}_{22}(s,i)$ and $\hat{\mathcal{R}}_{11}(s,i)-\hat{\mathcal{R}}_{12}(s,i)\hat{\mathcal{R}}_{22}(s,i)^{-1}\hat{\mathcal{R}}_{12}(s,i)^\top$ are invertible, then $\hat{\mathcal{R}}(s,i)$ is also invertible, and 
\begin{equation}\label{form1}
\begin{aligned}
	&\hat{\mathcal{S}}(s,i)^\top\hat{\mathcal{R}}(s,i)^{-1}\hat{\mathcal{S}}(s,i)=\hat{\mathcal{S}}_2(s,i)^\top\hat{\mathcal{R}}_{22}(s,i)^{-1}\hat{\mathcal{S}}_2(s,i)\\
	&\quad +\big(\hat{\mathcal{S}}_1(s,i)-\hat{\mathcal{R}}_{12}(s,i)\hat{\mathcal{R}}_{22}(s,i)^{-1}\hat{\mathcal{S}}_2(s,i)\big)^\top\big(\hat{\mathcal{R}}_{11}(s,i)\\
	&\quad -\hat{\mathcal{R}}_{12}(s,i)\hat{\mathcal{R}}_{22}(s,i)^{-1}\hat{\mathcal{R}}_{12}(s,i)^\top\big)^{-1}\big(\hat{\mathcal{S}}_1(s,i)-\hat{\mathcal{R}}_{12}(s,i)\hat{\mathcal{R}}_{22}(s,i)^{-1}\hat{\mathcal{S}}_2(s,i)\big).
\end{aligned}
\end{equation}\label{form2}
(\romannumeral2) If the matrices $\hat{\mathcal{R}}_{11}(s,i)$ and $\hat{\mathcal{R}}_{22}(s,i)-\hat{\mathcal{R}}_{12}(s,i)^\top\hat{\mathcal{R}}_{11}(s,i)^{-1}\hat{\mathcal{R}}_{12}(s,i)$ are invertible, then $\hat{\mathcal{R}}(s,i)$ is also invertible, and 
\begin{equation}
\begin{aligned}
	&\hat{\mathcal{S}}(s,i)^\top\hat{\mathcal{R}}(s,i)^{-1}\hat{\mathcal{S}}(s,i)=\hat{\mathcal{S}}_1(s,i)^\top\hat{\mathcal{R}}_{11}(s,i)^{-1}\hat{\mathcal{S}}_1(s,i)\\
	&\quad +\big(\hat{\mathcal{S}}_2(s,i)-\hat{\mathcal{R}}_{12}(s,i)^\top\hat{\mathcal{R}}_{11}(s,i)^{-1}\hat{\mathcal{S}}_1(s,i)\big)^\top\big(\hat{\mathcal{R}}_{22}(s,i)\\
	&\quad -\hat{\mathcal{R}}_{12}(s,i)^\top\hat{\mathcal{R}}_{11}(s,i)^{-1}\hat{\mathcal{R}}_{12}(s,i)\big)^{-1}\big(\hat{\mathcal{S}}_2(s,i)-\hat{\mathcal{R}}_{12}(s,i)^\top\hat{\mathcal{R}}_{11}(s,i)^{-1}\hat{\mathcal{S}}_1(s,i)\big).
\end{aligned}
\end{equation} 
\end{mylem}

\begin{proof}
We only prove (\romannumeral1), the proof of (\romannumeral2) follows the same technique. Firstly, the invertibility of $\hat{\mathcal{R}}(s,i)$ comes from the following relation:
\begin{equation*}
\begin{aligned}
	\begin{pmatrix}
		I & -\hat{\mathcal{R}}_{12}\hat{\mathcal{R}}_{22}^{-1} \\
		0 & I\\
	\end{pmatrix}\hat{\mathcal{R}}
	\begin{pmatrix}
		I & 0 \\
		-\hat{\mathcal{R}}_{22}^{-1}\hat{\mathcal{R}}_{12}^\top& I\\
	\end{pmatrix}=
	\begin{pmatrix}
		\hat{\mathcal{R}}_{11}-\hat{\mathcal{R}}_{12}\hat{\mathcal{R}}_{22}^{-1}\hat{\mathcal{R}}_{12}^\top & 0\\
		0 & \hat{\mathcal{R}}_{22}\\
	\end{pmatrix}.
\end{aligned}
\end{equation*}
Hence, the inverse of $\hat{\mathcal{R}}(s,i)$ can be written by
\begin{equation*}
	\hat{\mathcal{R}}(s,i)^{-1}=\begin{pmatrix}
		I & 0 \\
		-\hat{\mathcal{R}}_{22}^{-1}\hat{\mathcal{R}}_{12}^\top& I\\
	\end{pmatrix}\begin{pmatrix}
	(\hat{\mathcal{R}}_{11}-\hat{\mathcal{R}}_{12}\hat{\mathcal{R}}_{22}^{-1}\hat{\mathcal{R}}_{12}^\top)^{-1} & 0\\
	0 & \hat{\mathcal{R}}_{22}^{-1}\\
	\end{pmatrix}\begin{pmatrix}
	I & -\hat{\mathcal{R}}_{12}\hat{\mathcal{R}}_{22}^{-1} \\
	0 & I\\
	\end{pmatrix}.
\end{equation*} 
Then, substituting the above equation into $\hat{\mathcal{S}}^\top\hat{\mathcal{R}}^{-1}\hat{\mathcal{S}}$ and simplifying it, we can obtain 
\begin{equation*}
\begin{aligned}
	&\hat{\mathcal{S}}(s,i)^\top\hat{\mathcal{R}}(s,i)^{-1}\hat{\mathcal{S}}(s,i)\\
	&=\begin{pmatrix}
		\hat{\mathcal{S}}_1^\top & \hat{\mathcal{S}}_2^\top
	\end{pmatrix}\begin{pmatrix}
	I & 0 \\
	-\hat{\mathcal{R}}_{22}^{-1}\hat{\mathcal{R}}_{12}^\top& I\\
	\end{pmatrix}\begin{pmatrix}
	(\hat{\mathcal{R}}_{11}-\hat{\mathcal{R}}_{12}\hat{\mathcal{R}}_{22}^{-1}\hat{\mathcal{R}}_{12}^\top)^{-1} & 0\\
	0 & \hat{\mathcal{R}}_{22}^{-1}\\
	\end{pmatrix}\begin{pmatrix}
	I & -\hat{\mathcal{R}}_{12}\hat{\mathcal{R}}_{22}^{-1} \\
	0 & I\\
	\end{pmatrix}\begin{pmatrix}
	\hat{\mathcal{S}}_1 \\
	\hat{\mathcal{S}}_2
	\end{pmatrix}\\
	&=\hat{\mathcal{S}}_2^\top\hat{\mathcal{R}}_{22}^{-1}\hat{\mathcal{S}}_2+\big(	\hat{\mathcal{S}}_1-\hat{\mathcal{R}}_{12}\hat{\mathcal{R}}_{22}^{-1}\hat{\mathcal{S}}_2\big)^\top\big(\hat{\mathcal{R}}_{11}-\hat{\mathcal{R}}_{12}\hat{\mathcal{R}}_{22}^{-1}\hat{\mathcal{R}}_{12}^\top\big)^{-1}\big(	\hat{\mathcal{S}}_1-\hat{\mathcal{R}}_{12}\hat{\mathcal{R}}_{22}^{-1}\hat{\mathcal{S}}_2\big).\\
\end{aligned}
\end{equation*}
This completes the proof.
\end{proof}

Based on the above lemma, we introduce the following conditions for further derivation.
\textbf{Condition (\uppercase\expandafter{\romannumeral1}).}
\begin{equation*}
	\hat{\mathcal{R}}_{22}(s,i)\ll0,\quad \hat{\mathcal{R}}_{11}(s,i)-\hat{\mathcal{R}}_{12}(s,i)\hat{\mathcal{R}}_{22}(s,i)^{-1}\hat{\mathcal{R}}_{12}(s,i)^\top\gg0, \quad a.e.\; s\in[t,T],\quad i\in\mathcal{S}.
\end{equation*}
\textbf{Condition (\uppercase\expandafter{\romannumeral2}).}
\begin{equation*}
	\hat{\mathcal{R}}_{11}(s,i)\gg0,\quad \hat{\mathcal{R}}_{22}(s,i)-\hat{\mathcal{R}}_{12}(s,i)^\top\hat{\mathcal{R}}_{11}(s,i)^{-1}\hat{\mathcal{R}}_{12}(s,i)\ll0, \quad a.e.\; s\in[t,T],\quad i\in\mathcal{S}.
\end{equation*}
\textbf{Condition (\uppercase\expandafter{\romannumeral1}\&\uppercase\expandafter{\romannumeral2}).}
\begin{equation*}
	\hat{\mathcal{R}}_{11}(s,i)\gg0,\quad \hat{\mathcal{R}}_{22}(s,i)\ll0, \quad a.e.\; s\in[t,T],\quad i\in\mathcal{S}.
\end{equation*}

\begin{Remark}
It is obvious that both Condition (\uppercase\expandafter{\romannumeral1}) and Condition (\uppercase\expandafter{\romannumeral2}) hold true if and only if Condition (\uppercase\expandafter{\romannumeral1}$\&$\uppercase\expandafter{\romannumeral2}) holds.
\end{Remark}

\begin{mylem}\label{lem3}
Let (H1)-(H3) hold and $(\Pi(\cdot,\cdot),P(\cdot,\cdot))\in C([t,T]\times\mathcal{S};\mathbb{S}^n)\times C([t,T]\times\mathcal{S};\mathbb{S}^n)$ be the solution of the Riccati equations (\ref{Pi})-(\ref{P}) satisfying Condition (\uppercase\expandafter{\romannumeral1}). Then the admissible feedback control-strategy pair $(u^*(\cdot),\alpha_2^*(\cdot))$ of Player 1 is given by
\begin{equation}\label{OCS-1}
\left\{\begin{aligned}
	u^*(s)=&-\Big[\hat{\mathcal{R}}_{11}(s,\alpha(s))-\hat{\mathcal{R}}_{12}(s,\alpha(s))\hat{\mathcal{R}}_{22}(s,\alpha(s))^{-1}\hat{\mathcal{R}}_{12}(s,\alpha(s))^\top\Big]^{-1}\Bigl\{\Big[\hat{\mathcal{S}}_1(s,\alpha(s))\\
	&\qquad-\hat{\mathcal{R}}_{12}(s,\alpha(s))\hat{\mathcal{R}}_{22}(s,\alpha(s))^{-1}\hat{\mathcal{S}}_2(s,\alpha(s))\Big]\hat{x}(s)+B_1(s,\alpha(s))^\top\eta(s,\alpha(s))\\
	&\qquad+D_1(s,\alpha(s))^\top P(s,\alpha(s))\sigma(s,\alpha(s))+\bar{D}_1(s,\alpha(s))^\top \Pi(s,\alpha(s))\bar{\sigma}(s,\alpha(s))\\
	&\qquad+\rho_1(s,\alpha(s))-\hat{\mathcal{R}}_{12}(s,\alpha(s))\hat{\mathcal{R}}_{22}(s,\alpha(s))^{-1}\Big[B_2(s,\alpha(s))^\top\eta(s)+D_2(s,\alpha(s))^\top\\
	&\qquad\times P(s,\alpha(s))\sigma(s,\alpha(s))+\bar{D}_2(s,\alpha(s))^\top \Pi(s,\alpha(s))\bar{\sigma}(s,\alpha(s))+\rho_2(s,\alpha(s))\Big]\Bigr\},\\
	\hat{\alpha}_2^*(u)(s)=&-\hat{\mathcal{R}}_{22}(s,\alpha(s))^{-1}\Big[\hat{\mathcal{R}}_{12}(s,\alpha(s))^\top u(s)+\hat{\mathcal{S}}_2(s,\alpha(s))\hat{x}(s)+B_2(s,\alpha(s))^\top\eta(s,\alpha(s))\\
	&\qquad+D_2(s,\alpha(s))^\top P(s,\alpha(s))\sigma(s,\alpha(s))+\bar{D}_2(s,\alpha(s))^\top \Pi(s,\alpha(s))\bar{\sigma}(s,\alpha(s))\\
	&\qquad+\rho_2(s,\alpha(s))\Big],\\
	\tilde{v}^*(s)=&-\bar{\mathcal{R}}_2(s,\alpha(s))^{-1}\bar{\mathcal{S}}_2(s,\alpha(s))\tilde{x}(s),\quad
	\alpha_2^*(u)(s):=\hat{\alpha}_2^*(u)(s)+\tilde{v}^*(s),\quad s\in[t,T],\\
\end{aligned}\right.
\end{equation}
where $\eta(\cdot,\cdot)\in C([t,T]\times\mathcal{S};\mathbb{R}^n)$ is the solution of the following BODE:
\begin{equation}\label{eta1}
\left\{\begin{aligned}
	&\dot{\eta}(s,i)+A(s,i)^\top\eta(s,i)-\hat{\mathcal{S}}_2(s,i)^\top\hat{\mathcal{R}}_{22}(s,i)^{-1}\psi(s,i)-\big(\hat{\mathcal{S}}_1(s,i)-\hat{\mathcal{R}}_{12}(s,i)\hat{\mathcal{R}}_{22}(s,i)^{-1}\\
	&\quad\times\hat{\mathcal{S}}_2(s,i)\big)^\top\big(\hat{\mathcal{R}}_{11}(s,i)-\hat{\mathcal{R}}_{12}(s,i)\hat{\mathcal{R}}_{22}(s,i)^{-1}\hat{\mathcal{R}}_{12}(s,i)^\top\big)^{-1}\varphi(s,i)+C(s,i)^\top P(s,i)\sigma(s,i)\\
	&\quad+\bar{C}(s,i)^\top \Pi(s,i)\bar{\sigma}(s,i)+P(s,i)b(s,i)+q(s,i)+\sum_{j=1}^{D}\lambda_{ij}\eta(s,j)=0,\quad a.e.\; s\in[t,T],\\
	&\eta(T,i)=g(T,i),\quad i\in\mathcal{S},
\end{aligned}\right.
\end{equation}
and 
\begin{equation}
\begin{aligned}
	\psi(s,i)&:=B_2(s,i)^\top\eta(s,i)+D_2(s,i)^\top P(s,i)\sigma(s,i)+\bar{D}_2(s,i)^\top \Pi(s,i)\bar{\sigma}(s,i)+\rho_2(s,i),\\
	\varphi(s,i)&:=B_1(s,i)^\top\eta(s,i)+D_1(s,i)^\top P(s,i)\sigma(s,i)+\bar{D}_1(s,i)^\top \Pi(s,i)\bar{\sigma}(s,i)+\rho_1(s,i)\\
	&\quad -\hat{\mathcal{R}}_{12}(s,i)\hat{\mathcal{R}}_{22}(s,i)^{-1}\psi(s,i).
\end{aligned}
\end{equation}
Then we have the following:\\
(\romannumeral1) $J_\gamma(t,\xi,i;u(\cdot),\alpha_2(u)(\cdot))\leq J_\gamma(t,\xi,i;u(\cdot),\alpha_2^*(u)(\cdot))$ for any $u(\cdot)\in L_{\mathbb{G}}^2(t,T;\mathbb{R}^{m})$ and $\alpha_2(\cdot)\in\mathcal{A}_2[t,T]$. The equation
holds if and only if $\alpha_2(u)(\cdot)=\alpha_2^*(u)(\cdot)$.\\
(\romannumeral2) $J_\gamma(t,\xi,i;u^*(\cdot),\alpha_2^*(u^*)(\cdot))\leq J_\gamma(t,\xi,i;u(\cdot),\alpha_2^*(u)(\cdot))$ for any $u(\cdot)\in L_{\mathbb{G}}^2(t,T;\mathbb{R}^{m})$. The equation holds if and only if $u(\cdot)=u^*(\cdot)$.
\end{mylem}

\begin{proof}
Applying It\^{o}'s formula (\cite{Mao07}) to $s\mapsto\langle\Pi(s,\alpha(s))\tilde{x}(s),\tilde{x}(s)\rangle$ and $s\mapsto\langle P(s,\alpha(s))\hat{x}(s)+2\eta(s,\alpha(s)),\hat{x}(s)\rangle$, respectively, we have
\begin{equation*}
\begin{aligned}
	&d\langle\Pi(s,\alpha(s))\tilde{x}(s),\tilde{x}(s)\rangle\\
	&=\Big[\langle\dot{\Pi}(s,\alpha(s))\tilde{x}(s),\tilde{x}(s)\rangle+\sum_{j=1}^{D}\lambda_{\alpha(s-)j}\big\langle(\Pi(s,j)-\Pi(s,\alpha(s-)))\tilde{x}(s),\tilde{x}(s)\big\rangle\\
	&\qquad+\langle\Pi(s,\alpha(s))\big(A(s,\alpha(s))\tilde{x}(s)+B_2(s,\alpha(s))\tilde{v}(s)\big),\tilde{x}(s)\rangle\\
	&\qquad+\langle\Pi(s,\alpha(s))\tilde{x}(s),A(s,\alpha(s))\tilde{x}(s)+B_2(s,\alpha(s))\tilde{v}(s)\rangle\\
	&\qquad+\langle\Pi(s,\alpha(s))\big(C(s,\alpha(s))\tilde{x}(s)+D_2(s,\alpha(s))\tilde{v}(s)\big),C(s,\alpha(s))\tilde{x}(s)+D_2(s,\alpha(s))\tilde{v}(s)\rangle\\
	&\qquad+\langle\Pi(s,\alpha(s))\big(\bar{C}(s,\alpha(s))x(s)+\bar{D}_1(s,\alpha(s))u(s)+\bar{D}_2(s,\alpha(s))v(s)+\bar{\sigma}(s,\alpha(s))\big),\\
	&\qquad\qquad\bar{C}(s,\alpha(s))x(s)+\bar{D}_1(s,\alpha(s))u(s)+\bar{D}_2(s,\alpha(s))v(s)+\bar{\sigma}(s,\alpha(s))\rangle\Big]ds\\
	&\quad+\sum_{i,j=1}^{D}\big\langle(\Pi(s,j)-\Pi(s,i))\tilde{x}(s),\tilde{x}(s)\big\rangle d\tilde{N}_{ij}(s)+\big[\cdots\big]dW(s)+\big[\cdots\big]d\overline{W}(s),\\
	&d\langle P(s,\alpha(s))\hat{x}(s)+2\eta(s,\alpha(s)),\hat{x}(s)\rangle\\
	&=\Big[\langle\dot{P}(s,\alpha(s))\hat{x}(s),\hat{x}(s)\rangle+\sum_{j=1}^{D}\lambda_{\alpha(s-)j}\big\langle(P(s,j)-P(s,\alpha(s-)))\hat{x}(s),\hat{x}(s)\big\rangle\\
	&\qquad+\langle P(s,\alpha(s))\big(A(s,\alpha(s))\hat{x}(s)+B_1(s,\alpha(s))u(s)+B_2(s,\alpha(s))\hat{v}(s)+b(s,\alpha(s))\big),\hat{x}(s)\rangle\\
	&\qquad+\langle P(s,\alpha(s))\hat{x}(s),A(s,\alpha(s))\hat{x}(s)+B_1(s,\alpha(s))u(s)+B_2(s,\alpha(s))\hat{v}(s)+b(s,\alpha(s))\rangle\\
	&\qquad+\langle P(s,\alpha(s))\big(C(s,\alpha(s))\hat{x}(s)+D_1(s,\alpha(s))u(s)+D_2(s,\alpha(s))\hat{v}(s)+\sigma(s,\alpha(s))\big),\\
	&\qquad\qquad C(s,\alpha(s))\hat{x}(s)+D_1(s,\alpha(s))u(s)+D_2(s,\alpha(s))\hat{v}(s)+\sigma(s,\alpha(s))\rangle\\
	&\qquad+2\langle \dot{\eta}(s,\alpha(s)),\hat{x}(s)\rangle+2\sum_{j=1}^{D}\lambda_{\alpha(s-)j}\big\langle\eta(s,j)-\eta(s,\alpha(s-)),\hat{x}(s)\big\rangle\\
    &\qquad+2\langle \eta(s,\alpha(s)),A(s,\alpha(s))\hat{x}(s)+B_1(s,\alpha(s))u(s)+B_2(s,\alpha(s))\hat{v}(s)+b(s,\alpha(s))\rangle \Big]ds\\
\end{aligned}
\end{equation*}
\begin{equation*}
\begin{aligned}	
	&\quad\;+\sum_{i,j=1}^{D}\big\langle(P(s,j)-P(s,i))\hat{x}(s)+2(\eta(s,j)-\eta(s,i)),\hat{x}(s)\big\rangle d\tilde{N}_{ij}(s)+\big[\cdots\big]dW(s).
\end{aligned}
\end{equation*}
Integrating from $t$ to $T$, taking expectation, and substituting them into the cost functional $J_\gamma(t,\xi,i;u(\cdot),v(\cdot))$, by simplification, we obtain
\begin{equation*}
\begin{aligned}
	&J_\gamma(t,\xi,i;u(\cdot),v(\cdot))=\mathbb{E}\bigg\{\langle \Pi(t,i)\tilde{\xi},\tilde{\xi}\rangle+\int_t^T\Big[\langle\bar{\mathcal{R}}_2(s,\alpha(s))\tilde{v}(s),\tilde{v}(s)\rangle+2\langle \bar{\mathcal{S}}_2(s,\alpha(s))\tilde{x}(s),\tilde{v}(s)\rangle\\
	&\quad+\langle \bar{\mathcal{S}}_2(s,\alpha(s))^\top\bar{\mathcal{R}}_2(s,\alpha(s))^{-1}\bar{\mathcal{S}}_2(s,\alpha(s))\tilde{x}(s),\tilde{x}(s)\rangle
     +\langle\Pi(s,\alpha(s))\bar{\sigma}(s,\alpha(s)),\bar{\sigma}(s,\alpha(s))\rangle\Big]ds\bigg\}\\
	&\quad+\mathbb{E}\bigg\{\langle P(t,i)\hat{\xi},\hat{\xi}\rangle+2\langle \eta(t,i),\hat{\xi}\rangle
     +\int_t^T\Big[\langle \hat{\mathcal{S}}(s,\alpha(s))^\top\hat{\mathcal{R}}(s,\alpha(s))^{-1}\hat{\mathcal{S}}(s,\alpha(s))\hat{x}(s),\hat{x}(s)\rangle\\
	&\qquad\qquad+\langle \hat{\mathcal{R}}_{11}(s,\alpha(s))u(s),u(s)\rangle +2\langle\hat{\mathcal{S}}_1(s,\alpha(s))\hat{x}(s),u(s)\rangle+\langle\hat{\mathcal{R}}_{22}(s,\alpha(s))\hat{v}(s),\hat{v}(s)\rangle\\
	&\qquad\qquad+2\langle\hat{\mathcal{S}}_2(s,\alpha(s))\hat{x}(s),\hat{v}(s)+2\langle\hat{\mathcal{R}}_{12}(s,\alpha(s))^\top u(s),\hat{v}(s)\rangle\\ 
    &\qquad\qquad+2\langle\Big(\hat{\mathcal{S}}_2^\top\hat{\mathcal{R}}_{22}^{-1}\psi+\big(\hat{\mathcal{S}}_1^\top
     -\hat{\mathcal{R}}_{12}\hat{\mathcal{R}}_{22}^{-1}\hat{\mathcal{S}}_2\big)^\top\big(\hat{\mathcal{R}}_{11}-\hat{\mathcal{R}}_{12}\hat{\mathcal{R}}_{22}^{-1}\hat{\mathcal{R}}_{12}^\top\big)^{-1}\varphi\Big)(s,\alpha(s)),\hat{x}(s)\rangle \\
	&\qquad\qquad+2\langle\varphi(s,\alpha(s))+\hat{\mathcal{R}}_{12}(s,\alpha(s))\hat{\mathcal{R}}_{22}(s,\alpha(s))^{-1}\psi(s,\alpha(s)),u(s)\rangle+2\langle\psi(s,\alpha(s)),\hat{v}(s)\rangle\\
	&\qquad\qquad+\langle P(s,\alpha(s))\sigma(s,\alpha(s)),\sigma(s,\alpha(s))\rangle+2\langle\eta(s,\alpha(s)),b(s,\alpha(s))\rangle \Big]ds\bigg\}\\
	&:=\tilde{J}_\gamma(t,\tilde{\xi},i;\tilde{v}(\cdot))+\hat{J}_\gamma(t,\hat{\xi},i;u(\cdot),\hat{v}(\cdot)).
\end{aligned}
\end{equation*}
Using completion-of-squares for $\tilde{v}(\cdot)$ in the functional $\tilde{J}_\gamma(t,\tilde{\xi},i;\tilde{v}(\cdot))$, we get (the argument $(s,\alpha(s))$ is suppressed for simplicity)
\begin{equation}\label{tilde-J}
\begin{aligned}
	\tilde{J}_\gamma(t,\tilde{\xi},i;\tilde{v}(\cdot))=\mathbb{E}\bigg\{\langle\Pi(t,i)\tilde{\xi},\tilde{\xi}\rangle+\int_{t}^{T}\Big[\langle\bar{\mathcal{R}}_2\big(\tilde{v}
    +\bar{\mathcal{R}}_2^{-1}\bar{\mathcal{S}}_2\tilde{x}\big),\tilde{v}+\bar{\mathcal{R}}_2^{-1}\bar{\mathcal{S}}_2\tilde{x}\rangle+\langle\Pi\bar{\sigma},\bar{\sigma}\rangle\Big]ds\bigg\}.
\end{aligned}
\end{equation}

For $\hat{J}_\gamma(t,\hat{\xi},i;u(\cdot),\hat{v}(\cdot))$, using completion-of-squares for $\hat{v}(\cdot)$, we have
\begin{equation*}
\begin{aligned}
	&\hat{J}_\gamma(t,\hat{\xi},i;u(\cdot),\hat{v}(\cdot))=\mathbb{E}\bigg\{\langle P(t,i)\hat{\xi}+2\eta(t,i),\hat{\xi}\rangle
     +\int_t^T\Big[\langle \hat{\mathcal{S}}^\top\hat{\mathcal{R}}^{-1}\hat{\mathcal{S}}\hat{x},\hat{x}\rangle+\langle \hat{\mathcal{R}}_{11}u,u\rangle\\
	&\quad +2\langle\hat{\mathcal{S}}_1\hat{x},u\rangle+\langle\hat{\mathcal{R}}_{22}\big(\hat{v}+\hat{\mathcal{R}}_{22}^{-1}\hat{\mathcal{R}}_{12}^\top u
     +\hat{\mathcal{R}}_{22}^{-1}\hat{\mathcal{S}}_2\hat{x}\big),\hat{v}+\hat{\mathcal{R}}_{22}^{-1}\hat{\mathcal{R}}_{12}^\top u+\hat{\mathcal{R}}_{22}^{-1}\hat{\mathcal{S}}_2\hat{x}\rangle\\ 
    &\quad-2\langle\hat{\mathcal{R}}_{12}\hat{\mathcal{R}}_{22}^{-1}\hat{\mathcal{S}}_2\hat{x},u\rangle-\langle\hat{\mathcal{S}}_2^\top\hat{\mathcal{R}}_{22}^{-1}\hat{\mathcal{S}}_2\hat{x},\hat{x}\rangle
     -\langle\hat{\mathcal{R}}_{12}\hat{\mathcal{R}}_{22}^{-1}\hat{\mathcal{R}}_{12}^\top u,u\rangle\\ 
    &\quad+2\langle\hat{\mathcal{S}}_2^\top\hat{\mathcal{R}}_{22}^{-1}\psi+\big(\hat{\mathcal{S}}_1^\top-\hat{\mathcal{R}}_{12}\hat{\mathcal{R}}_{22}^{-1}\hat{\mathcal{S}}_2\big)^\top\big(\hat{\mathcal{R}}_{11}
     -\hat{\mathcal{R}}_{12}\hat{\mathcal{R}}_{22}^{-1}\hat{\mathcal{R}}_{12}^\top\big)^{-1}\varphi ,\hat{x}\rangle \\
	&\quad+2\langle\varphi+\hat{\mathcal{R}}_{12}\hat{\mathcal{R}}_{22}^{-1}\psi,u\rangle+2\langle\psi,\hat{v}\rangle+\langle P\sigma,\sigma\rangle+2\langle\eta,b\rangle \Big]ds\bigg\}.
\end{aligned}
\end{equation*}
Then, with the help of Lemma \ref{lem2} (\romannumeral1), substituting (\ref{form1}) into the above equation, we obtain
\begin{equation*}
\begin{aligned}
	&\hat{J}_\gamma(t,\hat{\xi},i;u(\cdot),\hat{v}(\cdot))=\mathbb{E}\bigg\{\langle P(t,i)\hat{\xi}+2\eta(t,i),\hat{\xi}\rangle
     +\int_t^T\Big[\langle\big( \hat{\mathcal{R}}_{11}-\hat{\mathcal{R}}_{12}\hat{\mathcal{R}}_{22}^{-1}\hat{\mathcal{R}}_{12}^\top\big)u,u\rangle\\
	&\quad+2\langle\big(\hat{\mathcal{S}}_1-\hat{\mathcal{R}}_{12}\hat{\mathcal{R}}_{22}^{-1}\hat{\mathcal{S}}_2\big)\hat{x},u\rangle
      +\langle\big(\hat{\mathcal{S}}_1-\hat{\mathcal{R}}_{12}\hat{\mathcal{R}}_{22}^{-1}\hat{\mathcal{S}}_2\big)^\top\big(\hat{\mathcal{R}}_{11}
      -\hat{\mathcal{R}}_{12}\hat{\mathcal{R}}_{22}^{-1}\hat{\mathcal{R}}_{12}^\top\big)^{-1}\big(\hat{\mathcal{S}}_1\\
\end{aligned}
\end{equation*}
\begin{equation*}
\begin{aligned}
	&\quad-\hat{\mathcal{R}}_{12}\hat{\mathcal{R}}_{22}^{-1}\hat{\mathcal{S}}_2\big)\hat{x},\hat{x}\rangle+\langle\hat{\mathcal{R}}_{22}\big(\hat{v}+\hat{\mathcal{R}}_{22}^{-1}
     \hat{\mathcal{R}}_{12}^\top u+\hat{\mathcal{R}}_{22}^{-1}\hat{\mathcal{S}}_2\hat{x}\big),\hat{v}+\hat{\mathcal{R}}_{22}^{-1}\hat{\mathcal{R}}_{12}^\top u+\hat{\mathcal{R}}_{22}^{-1}\hat{\mathcal{S}}_2\hat{x}\rangle\\
	&\quad+2\langle\Big(\hat{\mathcal{S}}_2^\top\hat{\mathcal{R}}_{22}^{-1}\psi+\big(\hat{\mathcal{S}}_1^\top-\hat{\mathcal{R}}_{12}\hat{\mathcal{R}}_{22}^{-1}\hat{\mathcal{S}}_2\big)^\top\big(\hat{\mathcal{R}}_{11}
     -\hat{\mathcal{R}}_{12}\hat{\mathcal{R}}_{22}^{-1}\hat{\mathcal{R}}_{12}^\top\big)^{-1}\varphi\Big) ,\hat{x}\rangle \\
	&\quad+2\langle\varphi+\hat{\mathcal{R}}_{12}\hat{\mathcal{R}}_{22}^{-1}\psi,u\rangle+2\langle\psi,\hat{v}\rangle+\langle P\sigma,\sigma\rangle+2\langle\eta,b\rangle \Big]ds\bigg\}.
\end{aligned}
\end{equation*}
Continue completing the square for $\hat{v}(\cdot)$, we have
\begin{equation*}
\begin{aligned}
	&\hat{J}_\gamma(t,\hat{\xi},i;u(\cdot),\hat{v}(\cdot))=\mathbb{E}\bigg\{\langle P(t,i)\hat{\xi}+2\eta(t,i),\hat{\xi}\rangle
     +\int_t^T\Big[\langle\big( \hat{\mathcal{R}}_{11}-\hat{\mathcal{R}}_{12}\hat{\mathcal{R}}_{22}^{-1}\hat{\mathcal{R}}_{12}^\top\big)u,u\rangle\\
    &\quad+2\langle\big(\hat{\mathcal{S}}_1-\hat{\mathcal{R}}_{12}\hat{\mathcal{R}}_{22}^{-1}\hat{\mathcal{S}}_2\big)\hat{x},u\rangle
     +\langle \big(\hat{\mathcal{S}}_1-\hat{\mathcal{R}}_{12}\hat{\mathcal{R}}_{22}^{-1}\hat{\mathcal{S}}_2\big)^\top\big(\hat{\mathcal{R}}_{11}
     -\hat{\mathcal{R}}_{12}\hat{\mathcal{R}}_{22}^{-1}\hat{\mathcal{R}}_{12}^\top\big)^{-1}\big(\hat{\mathcal{S}}_1\\
    &\quad-\hat{\mathcal{R}}_{12}\hat{\mathcal{R}}_{22}^{-1}\hat{\mathcal{S}}_2\big)\hat{x},\hat{x}\rangle
     +\langle\hat{\mathcal{R}}_{22}\big(\hat{v}+\hat{\mathcal{R}}_{22}^{-1}\hat{\mathcal{R}}_{12}^\top u+\hat{\mathcal{R}}_{22}^{-1}\hat{\mathcal{S}}_2\hat{x}+\hat{\mathcal{R}}_{22}^{-1}\psi\big),
     \hat{v}+\hat{\mathcal{R}}_{22}^{-1}\hat{\mathcal{R}}_{12}^\top u\\
	&\quad+\hat{\mathcal{R}}_{22}^{-1}\hat{\mathcal{S}}_2\hat{x}+\hat{\mathcal{R}}_{22}^{-1}\psi\rangle
      +2\langle\big(\hat{\mathcal{S}}_1^\top-\hat{\mathcal{R}}_{12}\hat{\mathcal{R}}_{22}^{-1}\hat{\mathcal{S}}_2\big)^\top\big(\hat{\mathcal{R}}_{11}
      -\hat{\mathcal{R}}_{12}\hat{\mathcal{R}}_{22}^{-1}\hat{\mathcal{R}}_{12}^\top\big)^{-1}\varphi ,\hat{x}\rangle\\
	&\quad+2\langle\varphi,u\rangle-\langle\hat{\mathcal{R}}_{22}^{-1}\psi,\psi\rangle+\langle P\sigma,\sigma\rangle+2\langle\eta,b\rangle \Big]ds\bigg\}.
\end{aligned}
\end{equation*}
Next, the two-step complication-of-squares procedure for $u(\cdot)$ is implemented,
\begin{equation}\label{hat-J}
\begin{aligned}
	&\hat{J}_\gamma(t,\hat{\xi},i;u(\cdot),\hat{v}(\cdot))\\
	&=\mathbb{E}\bigg\{\langle P(t,i)\hat{\xi}+2\eta(t,i),\hat{\xi}\rangle+\int_t^T\Big[-\langle\hat{\mathcal{R}}_{22}^{-1}\psi,\psi\rangle+\langle P\sigma,\sigma\rangle+2\langle\eta,b\rangle\\
	&\qquad\quad+\big\vert u+( \hat{\mathcal{R}}_{11}-\hat{\mathcal{R}}_{12}\hat{\mathcal{R}}_{22}^{-1}\hat{\mathcal{R}}_{12}^\top)^{-1}(\hat{\mathcal{S}}_1-\hat{\mathcal{R}}_{12}\hat{\mathcal{R}}_{22}^{-1}\hat{\mathcal{S}}_2)\hat{x}
     \big\vert_{(\hat{\mathcal{R}}_{11}-\hat{\mathcal{R}}_{12}\hat{\mathcal{R}}_{22}^{-1}\hat{\mathcal{R}}_{12}^\top)}^2\\
	&\qquad\quad+\big\vert\hat{v}+\hat{\mathcal{R}}_{22}^{-1}\hat{\mathcal{R}}_{12}^\top u+\hat{\mathcal{R}}_{22}^{-1}\hat{\mathcal{S}}_2\hat{x}+\hat{\mathcal{R}}_{22}^{-1}\psi\big\vert_{\hat{\mathcal{R}}_{22}}^2\\
	&\qquad\quad+2\langle\big(\hat{\mathcal{S}}_1^\top-\hat{\mathcal{R}}_{12}\hat{\mathcal{R}}_{22}^{-1}\hat{\mathcal{S}}_2\big)^\top\big(\hat{\mathcal{R}}_{11}-\hat{\mathcal{R}}_{12}
     \hat{\mathcal{R}}_{22}^{-1}\hat{\mathcal{R}}_{12}^\top\big)^{-1}\varphi ,\hat{x}\rangle+2\langle\varphi,u\rangle \Big]ds\bigg\}\\
	&=\mathbb{E}\bigg\{\langle P(t,i)\hat{\xi}+2\eta(t,i),\hat{\xi}\rangle+\int_t^T\Big[-\langle\hat{\mathcal{R}}_{22}^{-1}\psi,\psi\rangle+\langle P\sigma,\sigma\rangle+2\langle\eta,b\rangle\\
	&\qquad\quad+\big\vert u+(\hat{\mathcal{R}}_{11}-\hat{\mathcal{R}}_{12}\hat{\mathcal{R}}_{22}^{-1}\hat{\mathcal{R}}_{12}^\top)^{-1}\big((\hat{\mathcal{S}}_1-\hat{\mathcal{R}}_{12}
     \hat{\mathcal{R}}_{22}^{-1}\hat{\mathcal{S}}_2)\hat{x}+\varphi\big)\big\vert_{( \hat{\mathcal{R}}_{11}-\hat{\mathcal{R}}_{12}\hat{\mathcal{R}}_{22}^{-1}\hat{\mathcal{R}}_{12}^\top)}^2\\
	&\qquad\quad-\langle(\hat{\mathcal{R}}_{11}-\hat{\mathcal{R}}_{12}\hat{\mathcal{R}}_{22}^{-1}\hat{\mathcal{R}}_{12}^\top)^{-1}\varphi,\varphi\rangle+\vert\hat{v}+\hat{\mathcal{R}}_{22}^{-1}
     \hat{\mathcal{R}}_{12}^\top u+\hat{\mathcal{R}}_{22}^{-1}\hat{\mathcal{S}}_2\hat{x}+\hat{\mathcal{R}}_{22}^{-1}\psi\vert_{\hat{\mathcal{R}}_{22}}^2\Big]ds\bigg\}.\\
\end{aligned}
\end{equation}
Therefore, adding equation (\ref{tilde-J}) and equation (\ref{hat-J}) together, we have
\begin{equation*}
\begin{aligned}
	&J_\gamma(t,\xi,i;u(\cdot),v(\cdot))=\tilde{J}_\gamma(t,\tilde{\xi},i;\tilde{v}(\cdot))+\hat{J}_\gamma(t,\hat{\xi},i;u(\cdot),\hat{v}(\cdot))\\
	&=\mathbb{E}\bigg\{\langle \Pi(t,i)\tilde{\xi},\tilde{\xi}\rangle+\langle P(t,i)\hat{\xi}+2\eta(t,i),\hat{\xi}\rangle+\int_{t}^{T}\Big[\langle\bar{\mathcal{R}}_2
     \big(\tilde{v}+\bar{\mathcal{R}}_2^{-1}\bar{\mathcal{S}}_2\tilde{x}\big),\tilde{v}+\bar{\mathcal{R}}_2^{-1}\bar{\mathcal{S}}_2\tilde{x}\rangle\\
	&\qquad\quad+\big\vert u+(\hat{\mathcal{R}}_{11}-\hat{\mathcal{R}}_{12}\hat{\mathcal{R}}_{22}^{-1}\hat{\mathcal{R}}_{12}^\top)^{-1}\big((\hat{\mathcal{S}}_1-\hat{\mathcal{R}}_{12}\hat{\mathcal{R}}_{22}^{-1}
     \hat{\mathcal{S}}_2)\hat{x}+\varphi\big)\big\vert_{(\hat{\mathcal{R}}_{11}-\hat{\mathcal{R}}_{12}\hat{\mathcal{R}}_{22}^{-1}\hat{\mathcal{R}}_{12}^\top)}^2\\
	&\qquad\quad+\big\vert\hat{v}+\hat{\mathcal{R}}_{22}^{-1}\hat{\mathcal{R}}_{12}^\top u+\hat{\mathcal{R}}_{22}^{-1}\hat{\mathcal{S}}_2\hat{x}+\hat{\mathcal{R}}_{22}^{-1}\psi
     \big\vert_{\hat{\mathcal{R}}_{22}}^2-\langle( \hat{\mathcal{R}}_{11}-\hat{\mathcal{R}}_{12}\hat{\mathcal{R}}_{22}^{-1}\hat{\mathcal{R}}_{12}^\top)^{-1}\varphi,\varphi\rangle\\
	&\qquad\quad-\langle\hat{\mathcal{R}}_{22}^{-1}\psi,\psi\rangle+\langle P\sigma,\sigma\rangle+\langle\Pi\bar{\sigma},\bar{\sigma}\rangle+2\langle\eta,b\rangle\Big]ds\bigg\}\\
\end{aligned}
\end{equation*}
\begin{equation}\label{J}
\begin{aligned}
	&=\mathbb{E}\bigg\{\langle \Pi(t,i)\tilde{\xi},\tilde{\xi}\rangle+\langle P(t,i)\hat{\xi}+2\eta(t,i),\hat{\xi}\rangle+\int_{t}^{T}\Big[\langle\bar{\mathcal{R}}_2(\tilde{v}-\tilde{v}^*),\tilde{v}-\tilde{v}^*\rangle\\
	&\qquad\quad+\langle(\hat{\mathcal{R}}_{11}-\hat{\mathcal{R}}_{12}\hat{\mathcal{R}}_{22}^{-1}\hat{\mathcal{R}}_{12}^\top)(u-u^*),u-u^*\rangle+\langle\hat{\mathcal{R}}_{22}(\hat{v}-\hat{\alpha}_2^*(u)),\hat{v}-\hat{\alpha}_2^*(u)\rangle\\
	&\qquad\quad-\langle(\hat{\mathcal{R}}_{11}-\hat{\mathcal{R}}_{12}\hat{\mathcal{R}}_{22}^{-1}\hat{\mathcal{R}}_{12}^\top)^{-1}\varphi,\varphi\rangle-\langle\hat{\mathcal{R}}_{22}^{-1}\psi,\psi\rangle
     +\langle P\sigma,\sigma\rangle+\langle\Pi\bar{\sigma},\bar{\sigma}\rangle+2\langle\eta,b\rangle\Big]ds\bigg\}.
\end{aligned}
\end{equation}
From Condition (\uppercase\expandafter{\romannumeral1}) and the uniform negative definiteness of matrix $\bar{\mathcal{R}}_2(s,i)$, for any $u(\cdot)\in L_{\mathbb{G}}^2(t,T;\mathbb{R}^{m})$ and $\alpha_2(u)(\cdot)=\hat{\alpha}_2(u)(\cdot)+\tilde{v}(\cdot)\in\mathcal{A}_2[t,T]$, we have 
\begin{equation*}
	J_\gamma(t,\xi,i;u(\cdot),\alpha_2(u)(\cdot))\leq J_\gamma(t,\xi,i;u(\cdot),\alpha_2^*(u)(\cdot)).\\
\end{equation*}
The equality in the above inequality holds if and only if $\hat{\alpha}_2(u)(\cdot)=\hat{\alpha}_2^*(u)(\cdot)$ and $\tilde{v}(\cdot)=\tilde{v}^*(\cdot)$, i.e., $\alpha_2(u)(\cdot)=\alpha_2^*(u)(\cdot)$, which is the first assertion of the lemma. Similarly, by virtue of the uniform positive definiteness of matrix $\hat{\mathcal{R}}_{11}-\hat{\mathcal{R}}_{12}\hat{\mathcal{R}}_{22}^{-1}\hat{\mathcal{R}}_{12}^\top$, we obtain the second assertion of the lemma and finish the proof.
\end{proof}

\begin{mythm}\label{Thm1}
Let (H1)-(H3) hold, for any disturbance attenuation level $\gamma>\gamma^*$, assume the Riccati equations (\ref{Pi})-(\ref{P}) admit a solution $(\Pi(\cdot,\cdot),P(\cdot,\cdot))\in C([t,T]\times\mathcal{S};\mathbb{S}^n)\times C([t,T]\times\mathcal{S};\mathbb{S}^n)$ satisfying Condition (\uppercase\expandafter{\romannumeral1}). Then, for any $(\xi,i)\in L_{\mathcal{F}_t}^2(\Omega;\mathbb{R}^n)\times\mathcal{S}$, Problem (SCG-1) admits an optimal control-strategy pair $(u^*(\cdot),\alpha_2^*(\cdot))\in L_{\mathbb{G}}^2(t,T;\mathbb{R}^{m})\times\mathcal{A}_2[t,T]$ for Player 1's value, which is in a feedback form and defined by (\ref{OCS-1}). Moreover, Player 1's value is given by
\begin{equation}\label{value1}
\begin{aligned}
	&J_\gamma(t,\xi,i;u^*(\cdot),\alpha_2^*(u^*)(\cdot))=\mathbb{E}\bigg\{\langle \Pi(t,i)\tilde{\xi},\tilde{\xi}\rangle+\langle P(t,i)\hat{\xi},\hat{\xi}\rangle+2\langle\eta(t,i),\hat{\xi}\rangle\\
	&\quad+\int_t^T\Big[\langle P\sigma,\sigma\rangle+\langle\Pi\bar{\sigma},\bar{\sigma}\rangle+2\langle\eta,b\rangle-\langle( \hat{\mathcal{R}}_{11}-\hat{\mathcal{R}}_{12}\hat{\mathcal{R}}_{22}^{-1}\hat{\mathcal{R}}_{12}^\top)^{-1}\varphi,\varphi\rangle-\langle\hat{\mathcal{R}}_{22}^{-1}\psi,\psi\rangle\Big]ds\bigg\}.\\
\end{aligned}
\end{equation}
\end{mythm}

\begin{proof}
We shall show that the control-strategy pair $(u^*(\cdot),\alpha_2^*(\cdot))$ is optimal, i.e., we shall verify (\ref{CS-1}) holds true. 
By Lemma \ref{lem3} (\romannumeral2),
\begin{equation*}
	J_\gamma(t,\xi,i;u^*(\cdot),\alpha_2^*(u^*)(\cdot))\leq\underset{u(\cdot)\in L_{\mathbb{G}}^2(t,T;\mathbb{R}^{m})}\inf J_\gamma(t,\xi,i;u(\cdot),\alpha_2^*(u)(\cdot)),
\end{equation*}
thus we have
\begin{equation*}
	J_\gamma(t,\xi,i;u^*(\cdot),\alpha_2^*(u^*)(\cdot))=\underset{u(\cdot)\in L_{\mathbb{G}}^2(t,T;\mathbb{R}^{m})}\inf J_\gamma(t,\xi,i;u(\cdot),\alpha_2^*(u)(\cdot)),
\end{equation*}
and
\begin{equation}\label{ineq1}
	J_\gamma(t,\xi,i;u^*(\cdot),\alpha_2^*(u^*)(\cdot))\leq\underset{\alpha_2(\cdot)\in\mathcal{A}_2[t,T]}\sup\underset{u(\cdot)\in L_{\mathbb{G}}^2(t,T;\mathbb{R}^{m})}\inf J_\gamma(t,\xi,i;u(\cdot),\alpha_2(u)(\cdot)).
\end{equation}
Similarly, by Lemma \ref{lem3} (\romannumeral1),
\begin{equation*}
	J_\gamma(t,\xi,i;u(\cdot),\alpha_2^*(u)(\cdot))=\underset{\alpha_2(\cdot)\in\mathcal{A}_2[t,T]}\sup J_\gamma(t,\xi,i;u(\cdot),\alpha_2(u)(\cdot)),\quad \forall u(\cdot)\in L_{\mathbb{G}}^2(t,T;\mathbb{R}^{m}).
\end{equation*}
Therefore,
\begin{equation}\label{ineq2}
\begin{aligned}
	J_\gamma(t,\xi,i;u^*(\cdot),\alpha_2^*(u^*)(\cdot))=&\underset{u(\cdot)\in L_{\mathbb{G}}^2(t,T;\mathbb{R}^{m})}\inf\underset{\alpha_2(\cdot)\in\mathcal{A}_2[t,T]}\sup J_\gamma(t,\xi,i;u(\cdot),\alpha_2(u)(\cdot))\\
	\geq&\underset{\alpha_2(\cdot)\in\mathcal{A}_2[t,T]}\sup\underset{u(\cdot)\in L_{\mathbb{G}}^2(t,T;\mathbb{R}^{m})}\inf J_\gamma(t,\xi,i;u(\cdot),\alpha_2(u)(\cdot)).
\end{aligned}
\end{equation}
From (\ref{ineq1}) and (\ref{ineq2}), we obtain the desired conclusion that $(u^*(\cdot),\alpha_2^*(\cdot))$ is an optimal control-strategy pair of Player 1's value $J_\gamma(t,\xi,i;u^*(\cdot),\alpha_2^*(u^*)(\cdot))$, which comes from the result of the five-step complication-of-squares procedure (see the proof of Lemma \ref{lem3}).
\end{proof}

A similar analysis as above can be applied to Player 2’s value, which leads to the following.

\begin{mythm}\label{Thm2}
Let (H1)-(H3) hold, for any given disturbance attenuation level $\gamma>\gamma^*$, assume the Riccati equations (\ref{Pi}) and (\ref{P}) admit a solution $(\Pi(\cdot,\cdot),P(\cdot,\cdot))\in C([t,T]\times\mathcal{S};\mathbb{S}^n)\times C([t,T]\times\mathcal{S};\mathbb{S}^n)$ satisfying Condition (\uppercase\expandafter{\romannumeral2}). Then, for any $(\xi,i)\in L_{\mathcal{F}_t}^2(\Omega;\mathbb{R}^n)\times\mathcal{S}$, Problem (SCG-2) admits an optimal control-strategy pair $(\alpha_1^*(\cdot),v^*(\cdot))\in \mathcal{A}_1[t,T]\times L_{\mathbb{F}}^2(t,T;\mathbb{R}^{n_v})$ for Player 2's value, which is in a feedback form and defined by
\begin{equation}\label{OCS-2}
\left\{\begin{aligned}
	\alpha_1^*(\hat{v})(s)=&-\hat{\mathcal{R}}_{11}(s,\alpha(s))^{-1}\Big[\hat{\mathcal{R}}_{12}(s,\alpha(s))\hat{v}(s)+\hat{\mathcal{S}}_1(s,\alpha(s))\hat{x}(s)\\
	&\qquad +B_1(s,\alpha(s))^\top\bar{\eta}(s,\alpha(s))+D_1(s,\alpha(s))^\top P(s,\alpha(s))\sigma(s,\alpha(s))\\
    &\qquad +\bar{D}_1(s,\alpha(s))^\top \Pi(s,\alpha(s))\bar{\sigma}(s,\alpha(s))+\rho_1(s,\alpha(s))\Big],\\
	\hat{v}^*(s)=&-\Big[\hat{\mathcal{R}}_{22}(s,\alpha(s))-\hat{\mathcal{R}}_{12}(s,\alpha(s))^\top\hat{\mathcal{R}}_{11}(s,\alpha(s))^{-1}\hat{\mathcal{R}}_{12}(s,\alpha(s))\Big]^{-1}\Bigl\{\Big[\hat{\mathcal{S}}_2(s,\alpha(s))\\
	&\qquad-\hat{\mathcal{R}}_{12}(s,\alpha(s))^\top\hat{\mathcal{R}}_{11}(s,\alpha(s))^{-1}\hat{\mathcal{S}}_1(s,\alpha(s))\Big]\hat{x}(s)+B_2(s,\alpha(s))^\top\bar{\eta}(s,\alpha(s))\\
	&\qquad+D_2(s,\alpha(s))^\top P(s,\alpha(s))\sigma(s,\alpha(s))+\bar{D}_2(s,\alpha(s))^\top \Pi(s,\alpha(s))\bar{\sigma}(s,\alpha(s))\\
	&\qquad+\rho_2(s,\alpha(s))-\hat{\mathcal{R}}_{12}(s,\alpha(s))^\top\hat{\mathcal{R}}_{11}(s,\alpha(s))^{-1}\Big[B_1(s,\alpha(s))^\top\bar{\eta}(s)+D_1(s,\alpha(s))^\top\\
	&\qquad\times P(s,\alpha(s))\sigma(s,\alpha(s))+\bar{D}_1(s,\alpha(s))^\top \Pi(s,\alpha(s))\bar{\sigma}(s,\alpha(s))+\rho_1(s,\alpha(s))\Big]\Bigr\},\\
	\tilde{v}^*(s)=&-\bar{\mathcal{R}}_2(s,\alpha(s))^{-1}\bar{\mathcal{S}}_2(s,\alpha(s))\tilde{x}(s),\quad v^*(s)=\hat{v}^*(s)+\tilde{v}^*(s),\quad s\in[t,T],
\end{aligned}\right.
\end{equation}
where $\bar{\eta}(\cdot,\cdot)\in C([t,T]\times\mathcal{S};\mathbb{R}^n)$ is the solution of the following BODE:
\begin{equation}\label{eta2}
\left\{\begin{aligned} 
    &\dot{\bar{\eta}}(s,i)+A(s,i)^\top\bar{\eta}(s,i)-\hat{\mathcal{S}}_1(s,i)^\top\hat{\mathcal{R}}_{11}(s,i)^{-1}\bar{\psi}(s,i)-\big(\hat{\mathcal{S}}_2(s,i)
     -\hat{\mathcal{R}}_{12}(s,i)^\top\hat{\mathcal{R}}_{11}(s,i)^{-1}\\
	&\quad\times\hat{\mathcal{S}}_1(s,i)\big)^\top\big(\hat{\mathcal{R}}_{22}(s,i)-\hat{\mathcal{R}}_{12}(s,i)^\top\hat{\mathcal{R}}_{11}(s,i)^{-1}\hat{\mathcal{R}}_{12}(s,i)\big)^{-1}\bar{\varphi}(s,i)+C(s,i)^\top P(s,i)\sigma(s,i)\\
	&\quad+\bar{C}(s,i)^\top \Pi(s,i)\bar{\sigma}(s,i)+P(s,i)b(s,i)+q(s,i)+\sum_{j=1}^{D}\lambda_{ij}\bar{\eta}(s,j)=0,\qquad a.e.\; s\in[t,T],\\
	&\bar{\eta}(T,i)=g(T,i),\quad i\in\mathcal{S},
\end{aligned}\right.
\end{equation}
and 
\begin{equation*}
\begin{aligned}
	\bar{\psi}(s,i)&:=B_1(s,i)^\top\bar{\eta}(s,i)+D_1(s,i)^\top P(s,i)\sigma(s,i)+\bar{D}_1(s,i)^\top \Pi(s,i)\bar{\sigma}(s,i)+\rho_1(s,i),\\
	\bar{\varphi}(s,i)&:=B_2(s,i)^\top\bar{\eta}(s,i)+D_2(s,i)^\top P(s,i)\sigma(s,i)+\bar{D}_2(s,i)^\top \Pi(s,i)\bar{\sigma}(s,i)+\rho_2(s,i)\\
	&\quad -\hat{\mathcal{R}}_{12}(s,i)^\top\hat{\mathcal{R}}_{11}(s,i)^{-1}\bar{\psi}(s,i).
\end{aligned}
\end{equation*}
Moreover, Player 2's value is given by
\begin{equation}\label{value2}
\begin{aligned}
	&J_\gamma(t,\xi,i;\alpha_1^*(\hat{v}^*)(\cdot),v^*(\cdot))=\mathbb{E}\bigg\{\langle \Pi(t,i)\tilde{\xi},\tilde{\xi}\rangle+\langle P(t,i)\hat{\xi},\hat{\xi}\rangle+2\langle\bar{\eta}(t,i),\hat{\xi}\rangle\\
	&\quad+\int_t^T\Big[\langle P\sigma,\sigma\rangle+\langle\Pi\bar{\sigma},\bar{\sigma}\rangle+2\langle\bar{\eta},b\rangle
     -\langle(\hat{\mathcal{R}}_{22}-\hat{\mathcal{R}}_{12}^\top\hat{\mathcal{R}}_{11}^{-1}\hat{\mathcal{R}}_{12})^{-1}\bar{\varphi},\bar{\varphi}\rangle-\langle\hat{\mathcal{R}}_{11}^{-1}\bar{\psi},\bar{\psi}\rangle\Big]ds\bigg\}.
\end{aligned}
\end{equation}
\end{mythm}

Next, to prove that the optimal control-strategy pair of Player 1 is consistent with that of Player 2, we first need to present the following result, which demonstrates the equivalence of the two BODEs given earlier.

\begin{mypro}\label{equivalence}
If the solutions $\eta(\cdot,\cdot)$ and $\bar{\eta}(\cdot,\cdot)$ exist for the BODEs (\ref{eta1}) and (\ref{eta2}), respectively, then the solutions are equal.
\end{mypro}

\begin{proof}
It can be seen from two BODEs (\ref{eta1}) and (\ref{eta2}) that, to prove the equality of the solutions $\eta(\cdot,\cdot)$ and $\bar{\eta}(\cdot,\cdot)$, it suffices to show that the corresponding coefficients in $-\hat{\mathcal{S}}_2^\top\hat{\mathcal{R}}_{22}^{-1}\psi-\big(\hat{\mathcal{S}}_1-\hat{\mathcal{R}}_{12}\hat{\mathcal{R}}_{22}^{-1}\hat{\mathcal{S}}_2\big)^\top
\big(\hat{\mathcal{R}}_{11}-\hat{\mathcal{R}}_{12}\hat{\mathcal{R}}_{22}^{-1}\hat{\mathcal{R}}_{12}^\top\big)^{-1}\varphi$, $-\hat{\mathcal{S}}_1^\top\hat{\mathcal{R}}_{11}^{-1}\bar{\psi}-\big(\hat{\mathcal{S}}_2-\hat{\mathcal{R}}_{12}^\top\hat{\mathcal{R}}_{11}^{-1}\hat{\mathcal{S}}_1\big)^\top
\big(\hat{\mathcal{R}}_{22}-\hat{\mathcal{R}}_{12}^\top\hat{\mathcal{R}}_{11}^{-1}\hat{\mathcal{R}}_{12}\big)^{-1}\bar{\varphi}$ are equal correspondingly.

By simplification and combining like terms, we have
\begin{equation*}
\begin{aligned} 
    &-\hat{\mathcal{S}}_2^\top\hat{\mathcal{R}}_{22}^{-1}\psi-\big(\hat{\mathcal{S}}_1-\hat{\mathcal{R}}_{12}\hat{\mathcal{R}}_{22}^{-1}\hat{\mathcal{S}}_2\big)^\top\big(\hat{\mathcal{R}}_{11}
     -\hat{\mathcal{R}}_{12}\hat{\mathcal{R}}_{22}^{-1}\hat{\mathcal{R}}_{12}^\top\big)^{-1}\varphi\\
	&=\Big[\big(\hat{\mathcal{S}}_1-\hat{\mathcal{R}}_{12}\hat{\mathcal{R}}_{22}^{-1}\hat{\mathcal{S}}_2\big)^\top\big(\hat{\mathcal{R}}_{11}-\hat{\mathcal{R}}_{12}\hat{\mathcal{R}}_{22}^{-1}\hat{\mathcal{R}}_{12}^\top\big)^{-1}
     \hat{\mathcal{R}}_{12}\hat{\mathcal{R}}_{22}^{-1}-\hat{\mathcal{S}}_2^\top\hat{\mathcal{R}}_{22}^{-1}\Big]\big(B_2^\top\eta+D_2^\top P\sigma+\bar{D}_2^\top \Pi\bar{\sigma}\\
	&\quad+\rho_2\big)-\big(\hat{\mathcal{S}}_1-\hat{\mathcal{R}}_{12}\hat{\mathcal{R}}_{22}^{-1}\hat{\mathcal{S}}_2\big)^\top\big(\hat{\mathcal{R}}_{11}-\hat{\mathcal{R}}_{12}\hat{\mathcal{R}}_{22}^{-1}
     \hat{\mathcal{R}}_{12}^\top\big)^{-1}\big(B_1^\top\eta+D_1^\top P\sigma+\bar{D}_1^\top \Pi\bar{\sigma}+\rho_1\big),
\end{aligned}
\end{equation*}
and
\begin{equation*}
\begin{aligned} 
    &-\hat{\mathcal{S}}_1^\top\hat{\mathcal{R}}_{11}^{-1}\bar{\psi}-\big(\hat{\mathcal{S}}_2-\hat{\mathcal{R}}_{12}^\top\hat{\mathcal{R}}_{11}^{-1}\hat{\mathcal{S}}_1\big)^\top\big(\hat{\mathcal{R}}_{22}
     -\hat{\mathcal{R}}_{12}^\top\hat{\mathcal{R}}_{11}^{-1}\hat{\mathcal{R}}_{12}\big)^{-1}\bar{\varphi}\\
	&=\Big[\big(\hat{\mathcal{S}}_2-\hat{\mathcal{R}}_{12}^\top\hat{\mathcal{R}}_{11}^{-1}\hat{\mathcal{S}}_1\big)^\top\big(\hat{\mathcal{R}}_{22}-\hat{\mathcal{R}}_{12}^\top\hat{\mathcal{R}}_{11}^{-1}
     \hat{\mathcal{R}}_{12}\big)^{-1}\hat{\mathcal{R}}_{12}^\top\hat{\mathcal{R}}_{11}^{-1}-\hat{\mathcal{S}}_1^\top\hat{\mathcal{R}}_{11}^{-1}\Big]\big(B_1^\top\bar{\eta}+D_1^\top P\sigma+\bar{D}_1^\top \Pi\bar{\sigma}\\
	&\quad+\rho_1\big)-\big(\hat{\mathcal{S}}_2-\hat{\mathcal{R}}_{12}^\top\hat{\mathcal{R}}_{11}^{-1}\hat{\mathcal{S}}_1\big)^\top\big(\hat{\mathcal{R}}_{22}-\hat{\mathcal{R}}_{12}^\top\hat{\mathcal{R}}_{11}^{-1}
     \hat{\mathcal{R}}_{12}\big)^{-1}\big(B_2^\top\bar{\eta}+D_2^\top P\sigma+\bar{D}_2^\top \Pi\bar{\sigma}+\rho_2\big).
\end{aligned}
\end{equation*}
In fact, it suffices to verify that the corresponding coefficients pertaining to $B_2^\top\eta+D_2^\top P\sigma+\bar{D}_2^\top \Pi\bar{\sigma}+\rho_2$ and $B_2^\top\bar{\eta}+D_2^\top P\sigma+\bar{D}_2^\top \Pi\bar{\sigma}+\rho_2$  as well as  $B_1^\top\eta+D_1^\top P\sigma+\bar{D}_1^\top \Pi\bar{\sigma}+\rho_1$ and $B_1^\top\bar{\eta}+D_1^\top P\sigma+\bar{D}_1^\top \Pi\bar{\sigma}+\rho_1$ are equal to complete the proof.

Next, we verify only the first one, the second one can be proved using a similar technique. For the sake of computational convenience, we assume that the matrix $\hat{\mathcal{R}}_{12}$ is invertible (otherwise, use the (Moore–Penrose)
\emph{pseudoinverse} of $\hat{\mathcal{R}}_{12}$), we have
\begin{equation*}
\begin{aligned}
	&\big(\hat{\mathcal{S}}_1-\hat{\mathcal{R}}_{12}\hat{\mathcal{R}}_{22}^{-1}\hat{\mathcal{S}}_2\big)^\top\big(\hat{\mathcal{R}}_{11}-\hat{\mathcal{R}}_{12}\hat{\mathcal{R}}_{22}^{-1}
\hat{\mathcal{R}}_{12}^\top\big)^{-1}\hat{\mathcal{R}}_{12}\hat{\mathcal{R}}_{22}^{-1}-\hat{\mathcal{S}}_2^\top\hat{\mathcal{R}}_{22}^{-1}\\
	&=\big(\hat{\mathcal{S}}_1-\hat{\mathcal{R}}_{12}\hat{\mathcal{R}}_{22}^{-1}\hat{\mathcal{S}}_2\big)^\top\big(I-\hat{\mathcal{R}}_{11}^{-1}\hat{\mathcal{R}}_{12}
\hat{\mathcal{R}}_{22}^{-1}\hat{\mathcal{R}}_{12}^\top\big)^{-1}\hat{\mathcal{R}}_{11}^{-1}\hat{\mathcal{R}}_{12}\hat{\mathcal{R}}_{22}^{-1}-\hat{\mathcal{S}}_2^\top\hat{\mathcal{R}}_{22}^{-1}\\
	&=\big(\hat{\mathcal{S}}_1-\hat{\mathcal{R}}_{12}\hat{\mathcal{R}}_{22}^{-1}\hat{\mathcal{S}}_2\big)^\top\hat{\mathcal{R}}_{11}^{-1}\big(I-\hat{\mathcal{R}}_{12}
\hat{\mathcal{R}}_{22}^{-1}\hat{\mathcal{R}}_{12}^\top\hat{\mathcal{R}}_{11}^{-1}\big)^{-1}\hat{\mathcal{R}}_{12}\hat{\mathcal{R}}_{22}^{-1}-\hat{\mathcal{S}}_2^\top\hat{\mathcal{R}}_{22}^{-1}\\
	&=\big(\hat{\mathcal{R}}_{11}^{-1}\hat{\mathcal{S}}_1-\hat{\mathcal{R}}_{11}^{-1}\hat{\mathcal{R}}_{12}\hat{\mathcal{R}}_{22}^{-1}\hat{\mathcal{S}}_2\big)^\top
\Big[\big(\hat{\mathcal{R}}_{12}\hat{\mathcal{R}}_{22}^{-1}\big)^{-1}-\hat{\mathcal{R}}_{12}^\top\hat{\mathcal{R}}_{11}^{-1}\Big]^{-1}-\hat{\mathcal{S}}_2^\top\hat{\mathcal{R}}_{22}^{-1}\\
	&=\Big[\big(\hat{\mathcal{R}}_{11}^{-1}\hat{\mathcal{S}}_1-\hat{\mathcal{R}}_{11}^{-1}\hat{\mathcal{R}}_{12}\hat{\mathcal{R}}_{22}^{-1}\hat{\mathcal{S}}_2\big)^\top
-\hat{\mathcal{S}}_2^\top\big(\hat{\mathcal{R}}_{12}^{-1}-\hat{\mathcal{R}}_{22}^{-1}\hat{\mathcal{R}}_{12}^\top\hat{\mathcal{R}}_{11}^{-1}\big)\Big]
\Big[\big(\hat{\mathcal{R}}_{12}\hat{\mathcal{R}}_{22}^{-1}\big)^{-1}-\hat{\mathcal{R}}_{12}^\top\hat{\mathcal{R}}_{11}^{-1}\Big]^{-1}\\
	&=\big(\hat{\mathcal{S}}_1^\top\hat{\mathcal{R}}_{11}^{-1}-\hat{\mathcal{S}}_2^\top\hat{\mathcal{R}}_{12}^{-1}\big)\Big[\big(\hat{\mathcal{R}}_{12}
\hat{\mathcal{R}}_{22}^{-1}\big)^{-1}-\hat{\mathcal{R}}_{12}^\top\hat{\mathcal{R}}_{11}^{-1}\Big]^{-1}\\
	&=\big(\hat{\mathcal{S}}_1^\top\hat{\mathcal{R}}_{11}^{-1}\hat{\mathcal{R}}_{12}-\hat{\mathcal{S}}_2^\top\big)\hat{\mathcal{R}}_{12}^{-1}\Big[\big(\hat{\mathcal{R}}_{12}
\hat{\mathcal{R}}_{22}^{-1}\big)^{-1}-\hat{\mathcal{R}}_{12}^\top\hat{\mathcal{R}}_{11}^{-1}\Big]^{-1}\\
	&=-\big(\hat{\mathcal{S}}_2-\hat{\mathcal{R}}_{12}^\top\hat{\mathcal{R}}_{11}^{-1}\hat{\mathcal{S}}_1\big)^\top
\big(\hat{\mathcal{R}}_{22}-\hat{\mathcal{R}}_{12}^\top\hat{\mathcal{R}}_{11}^{-1}\hat{\mathcal{R}}_{12}\Big)^{-1}.
\end{aligned}
\end{equation*}
Based on the above analysis and derivation, subtracting the two BODEs yields $\eta(\cdot,\cdot)=\bar{\eta}(\cdot,\cdot)$. The proof is completed.
\end{proof}

\begin{mythm}\label{Thm-OCS}
Let (H1)-(H3) hold, for any given disturbance attenuation level $\gamma>\gamma^*$, assume the Riccati equations (\ref{Pi}) and (\ref{P}) admit a solution $(\Pi(\cdot,\cdot),P(\cdot,\cdot))\in C([t,T]\times\mathcal{S};\mathbb{S}^n)\times C([t,T]\times\mathcal{S};\mathbb{S}^n)$ satisfying Condition (\uppercase\expandafter{\romannumeral1}$\&$\uppercase\expandafter{\romannumeral2}). Then, for any $(\xi,i)\in L_{\mathcal{F}_t}^2(\Omega;\mathbb{R}^n)\times\mathcal{S}$, we have
\begin{equation}\label{equal}
	u^*=\alpha_1^*(\hat{v}^*),\qquad \hat{v}^*=\hat{\alpha}_2^*(u^*),\qquad v^*=\alpha_2^*(u^*)=\hat{v}^*+\tilde{v}^*,
\end{equation}
where $u^*(\cdot)$, $\hat{\alpha}_2^*(\cdot)$, $\tilde{v}^*(\cdot)$, $\alpha_2^*(\cdot)$ and $\alpha_1^*(\cdot)$, $\hat{v}^*(\cdot)$, $v^*(\cdot)$ are defined by (\ref{OCS-1}) and (\ref{OCS-2}) with $\bar{\eta}(\cdot,\cdot)$ replaced by $\eta(\cdot,\cdot)$, respectively. Moreover, the value of Problem (SCG) exists, which is given by
\begin{equation}\label{value}
\begin{aligned}
	&J_\gamma(t,\xi,i;u^*(\cdot),v^*(\cdot))\\
	&=\mathbb{E}\bigg\{\langle \Pi(t,i)\tilde{\xi},\tilde{\xi}\rangle+\langle P(t,i)\hat{\xi},\hat{\xi}\rangle+2\langle\eta(t,i),\hat{\xi}\rangle+\int_{t}^{T}\Big[\langle P(s,\alpha(s))\sigma(s,\alpha(s)),\sigma(s,\alpha(s))\rangle\\
	&\qquad\quad+\langle\Pi(s,\alpha(s))\bar{\sigma}(s,\alpha(s)),\bar{\sigma}(s,\alpha(s))\rangle+2\langle\eta(s,\alpha(s)),b(s,\alpha(s))\rangle\\
	&\qquad\quad-\langle\hat{\mathcal{R}}(s,\alpha(s))^{-1}\Psi(s,\alpha(s)),\Psi(s,\alpha(s))\rangle\Big]ds\bigg\},
\end{aligned}
\end{equation}
where
\begin{equation*}
\begin{aligned}
	\Psi(s,\alpha(s))\equiv\begin{pmatrix}
		\bar{\psi}(s,\alpha(s)) \\ \psi(s,\alpha(s))
	\end{pmatrix}&:=B(s,\alpha(s))^\top\eta(s,\alpha(s))+D(s,\alpha(s))^\top P(s,\alpha(s))\sigma(s,\alpha(s))\\
	&\quad +\bar{D}(s,\alpha(s))^\top\Pi(s,\alpha(s))\bar{\sigma}(s,\alpha(s))+\rho(s,\alpha(s)),
\end{aligned}
\end{equation*}
and $\eta(\cdot,\cdot)\in C([t,T]\times\mathcal{S};\mathbb{R}^n)$ satisfies the following BODE:
\begin{equation}\label{eta}
\hspace{-3mm}\left\{\begin{aligned}
	&\dot{\eta}(s,i)+\big(A(s,i)-B(s,i)\hat{\mathcal{R}}(s,i)^{-1}\hat{\mathcal{S}}(s,i)\big)^\top\eta(s,i)+\big(C(s,i)-D(s,i)\hat{\mathcal{R}}(s,i)^{-1}\hat{\mathcal{S}}(s,i)\big)^\top\\
	&\quad\times P(s,i)\sigma(s,i)+\big(\bar{C}(s,i)-\bar{D}(s,i)\hat{\mathcal{R}}(s,i)^{-1}\hat{\mathcal{S}}(s,i)\big)^\top\Pi(s,i)\bar{\sigma}(s,i)-\hat{\mathcal{S}}(s,i)^\top\\
	&\quad\times\hat{\mathcal{R}}(s,i)^{-1}\rho(s,i)+P(s,i)b(s,i)+q(s,i)+\sum_{j=1}^{D}\lambda_{ij}\eta(s,j)=0,\quad a.e.\; s\in[t,T],\\
	&\eta(T,i)=g(T,i),\quad i\in\mathcal{S}.
\end{aligned}\right.
\end{equation}
\end{mythm}

\begin{proof}
Recall that under Condition (\uppercase\expandafter{\romannumeral1}$\&$\uppercase\expandafter{\romannumeral2}), both $(u^*,\hat{\alpha}_2^*(u^*))$ and $(\alpha_1^*(\hat{v}^*),\hat{v}^*)$ are the unique solution of the following algebra equations system:
\begin{equation}\label{algebra eq}
\left\{\begin{aligned}
	&\hat{\mathcal{R}}_{11}u+\hat{\mathcal{R}}_{12}\hat{v}+\hat{\mathcal{S}}_1\hat{x}+B_1^\top\eta+D_1^\top P\sigma+\bar{D}_1^\top\Pi\bar{\sigma}+\rho_1=0,\\
	&\hat{\mathcal{R}}_{12}^\top u+\hat{\mathcal{R}}_{22}\hat{v}+\hat{\mathcal{S}}_2\hat{x}+B_2^\top\eta+D_2^\top P\sigma+\bar{D}_2^\top\Pi\bar{\sigma}+\rho_2=0,
\end{aligned}\right.
\end{equation} 
where the argument $(s,i)$ is suppressed for simplicity. In fact, regarding $\hat{v}$ as a function of $(\hat{x},u)$ and solving it from
the second equation of (\ref{algebra eq}), substituting the expression of $\hat{v}$ into the first equation of (\ref{algebra eq}), we have
\begin{equation}\label{solution1}
\left\{\begin{aligned}
	u=&-\big(\hat{\mathcal{R}}_{11}-\hat{\mathcal{R}}_{12}\hat{\mathcal{R}}_{22}^{-1}\hat{\mathcal{R}}_{12}^\top\big)^{-1}\big[\big(\hat{\mathcal{S}}_1
       -\hat{\mathcal{R}}_{12}\hat{\mathcal{R}}_{22}^{-1}\hat{\mathcal{S}}_2\big)\hat{x}+B_1^\top\eta+D_1^\top P\sigma+\bar{D}_1^\top\Pi\bar{\sigma}+\rho_1\\
	&\qquad\qquad\qquad\qquad\qquad\qquad-\hat{\mathcal{R}}_{12}\hat{\mathcal{R}}_{22}^{-1}\big(B_2^\top\eta+D_2^\top P\sigma+\bar{D}_2^\top\Pi\bar{\sigma}+\rho_2\big)\big],\\
	\hat{v}=&-\hat{\mathcal{R}}_{22}^{-1}\big[\hat{\mathcal{R}}_{12}^\top u+\hat{\mathcal{S}}_2\hat{x}+B_2^\top\eta+D_2^\top P\sigma+\bar{D}_2^\top\Pi\bar{\sigma}+\rho_2\big],\\
\end{aligned}\right.
\end{equation}
which is coincides with $(u^*,\hat{\alpha}_2^*(u^*))$ given by (\ref{OCS-1}).
On the other hand, regarding $u$ as a function of $(\hat{x},\hat{v})$ and solving it from
the first equation of (\ref{algebra eq}), substituting the expression of $u$ into the second equation of (\ref{algebra eq}), we get
\begin{equation}\label{solution2}
\left\{\begin{aligned}
	u=&-\hat{\mathcal{R}}_{11}^{-1}\big[\hat{\mathcal{R}}_{12}\hat{v}+\hat{\mathcal{S}}_1\hat{x}+B_1^\top\eta+D_1^\top P\sigma+\bar{D}_1^\top\Pi\bar{\sigma}+\rho_1\big],\\
	\hat{v}=&-\big(\hat{\mathcal{R}}_{22}-\hat{\mathcal{R}}_{12}^\top\hat{\mathcal{R}}_{11}^{-1}\hat{\mathcal{R}}_{12}\big)^{-1}\big[\big(\hat{\mathcal{S}}_2
     -\hat{\mathcal{R}}_{12}^\top\hat{\mathcal{R}}_{11}^{-1}\hat{\mathcal{S}}_1\big)\hat{x}+B_2^\top\eta+D_2^\top P\sigma+\bar{D}_2^\top\Pi\bar{\sigma}+\rho_2\\
	&\qquad\qquad\qquad\qquad\qquad\qquad-\hat{\mathcal{R}}_{12}^\top\hat{\mathcal{R}}_{11}^{-1}\big(B_1^\top\eta+D_1^\top P\sigma+\bar{D}_1^\top\Pi\bar{\sigma}+\rho_1\big)\big].\\
\end{aligned}\right.
\end{equation}
It can be seen that (\ref{solution2}) is the same as $(\alpha_1^*(\hat{v}^*),\hat{v}^*)$ denoted by (\ref{OCS-2}). Moreover, by virtue of the existence and uniqueness of the solution to the algebraic equations (\ref{algebra eq}), we have
\begin{equation*}
	u^*=\alpha_1^*(\hat{v}^*),\qquad \hat{v}^*=\hat{\alpha}_2^*(u^*).\\
\end{equation*}
Therefore, 
\begin{equation*}
	v^*=\hat{v}^*+\tilde{v}^*=\hat{\alpha}_2^*(u^*)+\tilde{v}^*=\alpha_2^*(u^*),
\end{equation*}
then the corresponding value are equal, i.e.,
\begin{equation*}
	J_\gamma(t,\xi,i;u^*,\alpha_2^*(u^*))=J_\gamma(t,\xi,i;\alpha_1^*(\hat{v}^*),v^*).
\end{equation*}
From (\ref{value1}), (\ref{algebra eq}) and (\ref{solution1}), we have
\begin{equation*}
\begin{aligned}
	&J_\gamma(t,\xi,i;u^*,v^*)=J_\gamma(t,\xi,i;u^*,\alpha_2^*(u^*))\\
    &=\mathbb{E}\bigg\{\langle \Pi(t,i)\tilde{\xi},\tilde{\xi}\rangle+\langle P(t,i)\hat{\xi},\hat{\xi}\rangle+2\langle\eta(t,i),\hat{\xi}\rangle+\int_t^T\Big[\langle P\sigma,\sigma\rangle+\langle\Pi\bar{\sigma},\bar{\sigma}\rangle+2\langle\eta,b\rangle\\
    &\qquad\quad-\langle\hat{\mathcal{R}}_{22}^{-1}\big(\hat{\mathcal{R}}_{12}^\top u^*+\hat{\mathcal{R}}_{22}\hat{v}^*+\hat{\mathcal{S}}_2\hat{x}\big),\hat{\mathcal{R}}_{12}^\top u^*+\hat{\mathcal{R}}_{22}\hat{v}^*+\hat{\mathcal{S}}_2\hat{x}\rangle\\
    &\qquad\quad-\big\langle(\hat{\mathcal{R}}_{11}-\hat{\mathcal{R}}_{12}\hat{\mathcal{R}}_{22}^{-1}\hat{\mathcal{R}}_{12}^\top)^{-1}\big[\big(\hat{\mathcal{R}}_{11}-\hat{\mathcal{R}}_{12}\hat{\mathcal{R}}_{22}^{-1}\hat{\mathcal{R}}_{12}^\top\big)u^*
     +\big(\hat{\mathcal{S}}_1-\hat{\mathcal{R}}_{12}\hat{\mathcal{R}}_{22}^{-1}\hat{\mathcal{S}}_2\big)\hat{x}\big],\\
    &\qquad\qquad\big(\hat{\mathcal{R}}_{11}-\hat{\mathcal{R}}_{12}\hat{\mathcal{R}}_{22}^{-1}\hat{\mathcal{R}}_{12}^\top\big)u^*+\big(\hat{\mathcal{S}}_1
     -\hat{\mathcal{R}}_{12}\hat{\mathcal{R}}_{22}^{-1}\hat{\mathcal{S}}_2\big)\hat{x}\big\rangle\Big]ds\bigg\}\\
    &=\mathbb{E}\bigg\{\langle \Pi(t,i)\tilde{\xi},\tilde{\xi}\rangle+\langle P(t,i)\hat{\xi},\hat{\xi}\rangle+2\langle\eta(t,i),\hat{\xi}\rangle+\int_{t}^{T}\Big[\langle P\sigma,\sigma\rangle+\langle\Pi\bar{\sigma},\bar{\sigma}\rangle+2\langle\eta,b\rangle\\
    &\qquad\quad-\big\langle\big[\hat{\mathcal{S}}_2^\top\hat{\mathcal{R}}_{22}^{-1}\hat{\mathcal{S}}_2+\big(\hat{\mathcal{S}}_1-\hat{\mathcal{R}}_{12}\hat{\mathcal{R}}_{22}^{-1}\hat{\mathcal{S}}_2\big)^\top
     (\hat{\mathcal{R}}_{11}-\hat{\mathcal{R}}_{12}\hat{\mathcal{R}}_{22}^{-1}\hat{\mathcal{R}}_{12}^\top)^{-1}\big(\hat{\mathcal{S}}_1-\hat{\mathcal{R}}_{12}\hat{\mathcal{R}}_{22}^{-1}\hat{\mathcal{S}}_2\big)\big]\hat{x},\hat{x}\big\rangle\\
    &\qquad\quad-\langle\hat{\mathcal{R}}_{22}\hat{v}^*,\hat{v}^*\rangle-\langle\hat{\mathcal{R}}_{11}u^*,u^*\rangle-2\langle\hat{\mathcal{S}}_1\hat{x},u^*\rangle-2\langle\hat{\mathcal{S}}_2\hat{x},\hat{v}^*\rangle-2\langle\hat{\mathcal{R}}_{12}^\top u^*,\hat{v}^*\rangle\Big]ds\bigg\}.\\
\end{aligned}
\end{equation*}
Using Lemma \ref{lem2} (\romannumeral1) and the completion-of-squares method, we obtain
\begin{equation}\label{eq1}
\begin{aligned}
	&J_\gamma(t,\xi,i;u^*,v^*)=\mathbb{E}\bigg\{\langle \Pi(t,i)\tilde{\xi},\tilde{\xi}\rangle+\langle P(t,i)\hat{\xi},\hat{\xi}\rangle+2\langle\eta(t,i),\hat{\xi}\rangle\\ 
    &\qquad+\int_t^T\Big[\langle P\sigma,\sigma\rangle+\langle\Pi\bar{\sigma},\bar{\sigma}\rangle+2\langle\eta,b\rangle-\langle\hat{\mathcal{S}}^\top\hat{\mathcal{R}}^{-1}\hat{\mathcal{S}}\hat{x},\hat{x}\rangle
     -\langle\hat{\mathcal{R}}_{22}\hat{v}^*,\hat{v}^*\rangle\\
    &\qquad-\langle\hat{\mathcal{R}}_{11}u^*,u^*\rangle-2\langle\hat{\mathcal{S}}_1\hat{x},u^*\rangle-2\langle\hat{\mathcal{S}}_2\hat{x},\hat{v}^*\rangle-2\langle\hat{\mathcal{R}}_{12}^\top u^*,\hat{v}^*\rangle\Big]ds\bigg\}\\
	&=\mathbb{E}\bigg\{\langle \Pi(t,i)\tilde{\xi},\tilde{\xi}\rangle+\langle P(t,i)\hat{\xi},\hat{\xi}\rangle+2\langle\eta(t,i),\hat{\xi}\rangle
     +\int_t^T\Big[\langle P\sigma,\sigma\rangle+\langle\Pi\bar{\sigma},\bar{\sigma}\rangle\\
	&\qquad+2\langle\eta,b\rangle-\langle\hat{\mathcal{S}}^\top\hat{\mathcal{R}}^{-1}\hat{\mathcal{S}}\hat{x},\hat{x}\rangle-\bigg\langle\hat{\mathcal{R}}\begin{pmatrix}
		u^* \\ \hat{v}^*
	\end{pmatrix},\begin{pmatrix}
	u^* \\ \hat{v}^*
	\end{pmatrix}\bigg\rangle-2\bigg\langle\hat{\mathcal{S}}\hat{x},\begin{pmatrix}
	u^* \\ \hat{v}^*
	\end{pmatrix}\bigg\rangle\Big]ds\bigg\}\\
	&=\mathbb{E}\bigg\{\langle \Pi(t,i)\tilde{\xi},\tilde{\xi}\rangle+\langle P(t,i)\hat{\xi},\hat{\xi}\rangle+2\langle\eta(t,i),\hat{\xi}\rangle
     +\int_t^T\Big[\langle P\sigma,\sigma\rangle+\langle\Pi\bar{\sigma},\bar{\sigma}\rangle\\
	&\qquad+2\langle\eta,b\rangle-\bigg\langle\hat{\mathcal{R}}\Big(\begin{pmatrix}
		u^* \\ \hat{v}^*
	\end{pmatrix}+\hat{\mathcal{R}}^{-1}\hat{\mathcal{S}}\hat{x}\Big),\begin{pmatrix}
	u^* \\ \hat{v}^*
	\end{pmatrix}+\hat{\mathcal{R}}^{-1}\hat{\mathcal{S}}\hat{x}\bigg\rangle\Big]ds\bigg\}.\\
\end{aligned}
\end{equation}
Note that (\ref{algebra eq}) is equivalent to
\begin{equation*}
	\hat{\mathcal{R}}\begin{pmatrix}
		u \\ \hat{v}\end{pmatrix}+\hat{\mathcal{S}}\hat{x}+B^\top\eta+D^\top P\sigma+\bar{D}^\top\Pi\bar{\sigma}+\rho=0.
\end{equation*}
Since the matrix $\hat{\mathcal{R}}$ is invertible under Condition (\uppercase\expandafter{\romannumeral1}$\&$\uppercase\expandafter{\romannumeral2}), it follows that
\begin{equation*}
	\begin{pmatrix}
		u \\ \hat{v}\end{pmatrix}+\hat{\mathcal{R}}^{-1}\hat{\mathcal{S}}\hat{x}=-\hat{\mathcal{R}}^{-1}\big(B^\top\eta+D^\top P\sigma+\bar{D}^\top\Pi\bar{\sigma}+\rho\big).
\end{equation*}
Substituting the above equation into (\ref{eq1}) yields (\ref{value}), i.e.,
\begin{equation*}
\begin{aligned}
	J_\gamma(t,\xi,i;u^*,v^*)&=\mathbb{E}\bigg\{\langle \Pi(t,i)\tilde{\xi},\tilde{\xi}\rangle+\langle P(t,i)\hat{\xi},\hat{\xi}\rangle+2\langle\eta(t,i),\hat{\xi}\rangle+\int_{t}^{T}\Big[\langle P\sigma,\sigma\rangle+\langle\Pi\bar{\sigma},\bar{\sigma}\rangle+2\langle\eta,b\rangle\\
	&\qquad-\langle\hat{\mathcal{R}}^{-1}\big(B^\top\eta+D^\top P\sigma+\bar{D}^\top\Pi\bar{\sigma}+\rho\big),B^\top\eta+D^\top P\sigma+\bar{D}^\top\Pi\bar{\sigma}+\rho\rangle\Big]ds\bigg\},\\
\end{aligned}
\end{equation*}
where $\eta$ is the solution of (\ref{eta1}).

Next, we proof (\ref{eta1}) (or (\ref{eta2})) is equivalent to (\ref{eta}), which is a more compact equation.
\begin{equation*}
\begin{aligned}
	&\quad -\hat{\mathcal{S}}_2^\top\hat{\mathcal{R}}_{22}^{-1}\psi-\big(\hat{\mathcal{S}}_1-\hat{\mathcal{R}}_{12}\hat{\mathcal{R}}_{22}^{-1}\hat{\mathcal{S}}_2\big)^\top\big(\hat{\mathcal{R}}_{11}
     -\hat{\mathcal{R}}_{12}\hat{\mathcal{R}}_{22}^{-1}\hat{\mathcal{R}}_{12}^\top\big)^{-1}\varphi\\
	&=-\hat{\mathcal{S}}_2^\top\hat{\mathcal{R}}_{22}^{-1}\big(B_2^\top\eta+D_2^\top P\sigma+\bar{D}_2^\top\Pi\bar{\sigma}+\rho_2\big)-\big(\hat{\mathcal{S}}_1-\hat{\mathcal{R}}_{12}\hat{\mathcal{R}}_{22}^{-1}\hat{\mathcal{S}}_2\big)^\top
     \big(\hat{\mathcal{R}}_{11}-\hat{\mathcal{R}}_{12}\hat{\mathcal{R}}_{22}^{-1}\hat{\mathcal{R}}_{12}^\top\big)^{-1}\\
	&\quad\times\big[B_1^\top\eta+D_1^\top P\sigma+\bar{D}_1^\top\Pi\bar{\sigma}+\rho_1-\hat{\mathcal{R}}_{12}\hat{\mathcal{R}}_{22}^{-1}\big(B_2^\top\eta+D_2^\top P\sigma+\bar{D}_2^\top\Pi\bar{\sigma}+\rho_2\big)\big]\\
    &=-\big(\hat{\mathcal{S}}_1-\hat{\mathcal{R}}_{12}\hat{\mathcal{R}}_{22}^{-1}\hat{\mathcal{S}}_2\big)^\top\big(\hat{\mathcal{R}}_{11}-\hat{\mathcal{R}}_{12}\hat{\mathcal{R}}_{22}^{-1}\hat{\mathcal{R}}_{12}^\top\big)^{-1}\big(B_1^\top\eta+D_1^\top P\sigma+\bar{D}_1^\top\Pi\bar{\sigma}+\rho_1\big)\\
    &\quad+\big[\big(\hat{\mathcal{S}}_1-\hat{\mathcal{R}}_{12}\hat{\mathcal{R}}_{22}^{-1}\hat{\mathcal{S}}_2\big)^\top\big(\hat{\mathcal{R}}_{11}
     -\hat{\mathcal{R}}_{12}\hat{\mathcal{R}}_{22}^{-1}\hat{\mathcal{R}}_{12}^\top\big)^{-1}\hat{\mathcal{R}}_{12}\hat{\mathcal{R}}_{22}^{-1}-\hat{\mathcal{S}}_2^\top\hat{\mathcal{R}}_{22}^{-1}\big]\\
    &\quad\times\big(B_2^\top\eta+D_2^\top P\sigma+\bar{D}_2^\top\Pi\bar{\sigma}+\rho_2\big)\\
    &=\big(\hat{\mathcal{S}}_1-\hat{\mathcal{R}}_{12}\hat{\mathcal{R}}_{22}^{-1}\hat{\mathcal{S}}_2\big)^\top\big(\hat{\mathcal{R}}_{11}-\hat{\mathcal{R}}_{12}\hat{\mathcal{R}}_{22}^{-1}\hat{\mathcal{R}}_{12}^\top\big)^{-1}\big(\hat{\mathcal{R}}_{11}u^*+\hat{\mathcal{R}}_{12}\hat{v}^*+\hat{\mathcal{S}}_1\hat{x}\big)\\
    &\quad+\big[\hat{\mathcal{S}}_2^\top\hat{\mathcal{R}}_{22}^{-1}-\big(\hat{\mathcal{S}}_1-\hat{\mathcal{R}}_{12}\hat{\mathcal{R}}_{22}^{-1}\hat{\mathcal{S}}_2\big)^\top\big(\hat{\mathcal{R}}_{11}-\hat{\mathcal{R}}_{12}\hat{\mathcal{R}}_{22}^{-1}\hat{\mathcal{R}}_{12}^\top\big)^{-1}\hat{\mathcal{R}}_{12}\hat{\mathcal{R}}_{22}^{-1}\big]\big(\hat{\mathcal{R}}_{12}^\top u^*+\hat{\mathcal{R}}_{22}\hat{v}^*+\hat{\mathcal{S}}_2\hat{x}\big)\\
    &=\big[\hat{\mathcal{S}}_2^\top\hat{\mathcal{R}}_{22}^{-1}\hat{\mathcal{S}}_2+\big(\hat{\mathcal{S}}_1-\hat{\mathcal{R}}_{12}\hat{\mathcal{R}}_{22}^{-1}\hat{\mathcal{S}}_2\big)^\top( \hat{\mathcal{R}}_{11}-\hat{\mathcal{R}}_{12}\hat{\mathcal{R}}_{22}^{-1}\hat{\mathcal{R}}_{12}^\top)^{-1}\big(\hat{\mathcal{S}}_1-\hat{\mathcal{R}}_{12}\hat{\mathcal{R}}_{22}^{-1}\hat{\mathcal{S}}_2\big)\big]\hat{x}\\
    &\quad+\hat{\mathcal{S}}_1^\top u^*+\hat{\mathcal{S}}_2^\top\hat{v}^*\\
    &=-\hat{\mathcal{S}}^\top\hat{\mathcal{R}}^{-1}\big(B^\top\eta+D^\top P\sigma+\bar{D}^\top\Pi\bar{\sigma}+\rho\big).
\end{aligned}
\end{equation*}
With the help of the above equation, we can obtain (\ref{eta}) through simplification. This completes the proof.
\end{proof}

In what follows, we show that optimal feedback control-strategy pair $(u^*(\cdot),v^*(\cdot))\in\mathcal{U}[t,T]\times\mathcal{V}[t,T]$ for the zero-sum LQ stochastic differential game presented in Theorem \ref{Thm-OCS}, constitutes the outcome of the closed-loop saddle point for Problem (SCG).

\begin{mythm}\label{Thm-CL}
Let (H1)-(H3) hold, for any disturbance attenuation level $\gamma>\gamma^*$, assume the Riccati equations (\ref{Pi}) and (\ref{P}) admit a solution $(\Pi(\cdot,\cdot),P(\cdot,\cdot))\in C([t,T]\times\mathcal{S};\mathbb{S}^n)\times C([t,T]\times\mathcal{S};\mathbb{S}^n)$ satisfying Condition (\uppercase\expandafter{\romannumeral1}$\&$\uppercase\expandafter{\romannumeral2}), and $\eta(\cdot,\cdot)\in C([t,T]\times\mathcal{S};\mathbb{R}^n)$ is the solution to the BODE (\ref{eta}). Then Problem (SCG) admits a closed-loop saddle point $(\hat{\Theta}^*(\cdot,\alpha(\cdot)),\tilde{\Theta}^*(\cdot,\alpha(\cdot)),\bar{v}^*(\cdot))\in\mathcal{Q}[t,T]$ with $\hat{\Theta}^*(\cdot,\alpha(\cdot))\equiv(\hat{\Theta}_1^*(\cdot,\alpha(\cdot))^\top,\hat{\Theta}_2^*(\cdot,\alpha(\cdot))^\top)^\top$, $\tilde{\Theta}^*(\cdot,\alpha(\cdot))\equiv(0,\tilde{\Theta}_2^*(\cdot,\alpha(\cdot))^\top)^\top$, $\bar{v}^*(\cdot)\equiv(v_1^*(\cdot)^\top,v_2^*(\cdot)^\top)^\top$, which admits the following representation:
\begin{equation}\label{CLSP}
\left\{\begin{aligned}
	\hat{\Theta}^*(s,\alpha(s))&=-\hat{\mathcal{R}}(s,\alpha(s))^{-1}\hat{\mathcal{S}}(s,\alpha(s)),\\
	\tilde{\Theta}_2^*(s,\alpha(s))&=-\bar{\mathcal{R}}_2(s,\alpha(s))^{-1}\bar{\mathcal{S}}_2(s,\alpha(s)),\\
	\bar{v}^*(s)&=-\hat{\mathcal{R}}(s,\alpha(s))^{-1}\big[B(s,\alpha(s))^\top\eta(s,\alpha(s))+D(s,\alpha(s))^\top P(s,\alpha(s))\sigma(s,\alpha(s))\\
	&\quad+\bar{D}(s,\alpha(s))^\top\Pi(s,\alpha(s))\bar{\sigma}(s,\alpha(s))+\rho(s,\alpha(s))\big],
\end{aligned}\right.
\end{equation}
and the value function is
\begin{equation*}
\begin{aligned}
	V_\gamma(t,\xi,i)&=\mathbb{E}\bigg\{\langle \Pi(t,i)\tilde{\xi},\tilde{\xi}\rangle+\langle P(t,i)\hat{\xi},\hat{\xi}\rangle+2\langle\eta(t,i),\hat{\xi}\rangle+\int_{t}^{T}\Big[\langle P(s,\alpha(s))\sigma(s,\alpha(s)),\sigma(s,\alpha(s))\rangle\\
	&\qquad+\langle\Pi(s,\alpha(s))\bar{\sigma}(s,\alpha(s)),\bar{\sigma}(s,\alpha(s))\rangle+2\langle\eta(s,\alpha(s)),b(s,\alpha(s))\rangle\\
	&\qquad-\langle\hat{\mathcal{R}}(s,\alpha(s))^{-1}\Psi(s,\alpha(s)),\Psi(s,\alpha(s))\rangle\Big]ds\bigg\}.
\end{aligned}
\end{equation*}
\end{mythm}

\begin{proof}
We take any $\bar{u}(\cdot)\equiv(u(\cdot)^\top,v(\cdot)^\top)^\top\in L_{\mathbb{G}}^2(t,T;\mathbb{R}^{m})\times L_{\mathbb{F}}^2(t,T;\mathbb{R}^{n_v})$, and denote $\hat{\bar{u}}(\cdot)\equiv(u(\cdot)^\top,\hat{v}(\cdot)^\top)^\top$, $\tilde{\bar{u}}(\cdot)\equiv(0,\tilde{v}(\cdot)^\top)^\top$. Let $x(\cdot)\equiv x(\cdot;t,\xi,i,\bar{u}(\cdot))$ be the corresponding state process, then let $\hat{x}(\cdot)\equiv\hat{x}(\cdot;t,\hat{\xi},i,\hat{\bar{u}}(\cdot))$, $\tilde{x}(\cdot)\equiv\tilde{x}(\cdot;t,\tilde{\xi},i,\bar{u}(\cdot))$ be the filtering process and the difference, which satisfy the following SDEs, respectively:
\begin{equation*}
\left\{\begin{aligned}
	d\hat{x}(s)=&\big[A(s,\alpha(s))\hat{x}(s)+B(s,\alpha(s))\hat{\bar{u}}(s)+b(s,\alpha(s))\big]ds\\
    &+\big[C(s,\alpha(s))\hat{x}(s)+D(s,\alpha(s))\hat{\bar{u}}(s)+\sigma(s,\alpha(s))\big]dW(s),\\
\hat{x}(t)=&\ \hat{\xi},\quad\alpha(t)=i,
\end{aligned}\right.
\end{equation*}
\begin{equation*}
\left\{\begin{aligned}
	d\tilde{x}(s)=&\big[A(s,\alpha(s))\tilde{x}(s)+B(s,\alpha(s))\tilde{\bar{u}}(s)\big]ds+\big[C(s,\alpha(s))\tilde{x}(s)+D(s,\alpha(s))\tilde{\bar{u}}(s)\big]dW(s)\\
	&+\big[\bar{C}(s,\alpha(s))x(s)+\bar{D}(s,\alpha(s))\bar{u}(s)+\bar{\sigma}(s,\alpha(s))\big]d\overline{W}(s),\\
	\tilde{x}(t)=&\ \tilde{\xi},\quad\alpha(t)=i,
\end{aligned}\right.
\end{equation*}
Moreover, the cost functional is 
\begin{equation*}
\begin{aligned}
	&J_\gamma(t,\xi,i;\bar{u}(\cdot))=\mathbb{E}\bigg\{\langle G(T,\alpha(T))x(T),x(T)\rangle+2\langle g(T,\alpha(T)),x(T)\rangle\\
	&\qquad+\int_t^T\Big[\langle Q(s,\alpha(s))x(s),x(s)\rangle+\langle R_\gamma(s,\alpha(s))\bar{u}(s),\bar{u}(s)\rangle \\
	&\qquad\qquad+2\langle S(s,\alpha(s))x(s),\bar{u}(s)\rangle+2\langle q(s,\alpha(s)),x(s)\rangle+2\langle \rho(s,\alpha(s)),\bar{u}(s)\rangle\Big]ds\bigg\}\\
	&=\mathbb{E}\bigg\{\langle G(T,\alpha(T))\tilde{x}(T),\tilde{x}(T)\rangle+\int_t^T\Big[  \langle Q(s,\alpha(s))\tilde{x}(s),\tilde{x}(s)\rangle\\
	&\qquad\qquad+\langle R_\gamma(s,\alpha(s))\tilde{\bar{u}}(s),\tilde{\bar{u}}(s)\rangle+2\langle S(s,\alpha(s))\tilde{x}(s),\tilde{\bar{u}}(s)\rangle\Big]ds\bigg\}\\
	&\qquad\;+\mathbb{E}\bigg\{\langle G(T,\alpha(T))\hat{x}(T),\hat{x}(T)\rangle+2\langle g(T,\alpha(T)),\hat{x}(T)\rangle\\
	&\qquad\qquad+\int_t^T\Big[  \langle Q(s,\alpha(s))\hat{x}(s),\hat{x}(s)\rangle+\langle R_\gamma(s,\alpha(s))\hat{\bar{u}}(s),\hat{\bar{u}}(s)\rangle \\
	&\qquad\qquad\qquad+2\langle S(s,\alpha(s))\hat{x}(s),\hat{\bar{u}}(s)\rangle+2\langle q(s,\alpha(s)),\hat{x}(s)\rangle+2\langle \rho(s,\alpha(s)),\hat{\bar{u}}(s)\rangle\Big]ds\bigg\}.
\end{aligned}
\end{equation*}
Applying It\^{o}'s formula to $s\mapsto\langle\Pi(s,\alpha(s))\tilde{x}(s),\tilde{x}(s)\rangle$, $s\mapsto\langle P(s,\alpha(s))\hat{x}(s),\hat{x}(s)\rangle$ and $s\mapsto\langle \eta(s,\alpha(s)),\hat{x}(s)\rangle$, respectively, Integrating from $t$ to $T$, taking expectation, and substituting them into the cost functional $J_\gamma(t,\xi,i;\bar{u}(\cdot))$, by the completion-of-squares method, we have
\begin{equation*}
\begin{aligned}
	&J_\gamma(t,\xi,i;\bar{u}(\cdot))=\mathbb{E}\bigg\{\langle\Pi(t,i)\tilde{\xi},\tilde{\xi}\rangle+\langle P(t,i)\hat{\xi},\hat{\xi}\rangle+2\langle\eta(t,i),\hat{\xi}\rangle\\
    &\qquad +\int_t^T\Big[\langle\bar{\mathcal{S}}_2^\top\bar{\mathcal{R}}_2^{-1}\bar{\mathcal{S}}_2\tilde{x},\tilde{x}\rangle+\langle\bar{\mathcal{R}}_2\tilde{v},\tilde{v}\rangle
     +2\langle\bar{\mathcal{S}}_2\tilde{x},\tilde{v}\rangle+\langle\hat{\mathcal{S}}^\top\hat{\mathcal{R}}^{-1}\hat{\mathcal{S}}\hat{x},\hat{x}\rangle\\
	&\qquad\qquad +\langle\hat{\mathcal{R}}\hat{\bar{u}},\hat{\bar{u}}\rangle+2\langle\hat{\mathcal{S}}\hat{x},\hat{\bar{u}}\rangle
     +2\langle\hat{\mathcal{S}}^\top\hat{\mathcal{R}}^{-1}\big(B^\top\eta+D^\top P\sigma+\bar{D}^\top\Pi\bar{\sigma}+\rho\big),\hat{x}\rangle\\
	&\qquad\qquad +2\langle B^\top\eta+D^\top P\sigma+\bar{D}^\top\Pi\bar{\sigma}+\rho,\hat{\bar{u}}\rangle+\langle P\sigma,\sigma\rangle+\langle\Pi\bar{\sigma},\bar{\sigma}\rangle+2\langle\eta,b\rangle\Big]ds\bigg\}\\
	&=\mathbb{E}\bigg\{\langle\Pi(t,i)\tilde{\xi},\tilde{\xi}\rangle+\langle P(t,i)\hat{\xi},\hat{\xi}\rangle+2\langle\eta(t,i),\hat{\xi}\rangle+\int_t^T\Big[\langle P\sigma,\sigma\rangle+\langle\Pi\bar{\sigma},\bar{\sigma}\rangle+2\langle\eta,b\rangle\\
	&\qquad-\langle\hat{\mathcal{R}}^{-1}\big(B^\top\eta+D^\top P\sigma+\bar{D}^\top\Pi\bar{\sigma}+\rho\big),B^\top\eta+D^\top P\sigma+\bar{D}^\top\Pi\bar{\sigma}+\rho\rangle\\
	&\qquad+\big\langle\hat{\mathcal{R}}\big[\hat{\bar{u}}+\hat{\mathcal{R}}^{-1}\hat{\mathcal{S}}\hat{x}+\hat{\mathcal{R}}^{-1}\big(B^\top\eta+D^\top P\sigma+\bar{D}^\top\Pi\bar{\sigma}+\rho\big)\big],\\
	&\qquad\qquad\hat{\bar{u}}+\hat{\mathcal{R}}^{-1}\hat{\mathcal{S}}\hat{x}+\hat{\mathcal{R}}^{-1}\big(B^\top\eta+D^\top P\sigma+\bar{D}^\top\Pi\bar{\sigma}+\rho\big)\big\rangle\\
	&\qquad+\langle\bar{\mathcal{R}}_2\big(\tilde{v}+\bar{\mathcal{R}}_2^{-1}\bar{\mathcal{S}}_2\tilde{x}\big),\tilde{v}+\bar{\mathcal{R}}_2^{-1}\bar{\mathcal{S}}_2\tilde{x}\rangle\Big]ds\bigg\}\\
	&=\mathbb{E}\bigg\{\langle\Pi(t,i)\tilde{\xi},\tilde{\xi}\rangle+\langle P(t,i)\hat{\xi},\hat{\xi}\rangle+2\langle\eta(t,i),\hat{\xi}\rangle+\int_t^T\Big[\langle P\sigma,\sigma\rangle+\langle\Pi\bar{\sigma},\bar{\sigma}\rangle+2\langle\eta,b\rangle\\
	&\qquad-\langle\hat{\mathcal{R}}^{-1}\big(B^\top\eta+D^\top P\sigma+\bar{D}^\top\Pi\bar{\sigma}+\rho\big),B^\top\eta+D^\top P\sigma+\bar{D}^\top\Pi\bar{\sigma}+\rho\rangle\\
	&\qquad+\langle\bar{\mathcal{R}}_2\big(\tilde{v}-\tilde{\Theta}_2^*\tilde{x}\big),\tilde{v}-\tilde{\Theta}_2^*\tilde{x}\rangle+\langle\hat{\mathcal{R}}\big(\hat{\bar{u}}-\hat{\Theta}^*\hat{x}-\bar{v}^*\big),\hat{\bar{u}}
     -\hat{\Theta}^*\hat{x}-\bar{v}^*\rangle\Big]ds\bigg\}\\
	&=J_\gamma(t,\xi,i;\hat{\Theta}^*(\cdot,\alpha(\cdot))\hat{x}(\cdot)+\tilde{\Theta}^*(\cdot,\alpha(\cdot))\tilde{x}(\cdot)+\bar{v}^*(\cdot))\\
	&\quad+\mathbb{E}\int_t^T\Big[\langle\bar{\mathcal{R}}_2\big(\tilde{v}-\tilde{\Theta}_2^*\tilde{x}\big),\tilde{v}-\tilde{\Theta}_2^*\tilde{x}\rangle
     +\langle\hat{\mathcal{R}}\big(\hat{\bar{u}}-\hat{\Theta}^*\hat{x}-\bar{v}^*\big),\hat{\bar{u}}-\hat{\Theta}^*\hat{x}-\bar{v}^*\rangle\Big]ds.
\end{aligned}
\end{equation*}
Consequently,
\begin{equation*}
\begin{aligned}
	&J_\gamma(t,\xi,i;\hat{\Theta}^*_1\hat{x}+v_1,\hat{\Theta}^*_2\hat{x}+\tilde{\Theta}_2^*\tilde{x}+v_2^*)\\
	&=J_\gamma(t,\xi,i;\hat{\Theta}^*\hat{x}+\tilde{\Theta}^*\tilde{x}+\bar{v}^*)+\mathbb{E}\int_t^T\langle\hat{\mathcal{R}}_{11}\big(v_1-v_1^*\big),v_1-v_1^*\rangle ds.\\
\end{aligned}
\end{equation*}
Hence,
\begin{equation*}
	J_\gamma(t,\xi,i;\hat{\Theta}^*\hat{x}+\tilde{\Theta}^*\tilde{x}+\bar{v}^*)\leq J_\gamma(t,\xi,i;\hat{\Theta}^*_1\hat{x}+v_1,\hat{\Theta}^*_2\hat{x}+\tilde{\Theta}_2^*\tilde{x}+v_2^*),\quad\forall v_1\in L_{\mathbb{G}}^2(t,T;\mathbb{R}^{m}),
\end{equation*}
if and only if
\begin{equation*}
	\hat{\mathcal{R}}_{11}(s,\alpha(s))\geq0,\quad a.e.\; s\in[t,T].
\end{equation*}
Similarly,
\begin{equation*}
\begin{aligned}
	&J_\gamma(t,\xi,i;\hat{\Theta}^*_1\hat{x}+v_1^*,\hat{\Theta}^*_2\hat{x}+\tilde{\Theta}_2^*\tilde{x}+v_2)\\
	&=J_\gamma(t,\xi,i;\hat{\Theta}^*\hat{x}+\tilde{\Theta}^*\tilde{x}+\bar{v}^*)+\mathbb{E}\int_t^T\Big[\langle\bar{\mathcal{R}}_2\tilde{v}_2,\tilde{v}_2\rangle+\langle\hat{\mathcal{R}}_{22}\big(\hat{v}_2-v_2^*\big),\hat{v}_2-v_2^*\rangle\Big]ds.\\
\end{aligned}
\end{equation*}
Hence,
\begin{equation*}
	J_\gamma(t,\xi,i;\hat{\Theta}^*\hat{x}+\tilde{\Theta}^*\tilde{x}+\bar{v}^*)\geq J_\gamma(t,\xi,i;\hat{\Theta}^*_1\hat{x}+v_1^*,\hat{\Theta}^*_2\hat{x}+\tilde{\Theta}_2^*\tilde{x}+v_2),\quad\forall v_2\in L_{\mathbb{F}}^2(t,T;\mathbb{R}^{n_v}),
\end{equation*}
if and only if
\begin{equation*}
	\hat{\mathcal{R}}_{22}(s,\alpha(s))\leq0,\quad a.e.\; s\in[t,T].
\end{equation*}
From Remark \ref{Remark1}, we get that under Condition (\uppercase\expandafter{\romannumeral1}$\&$\uppercase\expandafter{\romannumeral2}), $(\hat{\Theta}^*,\tilde{\Theta}^*,\bar{v}^*)\in\mathcal{Q}[t,T]$ is a closed-loop saddle point of Problem (SCG).
\end{proof}

Theorem \ref{Thm-CL} provides the sufficient conditions for the existence of the closed-loop saddle point of Problem (SCG). Next, the following result is concerned with $H_\infty$ performance criterion of the outcome of the above closed-loop saddle point, which indicates that the outcome of the closed-loop saddle point $(u^*,v^*)\in\mathcal{U}[t,T]\times\mathcal{V}[t,T]$ constitutes the robust $H_\infty$ control and the worst-case disturbance for Problem (R-SCG).

From Lemma \ref{lem2} (\romannumeral1) and
\begin{equation*}
\begin{aligned}
    \hat{\Theta}_1^*(s,i)=-\big(\hat{\mathcal{R}}_{11}(s,i)-\hat{\mathcal{R}}_{12}(s,i)\hat{\mathcal{R}}_{22}(s,i)^{-1}\hat{\mathcal{R}}_{12}(s,i)^\top\big)^{-1}
    \big(\hat{\mathcal{S}}_1(s,i)-\hat{\mathcal{R}}_{12}(s,i)\hat{\mathcal{R}}_{22}(s,i)^{-1}\hat{\mathcal{S}}_2(s,i)\big),\\	
\end{aligned}
\end{equation*}
it follows that 
\begin{equation*}
\begin{aligned}
	-\hat{\mathcal{S}}^\top\hat{\mathcal{R}}^{-1}\hat{\mathcal{S}}=&-\hat{\mathcal{S}}_2^\top\hat{\mathcal{R}}_{22}^{-1}\hat{\mathcal{S}}_2
    -\big(\hat{\mathcal{S}}_1-\hat{\mathcal{R}}_{12}\hat{\mathcal{R}}_{22}^{-1}\hat{\mathcal{S}}_2\big)^\top\big(\hat{\mathcal{R}}_{11}-\hat{\mathcal{R}}_{12}\hat{\mathcal{R}}_{22}^{-1}
    \hat{\mathcal{R}}_{12}^\top\big)^{-1}\big(	\hat{\mathcal{S}}_1-\hat{\mathcal{R}}_{12}\hat{\mathcal{R}}_{22}^{-1}\hat{\mathcal{S}}_2\big)\\
	=&-\hat{\mathcal{S}}_2^\top\hat{\mathcal{R}}_{22}^{-1}\hat{\mathcal{S}}_2+\big(\hat{\mathcal{S}}_1-\hat{\mathcal{R}}_{12}\hat{\mathcal{R}}_{22}^{-1}\hat{\mathcal{S}}_2\big)^\top\hat{\Theta}_1^*
    +\hat{\Theta}_1^{*^\top}\big(\hat{\mathcal{S}}_1-\hat{\mathcal{R}}_{12}\hat{\mathcal{R}}_{22}^{-1}\hat{\mathcal{S}}_2\big)\\
	&+\hat{\Theta}_1^{*^\top}\big(\hat{\mathcal{R}}_{11}-\hat{\mathcal{R}}_{12}\hat{\mathcal{R}}_{22}^{-1}\hat{\mathcal{R}}_{12}^\top\big)\hat{\Theta}_1^*\\
	=&\ \hat{\mathcal{S}}_1^\top\hat{\Theta}_1^*+\hat{\Theta}_1^{*^\top}\hat{\mathcal{S}}_1+\hat{\Theta}_1^{*^\top}\hat{\mathcal{R}}_{11}\hat{\Theta}_1^*-\big(\hat{\mathcal{S}}_2
    +\hat{\mathcal{R}}_{12}^\top\hat{\Theta}_1^*\big)^\top\hat{\mathcal{R}}_{22}^{-1}\big(\hat{\mathcal{S}}_2+\hat{\mathcal{R}}_{12}^\top\hat{\Theta}_1^*\big)\\
	=&\ \big(PB_1+C^\top PD_1+\bar{C}^\top\Pi\bar{D}_1+S_1^\top\big)\hat{\Theta}_1^*+\hat{\Theta}_1^{*^\top}\big(B_1^\top P+D_1^\top PC+\bar{D}_1^\top\Pi\bar{C}+S_1\big)\\
	&+\hat{\Theta}_1^{*^\top}\big(R_1+D_1^\top PD_1+\bar{D}_1^\top\Pi\bar{D}_1\big)\hat{\Theta}_1^*-\mathbb{S}_2^\top\hat{\mathcal{R}}_{22}^{-1}\mathbb{S}_2,
\end{aligned}
\end{equation*}
where
\begin{equation*}
\begin{aligned}
	\mathbb{S}_2(s,i)&:=\hat{\mathcal{S}}_2(s,i)+\hat{\mathcal{R}}_{12}(s,i)^\top\hat{\Theta}_1^*(s,i)\\
	&=B_2(s,i)^\top P(s,i)+D_2(s,i)^\top P(s,i)\big(C(s,i)+D_1(s,i)\hat{\Theta}_1^*(s,i)\big)\\
    &\quad+\bar{D}_2(s,i)^\top\Pi(s,i)\big(\bar{C}(s,i)+\bar{D}_1(s,i)\hat{\Theta}_1^*(s,i)\big)+S_2(s,i).	
\end{aligned}
\end{equation*}

Therefore, through the above substitution, the Riccati equation (\ref{P}) can be rearranged as:
\begin{equation}\label{P'}
\left\{\begin{aligned}
	\dot{P}&(s,i)+P(s,i)\big(A(s,i)+B_1(s,i)\hat{\Theta}_1^*(s,i)\big)+\big(A(s,i)+B_1(s,i)\hat{\Theta}_1^*(s,i)\big)^\top P(s,i)\\
	&+\big(C(s,i)+D_1(s,i)\hat{\Theta}_1^*(s,i)\big)^\top P(s,i)\big(C(s,i)+D_1(s,i)\hat{\Theta}_1^*(s,i)\big)\\
	&+\big(\bar{C}(s,i)+\bar{D}_1(s,i)\hat{\Theta}_1^*(s,i)\big)^\top\Pi(s,i)\big(\bar{C}(s,i)+\bar{D}_1(s,i)\hat{\Theta}_1^*(s,i)\big)\\
	&+\hat{\Theta}_1^*(s,i)^\top R_1(s,i)\hat{\Theta}_1^*(s,i)+\hat{\Theta}_1^*(s,i)^\top S_1(s,i)+S_1(s,i)^\top\hat{\Theta}_1^*(s,i)+Q(s,i)\\
	&-\mathbb{S}_2(s,i)^\top\hat{\mathcal{R}}_{22}(s,i)^{-1}\mathbb{S}_2(s,i)+\sum_{j=1}^{D}\lambda_{ij}P(s,j)=0,\quad a.e.\; s\in[t,T],\\
	P&(T,i)=G(T,i),\quad i\in\mathcal{S}.
\end{aligned}\right.
\end{equation}

\begin{mythm}\label{Thm-R}
Let (H1)-(H3) hold, for any given disturbance attenuation level $\gamma>\gamma^*$, assume that the Riccati equations (\ref{Pi}) and (\ref{P'}) admit a solution $(\Pi(\cdot,\cdot),P(\cdot,\cdot))\in C([t,T]\times\mathcal{S};\mathbb{S}^n)\times C([t,T]\times\mathcal{S};\mathbb{S}^n)$ satisfying Condition (\uppercase\expandafter{\romannumeral1}$\&$\uppercase\expandafter{\romannumeral2}), then one of the outcome of the closed-loop saddle point $(\hat{\Theta}^*(\cdot,\alpha(\cdot)),\tilde{\Theta}^*(\cdot,\alpha(\cdot)),\bar{v}^*(\cdot))\in\mathcal{Q}[t,T]$ for Problem (SCG) 
\begin{equation*}
	u^*(s)=\hat{\Theta}_1^*(s,\alpha(s))\hat{x}(s)+v_1^*(s),\quad s\in[t,T],
\end{equation*}
satisfies $H_\infty$-performance, i.e., $\Vert\mathcal{L}_{u^*}\Vert<\gamma$.
\end{mythm}

\begin{proof}
Substitute $u^*=\hat{\Theta}_1^*\hat{x}^0$ into the homogeneous system corresponding to (\ref{state eq}), we have
\begin{equation*}
\left\{\begin{aligned}
	dx^0(s)=&\big[A(s,\alpha(s))x^0(s)+B_1(s,\alpha(s))\hat{\Theta}_1^*(s,\alpha(s))\hat{x}^0(s)+B_2(s,\alpha(s))v(s)\big]ds\\
    &+\big[C(s,\alpha(s))x^0(s)+D_1(s,\alpha(s))\hat{\Theta}_1^*(s,\alpha(s))\hat{x}^0(s)+D_2(s,\alpha(s))v(s)\big]dW(s)\\
    &+\big[\bar{C}(s,\alpha(s))x^0(s)+\bar{D}_1(s,\alpha(s))\hat{\Theta}_1^*(s,\alpha(s))\hat{x}^0(s)+\bar{D}_2(s,\alpha(s))v(s)\big]d\overline{W}(s),\\
    x^0(t)=&\ \xi,\qquad\alpha(t)=i,
\end{aligned}\right.
\end{equation*}
and the corresponding filtering process and the difference satisfy the following SDEs, respectively,
\begin{equation*}
\left\{\begin{aligned}
	d\hat{x}^0(s)=&\big[\big(A(s,\alpha(s))+B_1(s,\alpha(s))\hat{\Theta}_1^*(s,\alpha(s))\big)\hat{x}^0(s)+B_2(s,\alpha(s))\hat{v}(s)\big]ds\\
	&+\big[\big(C(s,\alpha(s))+D_1(s,\alpha(s))\hat{\Theta}_1^*(s,\alpha(s))\big)\hat{x}^0(s)+D_2(s,\alpha(s))\hat{v}(s)\big]dW(s),\\
	\hat{x}^0(t)=&\ \hat{\xi},\quad\alpha(t)=i,
\end{aligned}\right.
\end{equation*}
\begin{equation*}
\left\{\begin{aligned}
	d\tilde{x}^0(s)=&\big[A(s,\alpha(s))\tilde{x}^0(s)+B_2(s,\alpha(s))\tilde{v}(s)\big]ds+\big[C(s,\alpha(s))\tilde{x}^0(s)+D_2(s,\alpha(s))\tilde{v}(s)\big]dW(s)\\
	&+\big[\bar{C}(s,\alpha(s))x^0(s)+\bar{D}_1(s,\alpha(s))\hat{\Theta}_1^*(s,\alpha(s))\hat{x}^0(s)+\bar{D}_2(s,\alpha(s))v(s)\big]d\overline{W}(s),\\
	\tilde{x}^0(t)=&\ \tilde{\xi},\quad\alpha(t)=i.
\end{aligned}\right.
\end{equation*}
The cost functional is given by
\begin{equation*}
\begin{aligned}
	&J_\gamma^0(t,\xi,i;u^*(\cdot),v(\cdot))=\mathbb{E}\bigg\{\langle G(T,\alpha(T))\tilde{x}^0(T),\tilde{x}^0(T)\rangle+\int_t^T\Big[  \langle Q(s,\alpha(s))\tilde{x}^0(s),\tilde{x}^0(s)\rangle\\
    &\qquad+2\langle S_2(s,\alpha(s))\tilde{x}^0(s),\tilde{v}(s)\rangle+\langle (R_2(s,\alpha(s))-\gamma^2I)\tilde{v}(s),\tilde{v}(s)\rangle\Big]ds\bigg\}\\
    &+\mathbb{E}\bigg\{\langle G(T,\alpha(T))\hat{x}^0(T),\hat{x}^0(T)\rangle+\int_t^T\Big[\langle\big(Q+\hat{\Theta}_1^{*^\top}R_1\hat{\Theta}_1^*+\hat{\Theta}_1^{*^\top}S_1+S_1^\top\hat{\Theta}_1^*\big)(s,\alpha(s))\hat{x}^0(s),\hat{x}^0(s)\rangle\\
    &\qquad+2\langle S_2(s,\alpha(s))\hat{x}^0(s),\hat{v}(s)\rangle+\langle (R_2(s,\alpha(s))-\gamma^2I)\hat{v}(s),\hat{v}(s)\rangle\Big]ds\bigg\}.
\end{aligned}
\end{equation*}
Similarly, by applying It\^{o}'s formula to $s\mapsto\langle\Pi(s,\alpha(s))\tilde{x}^0(s),\tilde{x}^0(s)\rangle$ and $s\mapsto\langle P(s,\alpha(s))\hat{x}^0(s),\hat{x}^0(s)\rangle$, we have
\begin{equation*}
\begin{aligned}
	&-J_\gamma^0(t,\xi,i;u^*(\cdot),v(\cdot))=\mathbb{E}\bigg\{-\langle\Pi(t,i)\tilde{\xi},\tilde{\xi}\rangle-\langle P(t,i)\hat{\xi},\hat{\xi}\rangle\\
	&\quad+\int_t^T\Big[ -\langle\bar{\mathcal{R}}_2\big(\tilde{v}+\bar{\mathcal{R}}_2^{-1}\bar{\mathcal{S}}_2\tilde{x}^0\big),\tilde{v}+\bar{\mathcal{R}}_2^{-1}\bar{\mathcal{S}}_2\tilde{x}^0\rangle-\langle\hat{\mathcal{R}}_{22}\big(\hat{v}+\hat{\mathcal{R}}_{22}^{-1}\mathbb{S}_2\hat{x}^0\big),\hat{v}+\hat{\mathcal{R}}_{22}^{-1}\mathbb{S}_2\hat{x}^0\rangle\Big]ds\bigg\}.\\
\end{aligned}
\end{equation*}
Notice that matrices $\bar{\mathcal{R}}_2$ and $\hat{\mathcal{R}}_{22}$ are uniformly negative definite, we obtain
\begin{equation}\label{ineq-end1}
\begin{aligned}
	&-J_\gamma^0(0,0,i;u^*(\cdot),v(\cdot))\\
	&=\mathbb{E}\int_0^T\Big[ -\langle\bar{\mathcal{R}}_2\big(\tilde{v}+\bar{\mathcal{R}}_2^{-1}\bar{\mathcal{S}}_2\tilde{x}^0\big),\tilde{v}+\bar{\mathcal{R}}_2^{-1}\bar{\mathcal{S}}_2\tilde{x}^0\rangle-\langle\hat{\mathcal{R}}_{22}\big(\hat{v}
    +\hat{\mathcal{R}}_{22}^{-1}\mathbb{S}_2\hat{x}^0\big),\hat{v}+\hat{\mathcal{R}}_{22}^{-1}\mathbb{S}_2\hat{x}^0\rangle\Big]ds\\
	&\geq\delta\mathbb{E}\int_0^T\Big[ \vert \tilde{v}+\bar{\mathcal{R}}_2^{-1}\bar{\mathcal{S}}_2\tilde{x}^0\vert^2+\vert\hat{v}+\hat{\mathcal{R}}_{22}^{-1}\mathbb{S}_2\hat{x}^0\vert^2\Big]ds\\
	&=\delta\mathbb{E}\int_0^T\Big[ \vert v- \hat{v}+\bar{\mathcal{R}}_2^{-1}\bar{\mathcal{S}}_2\tilde{x}^0\vert^2+\vert\hat{v}+\hat{\mathcal{R}}_{22}^{-1}\mathbb{S}_2\hat{x}^0\vert^2\Big]ds\\
	&=\delta\mathbb{E}\int_0^T\Big[ \vert v+\bar{\mathcal{R}}_2^{-1}\bar{\mathcal{S}}_2\tilde{x}^0\vert^2+\vert\hat{v}\vert^2-2\langle v
    +\bar{\mathcal{R}}_2^{-1}\bar{\mathcal{S}}_2\tilde{x}^0,\hat{v}\rangle +\vert\hat{v}+\hat{\mathcal{R}}_{22}^{-1}\mathbb{S}_2\hat{x}^0\vert^2\Big]ds\\
	&\geq\delta\mathbb{E}\int_0^T\Big[ \frac{\beta}{1+\beta}\vert v+\bar{\mathcal{R}}_2^{-1}\bar{\mathcal{S}}_2\tilde{x}^0\vert^2-\beta\vert\hat{v}\vert^2 +\vert\hat{v}+\hat{\mathcal{R}}_{22}^{-1}\mathbb{S}_2\hat{x}^0\vert^2\Big]ds.
\end{aligned}
\end{equation} 
Next, define a bounded linear operator $\Gamma_1:L_{\mathbb{F}}^2(0,T;\mathbb{R}^{n_v})\to L_{\mathbb{F}}^2(0,T;\mathbb{R}^{n_v})$,
\begin{equation*}
	(\Gamma_1v)(\cdot):=v(\cdot)+\bar{\mathcal{R}}_2(\cdot,\alpha(\cdot))^{-1}\bar{\mathcal{S}}_2(\cdot,\alpha(\cdot))\tilde{x}^0(\cdot).\\
\end{equation*}
Then $\Gamma_1$ is bijective and its inverse $\Gamma_1^{-1}$ is given by
\begin{equation*}
	(\Gamma_1^{-1}v)(\cdot):=v(\cdot)-\bar{\mathcal{R}}_2(\cdot,\alpha(\cdot))^{-1}\bar{\mathcal{S}}_2(\cdot,\alpha(\cdot))x^v(\cdot)+\bar{\mathcal{R}}_2(\cdot,\alpha(\cdot))^{-1}\bar{\mathcal{S}}_2(\cdot,\alpha(\cdot))\hat{x}^v(\cdot),\\
\end{equation*}
where $x^v(\cdot)$ and $\hat{x}^v(\cdot)$ are the solutions of
\begin{equation*}
\left\{\begin{aligned}
	dx^v(s)=&\big[\big(A-B_2\bar{\mathcal{R}}_2^{-1}\bar{\mathcal{S}}_2\big)x^v+\big(B_1\hat{\Theta}_1^*+B_2\bar{\mathcal{R}}_2^{-1}\bar{\mathcal{S}}_2\big)\hat{x}^v+B_2v\big]ds\\
	&+\big[\big(C-D_2\bar{\mathcal{R}}_2^{-1}\bar{\mathcal{S}}_2\big)x^v+\big(D_1\hat{\Theta}_1^*+D_2\bar{\mathcal{R}}_2^{-1}\bar{\mathcal{S}}_2\big)\hat{x}^v+D_2v\big]dW(s)\\
	&+\big[\big(\bar{C}-\bar{D}_2\bar{\mathcal{R}}_2^{-1}\bar{\mathcal{S}}_2\big)x^v+\big(\bar{D}_1\hat{\Theta}_1^*+\bar{D}_2\bar{\mathcal{R}}_2^{-1}\bar{\mathcal{S}}_2\big)\hat{x}^v+\bar{D}_2v\big]d\overline{W}(s),\\
	x^v(0)=&\ 0,\quad\alpha(0)=i,
\end{aligned}\right.
\end{equation*}
\begin{equation*}
\left\{\begin{aligned}
	d\hat{x}^v(s)=&\big[\big(A+B_1\hat{\Theta}_1^*+B_2\bar{\mathcal{R}}_2^{-1}\bar{\mathcal{S}}_2\big)\hat{x}^v-B_2\bar{\mathcal{R}}_2^{-1}\bar{\mathcal{S}}_2x^v+B_2\bar{\mathcal{R}}_2^{-1}\bar{\mathcal{S}}_2\tilde{x}^v+B_2\hat{v}\big]ds\\
	&+\big[\big(C+D_1\hat{\Theta}_1^*+D_2\bar{\mathcal{R}}_2^{-1}\bar{\mathcal{S}}_2\big)\hat{x}^v-D_2\bar{\mathcal{R}}_2^{-1}\bar{\mathcal{S}}_2x^v+D_2\bar{\mathcal{R}}_2^{-1}\bar{\mathcal{S}}_2\tilde{x}^v+D_2\hat{v}\big]dW(s),\\
	\hat{x}^v(0)=&\ 0,\quad\alpha(0)=i,
\end{aligned}\right.
\end{equation*}
and
\begin{equation*}
\left\{\begin{aligned}
	d\tilde{x}^v(s)=&\big[\big(A-B_2\bar{\mathcal{R}}_2^{-1}\bar{\mathcal{S}}_2\big)\tilde{x}^v+B_2\tilde{v}\big]ds+\big[\big(C-D_2\bar{\mathcal{R}}_2^{-1}\bar{\mathcal{S}}_2\big)\tilde{x}^v+D_2\tilde{v}\big]dW(s)\\
	&+\big[\big(\bar{C}-\bar{D}_2\bar{\mathcal{R}}_2^{-1}\bar{\mathcal{S}}_2\big)\tilde{x}^v+\bar{D}_2\tilde{v}\big]d\overline{W}(s),\\
	\tilde{x}^v(0)=&\ 0,\quad\alpha(0)=i.
\end{aligned}\right.
\end{equation*}
By the bounded inverse theorem, $\Gamma_1^{-1}$ is bounded with $\Vert\Gamma_1^{-1}\Vert>0$. Thus,
\begin{equation}\label{ineq-end2}
\begin{aligned}
	&\mathbb{E}\int_0^T\vert v(s)\vert^2ds=\mathbb{E}\int_0^T\vert (\Gamma_1^{-1}\Gamma_1v)(s)\vert^2ds\\
	&\leq\Vert\Gamma_1^{-1}\Vert\mathbb{E}\int_0^T\vert (\Gamma_1v)(s)\vert^2ds=\Vert\Gamma_1^{-1}\Vert\mathbb{E}\int_0^T\vert v+\bar{\mathcal{R}}_2^{-1}\bar{\mathcal{S}}_2\tilde{x}^0\vert^2ds.
\end{aligned}
\end{equation}
Similarly, define another bounded linear operator $\Gamma_2:L_{\mathbb{G}}^2(0,T;\mathbb{R}^{n_v})\to L_{\mathbb{G}}^2(0,T;\mathbb{R}^{n_v})$,
\begin{equation*}
	(\Gamma_2\hat{v})(\cdot):=\hat{v}(\cdot)+\hat{\mathcal{R}}_{22}(\cdot,\alpha(\cdot))^{-1}\mathbb{S}_2(\cdot,\alpha(\cdot))\hat{x}^0(\cdot).
\end{equation*}
Obviously, $\Gamma_2$ is bijective and its inverse $\Gamma_2^{-1}$ is denoted by
\begin{equation*}
	(\Gamma_2^{-1}\hat{v})(\cdot):=\hat{v}(\cdot)-\hat{\mathcal{R}}_{22}(\cdot,\alpha(\cdot))^{-1}\mathbb{S}_2(\cdot,\alpha(\cdot))\check{x}^v(\cdot),
\end{equation*}
where $\check{x}^v(\cdot)$ satisfies
\begin{equation*}
\left\{\begin{aligned}
	d\check{x}^v(s)=&\big[\big(A+B_1\hat{\Theta}_1^*-B_2\hat{\mathcal{R}}_{22}^{-1}\mathbb{S}_2\big)\check{x}^v+B_2\hat{v}\big]ds+\big[\big(C+D_1\hat{\Theta}_1^*-D_2\hat{\mathcal{R}}_{22}^{-1}\mathbb{S}_2\big)\check{x}^v+D_2\hat{v}\big]dW(s),\\
	\check{x}^v(0)=&\ 0,\quad\alpha(0)=i.
\end{aligned}\right.
\end{equation*}
By the bounded inverse theorem, $\Gamma_2^{-1}$ is bounded with $\Vert\Gamma_2^{-1}\Vert>0$, and
\begin{equation}\label{ineq-end3}
\begin{aligned}
	\mathbb{E}\int_{0}^{T}\vert \hat{v}(s)\vert^2ds\leq\Vert\Gamma_2^{-1}\Vert\mathbb{E}\int_{0}^{T}\vert (\Gamma_2\hat{v})(s)\vert^2ds=\Vert\Gamma_2^{-1}\Vert\mathbb{E}\int_{0}^{T}\vert\hat{v}+\hat{\mathcal{R}}_{22}^{-1}\mathbb{S}_2\hat{x}^0\vert^2ds.\\
\end{aligned}
\end{equation}
With the help of (\ref{ineq-end2}) and (\ref{ineq-end3}), (\ref{ineq-end1}) can be bounded as follows:
\begin{equation*}
\begin{aligned}
	-J_\gamma^0(0,0,i;u^*(\cdot),v(\cdot))&\geq\delta\mathbb{E}\int_0^T\bigg[\frac{\beta}{\Vert\Gamma_1^{-1}\Vert(1+\beta)}\vert v(s)\vert^2-\beta\vert\hat{v}(s)\vert^2 +\frac{1}{\Vert\Gamma_2^{-1}\Vert}\vert\hat{v}(s)\vert^2\bigg]ds\\
	&=\frac{\delta\beta}{\Vert\Gamma_1^{-1}\Vert(1+\beta)}\mathbb{E}\int_0^T \vert v(s)\vert^2ds.
\end{aligned}
\end{equation*} 
Set $\beta:=\frac{1}{\Vert\Gamma_2^{-1}\Vert}$, and the last equality holds. Since for any $v$,
\begin{equation*}
\begin{aligned}
	J^0(0,0,i;u^*(\cdot),v(\cdot))&=J_\gamma^0(0,0,i;u^*(\cdot),v(\cdot))+\gamma^2\mathbb{E}\int_0^T \vert v(s)\vert^2ds\\	
	&\leq\bigg(\gamma^2-\frac{\delta\beta}{\Vert\Gamma_1^{-1}\Vert(1+\beta)}\bigg)\mathbb{E}\int_0^T \vert v(s)\vert^2ds,
\end{aligned}
\end{equation*}
by the definition of $\Vert\mathcal{L}_{u^*}\Vert$, we have 
\begin{equation*}
	\Vert\mathcal{L}_{u^*}\Vert^2\equiv\underset{\substack{v\neq0\\v(\cdot)\in L_{\mathbb{F}}^2(0,T;\mathbb{R}^{n_v})}}
    \sup\frac{J^0(0,0,i;u^*(\cdot),v(\cdot))}{\mathbb{E}\int_0^T\vert v(s)\vert^2ds}\leq\gamma^2-\frac{\delta\beta}{\Vert\Gamma_1^{-1}\Vert(1+\beta)}<\gamma^2.
\end{equation*}
The proof is completed.
\end{proof}

\section{A numerical example}

In this section, we would like to present a numerical example to better demonstrate the efficacy of the proposed $H_\infty$ control strategy. Assume that $x(\cdot)$, the state of the  stock market, satisfies a SDE, $u(\cdot)$ denotes the investment strategy of an institutional investor in the market, and $v(\cdot)$ is the external unknown disturbance faced by the institutional investor. We use bull market and bear market to describe market quotation, and they correspond to two states of Markov chain $\alpha$, respectively. Suppose that $\alpha$ takes values in $\mathcal{S}=\{1,2\}$ with following generator
\begin{equation*}
	\Lambda=\begin{pmatrix}
		-1 & 1 \\
		2 & -2 \\
	\end{pmatrix}.
\end{equation*}
The horizon length for the simulation is selected as $T=3.5$. For the sake of simplifying computational complexity, we make the assumption that the system is homogeneous. Morover, we assume that the market is bear market when $\alpha=1$ and it is bull market when $\alpha=2$, which correspond to the different system coefficients shown in the following tables.

\vspace{-2mm}

\begin{table}[H]
\centering
\caption{Simulation parameters corresponding to $\alpha=1$}
\begin{tabular}{c c c c c c c c c c c c c c c c}
	\hline
	$A$ & $B_1$ & $B_2$ & $C$ & $D_1$ & $D_2$ & $\bar{C}$ & $\bar{D}_1$ & $\bar{D}_2$ & $Q$ & $R_1$ & $R_2$ & $S_1$ & $S_2$ & $G$ & $\xi$\\
	\hline
	0.1 & 0.3 & -0.2 & 0.3 & 0.3 & -0.25 & 0.1 & 0.3 & -0.2 & 0.3 & 0.2 & 0.1 & 0.2 & -0.1 & 0 & 1 \\
	\hline
	\label{1}
	\vspace{-8mm}
\end{tabular}
\end{table}

\begin{table}[H]
\centering
\caption{Simulation parameters corresponding to $\alpha=2$}
\begin{tabular}{c c c c c c c c c c c c c c c c}
	\hline
	$A$ & $B_1$ & $B_2$ & $C$ & $D_1$ & $D_2$ & $\bar{C}$ & $\bar{D}_1$ & $\bar{D}_2$ & $Q$ & $R_1$ & $R_2$ & $S_1$ & $S_2$ & $G$ & $\xi$ \\
	\hline
	0.2 & 0.2 & -0.1 & 0.2 & 0.1 & -0.1 & 0.2 & 0.1 & -0.05 & 0.2 & 0.1 & 0.05 & 0.1 & -0.05 & 0 & 1  \\
	\hline
	\label{2}
	\vspace{-8mm}
\end{tabular}
\end{table}
Obviously, the coefficients satisfy the assumption (H3). Set the disturbance attenuation level $\gamma=1$. Through computational simulations, we have obtained the following figures. 

\begin{figure}[H]
	\centering
    \includegraphics[width=0.65\linewidth]{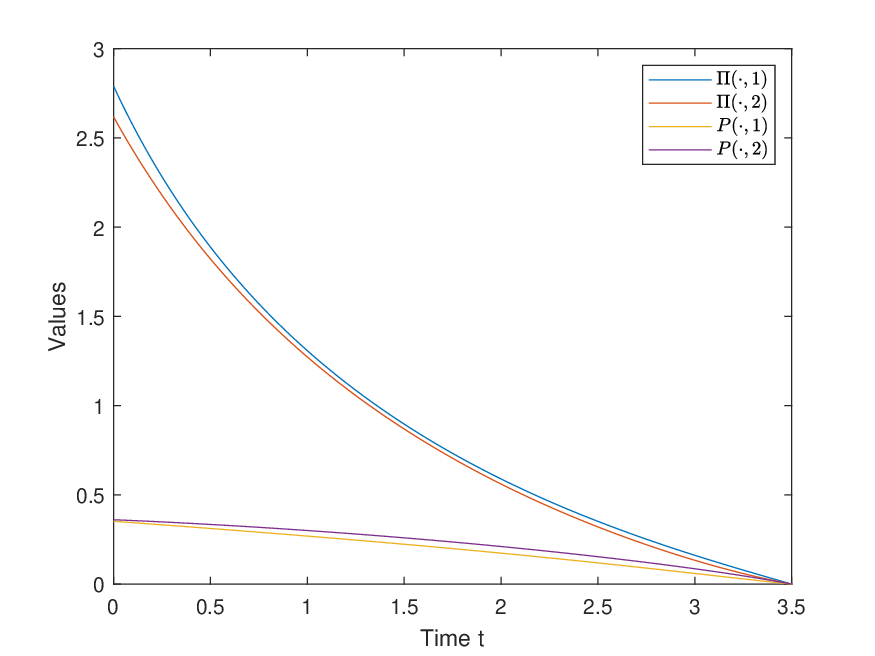}
	\caption{The numerical solutions of Riccati equations $\Pi(\cdot,1)$, $\Pi(\cdot,2)$, $P(\cdot,1)$ and $P(\cdot,2)$}
	\label{fig:1}
\end{figure}
Figure \ref{fig:1} gives the numerical solutions of Riccati equations $\Pi(\cdot,1)$, $\Pi(\cdot,2)$, $P(\cdot,1)$ and $P(\cdot,2)$.

\begin{figure}[H]
	\centering
	\includegraphics[width=0.65\linewidth]{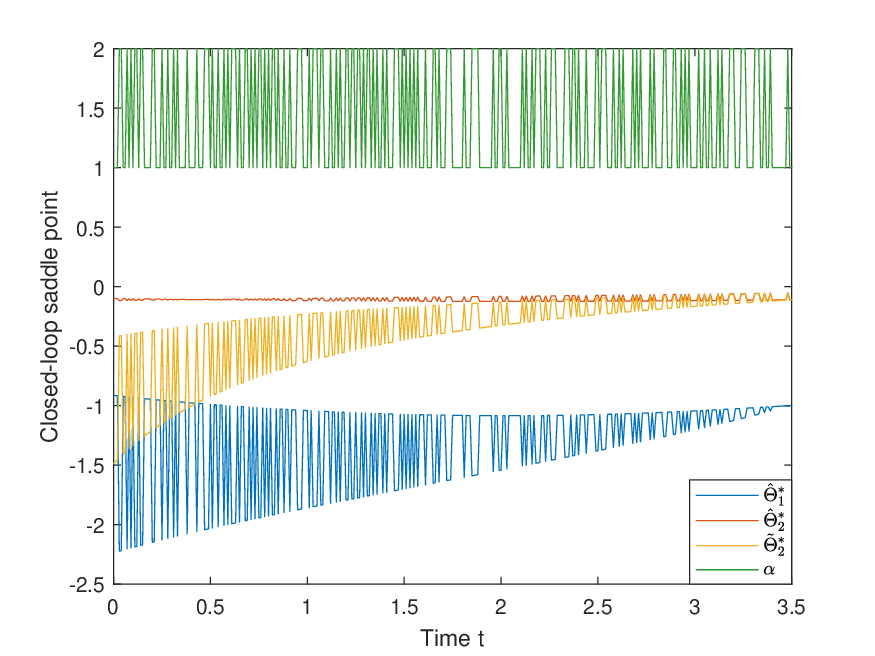}\vspace{-2mm}
	\caption{The closed-loop saddle point $(\hat{\Theta}^*(\cdot,\alpha(\cdot)),\tilde{\Theta}^*(\cdot,\alpha(\cdot)))$}
	\label{fig:2}
\end{figure}

\vspace{-2mm}

The closed-loop saddle point $(\hat{\Theta}^*(\cdot,\alpha(\cdot)),\tilde{\Theta}^*(\cdot,\alpha(\cdot)))$ of Problem (SCG) is  shown in Figure \ref{fig:2}. As can be seen from the figure, compared with the bull market ($\alpha=2$), the intensity of $\hat{\Theta}_1^*$ is relatively weaker while the intensities of $\hat{\Theta}_2^*$ and $\tilde{\Theta}_2^*$ are relatively stronger when the market is in a bear market ($\alpha=1$). This is because when the market is in a bearish state, market sentiment turns sluggish, negative news emerges frequently, and uncertainty factors mount, which leads to higher intensities of $\hat{\Theta}_2^*$ and $\tilde{\Theta}_2^*$. Meanwhile, the vast majority of investors run into losses and generally opt for a wait-and-see stance, thus resulting in a weaker intensity of $\hat{\Theta}_1^*$.

\vspace{-3mm}

\begin{figure}[H]
	\centering
	\includegraphics[width=0.65\linewidth]{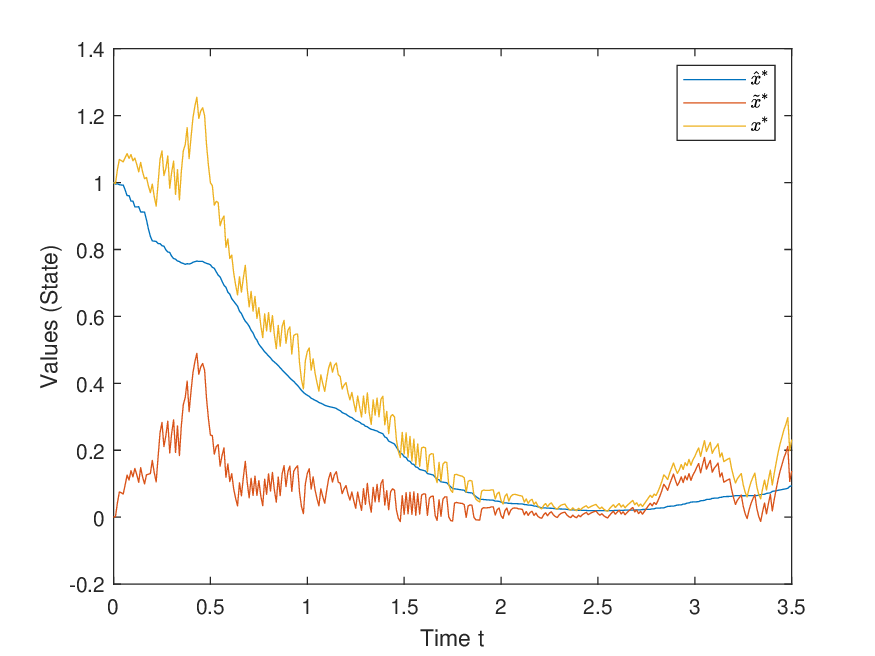}
	\caption{The state process $x^*(\cdot)$, the filtering state process $\hat{x}^*(\cdot)$ and the difference $\tilde{x}^*(\cdot)$}
	\label{fig:3}
\end{figure}
Figure \ref{fig:3} shows the corresponding state process $x^*(\cdot)$, the filtering state process $\hat{x}^*(\cdot)$ and the difference $\tilde{x}^*(\cdot)$ of the closed-loop system under the closed-loop point $(\hat{\Theta}^*(\cdot,\alpha(\cdot)),\tilde{\Theta}^*(\cdot,\alpha(\cdot)))$.

\begin{figure}[H]
	\centering
	\includegraphics[width=0.65\linewidth]{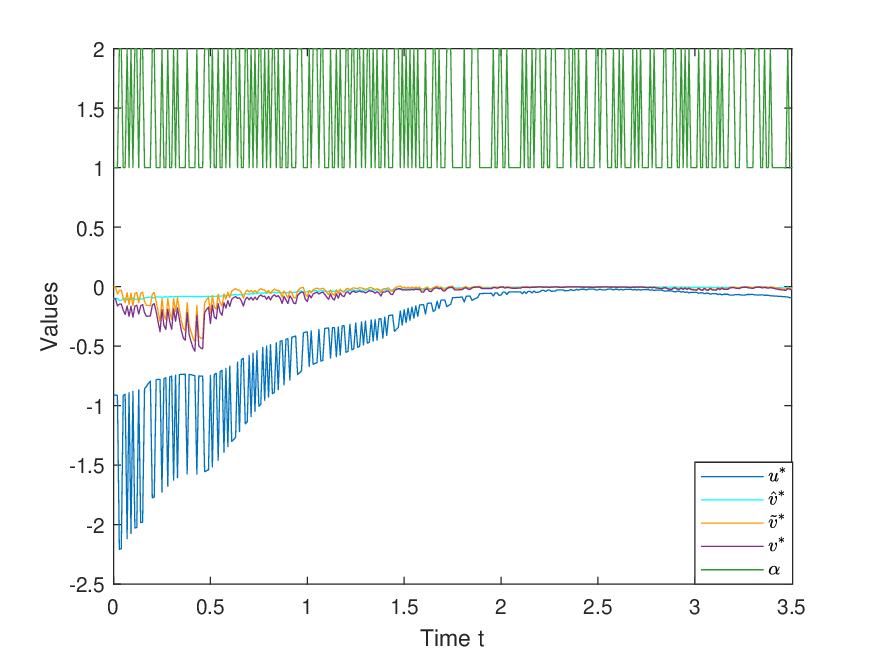}
	\caption{The robust $H_\infty$ optimal control $u^*(\cdot)$ and the worst-case disturbance $v^*(\cdot)$}
	\label{fig:4}
\end{figure}
Figure \ref{fig:4} illustrates the corresponding robust $H_\infty$ control strategy $u^*(\cdot)$ and the worst-case disturbance $v^*(\cdot)$ of Problem (R-SCG), together with the corresponding filtering process $\hat{v}^*(\cdot)$ and the difference $\tilde{v}^*(\cdot)$ associated with the worst-case disturbance $v^*(\cdot)$. As can be seen from the figure, when the market is in a bear market, the intensity of the worst-case disturbance perceived by institutional investors is relatively high; consequently, the investment intensity $u^*(\cdot)$ of institutional investors decreases. This is consistent with practical market conditions: in a bear market, insufficient investor confidence and pessimistic market sentiment lead to reduced willingness of participation on the part of both buyers and sellers, which in turn results in diminished market trading activity and lower trading volume.

To conduct a sensitivity analysis of the disturbance attenuation level for the Markovian jump system, the parameter $\gamma$ is hereby set to 2. In the same way, we can obtain the figures of the corresponding closed-loop saddle point $(\hat{\Theta}^*(\cdot,\alpha(\cdot)),\tilde{\Theta}^*(\cdot,\alpha(\cdot)))$ of Problem (SCG), the corresponding state processes ($x^*(\cdot)$, $\hat{x}^*(\cdot)$, $\tilde{x}^*(\cdot)$) of the closed-loop system under $(\hat{\Theta}^*(\cdot,\alpha(\cdot)),\tilde{\Theta}^*(\cdot,\alpha(\cdot)))$, and the solution $(u^*(\cdot), v^*(\cdot))$ of Problem (R-SCG) (i.e., the robust $H_\infty$ closed-loop optimal control $u^*(\cdot)\in\mathcal{U}[t,T]$ and the worst-case disturbance $v^*(\cdot)\in\mathcal{V}[t,T]$) corresponding to a disturbance attenuation level of 2. See Figures \ref{fig:5}-\ref{fig:7}, respectively.

\begin{figure}[H]
	\centering
	\includegraphics[width=0.65\linewidth]{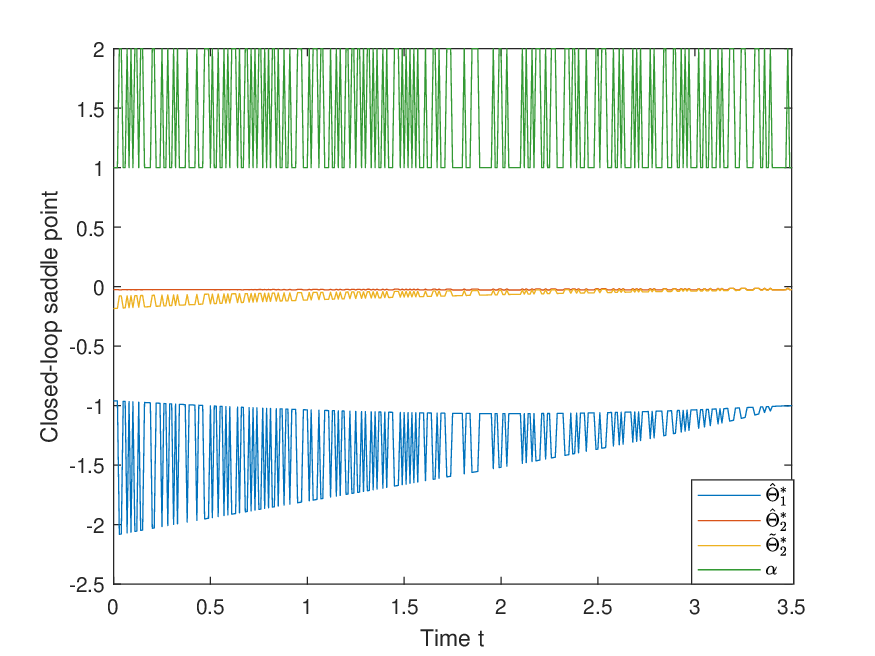}
	\caption{The closed-loop saddle point $(\hat{\Theta}^*(\cdot,\alpha(\cdot)),\tilde{\Theta}^*(\cdot,\alpha(\cdot)))$\;($\gamma=2$)}
	\label{fig:5}
\end{figure}

\begin{figure}[H]
	\centering
	\includegraphics[width=0.65\linewidth]{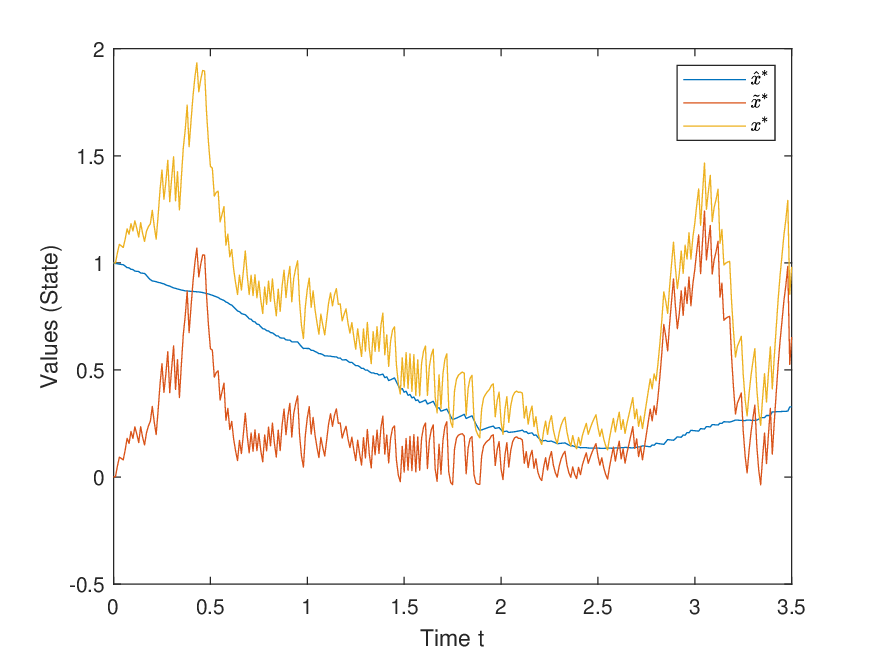}
	\caption{The state processes $x^*(\cdot)$, $\hat{x}^*(\cdot)$ and $\tilde{x}^*(\cdot)$ \;($\gamma=2$)}
	\label{fig:6}
\end{figure}

\begin{figure}[H]
	\centering
	\includegraphics[width=0.65\linewidth]{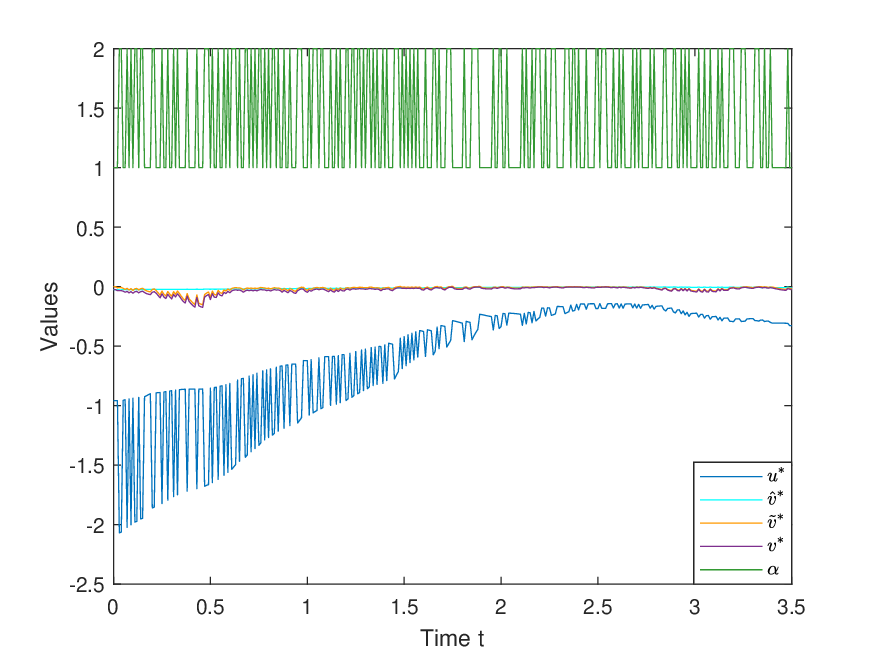}
	\caption{The robust $H_\infty$ optimal control $u^*(\cdot)$ and the worst-case disturbance $v^*(\cdot)$\;($\gamma=2$)}
	\label{fig:7}
\end{figure}

Next, we intend to conduct a sensitivity analysis on the disturbance attenuation level $\gamma$. With all other parameters fixed, we observe the impact of the disturbance attenuation level on the closed-loop saddle point, the corresponding state processes of the closed-loop system, the robust $H_\infty$ optimal control, and the worst-case disturbance. See Figures \ref{fig:8}-\ref{fig:10}, respectively. 

\begin{figure}[H]
	\centering
	\includegraphics[width=0.65\linewidth]{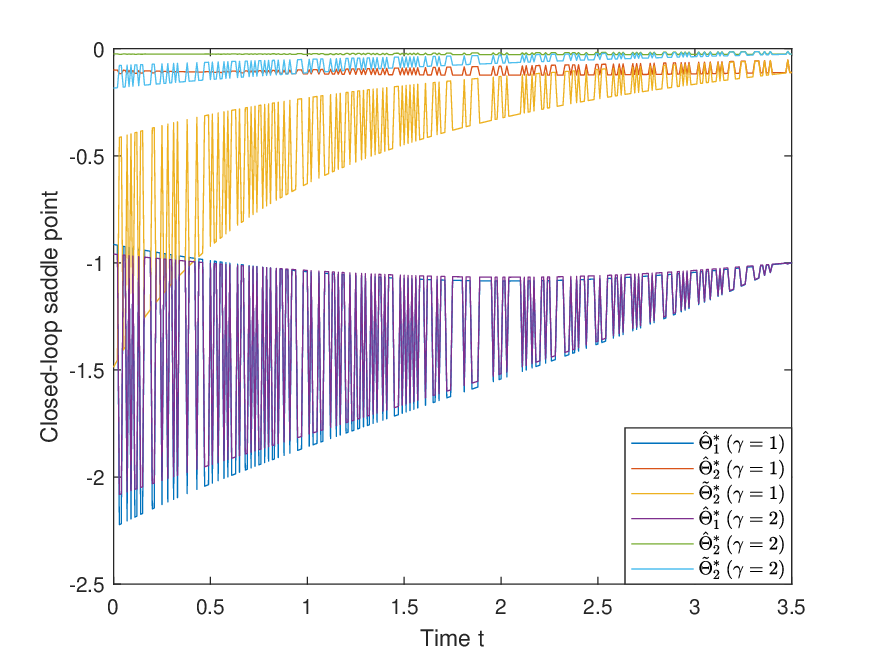}
	\caption{The impact of $\gamma$ on the closed-loop saddle point $(\hat{\Theta}^*(\cdot,\alpha(\cdot)),\tilde{\Theta}^*(\cdot,\alpha(\cdot)))$}
	\label{fig:8}
\end{figure}
From Figure \ref{fig:8}, we can see that the impact of the disturbance attenuation level on $\hat{\Theta}_1^*(\cdot,\alpha(\cdot))$ is relatively insignificant. However, as the disturbance attenuation level increases, the corresponding intensities of $\hat{\Theta}_2^*(\cdot,\alpha(\cdot))$ and $\tilde{\Theta}_2^*(\cdot,\alpha(\cdot))$ decrease significantly. 

\begin{figure}[H]
	\centering
	\includegraphics[width=0.65\linewidth]{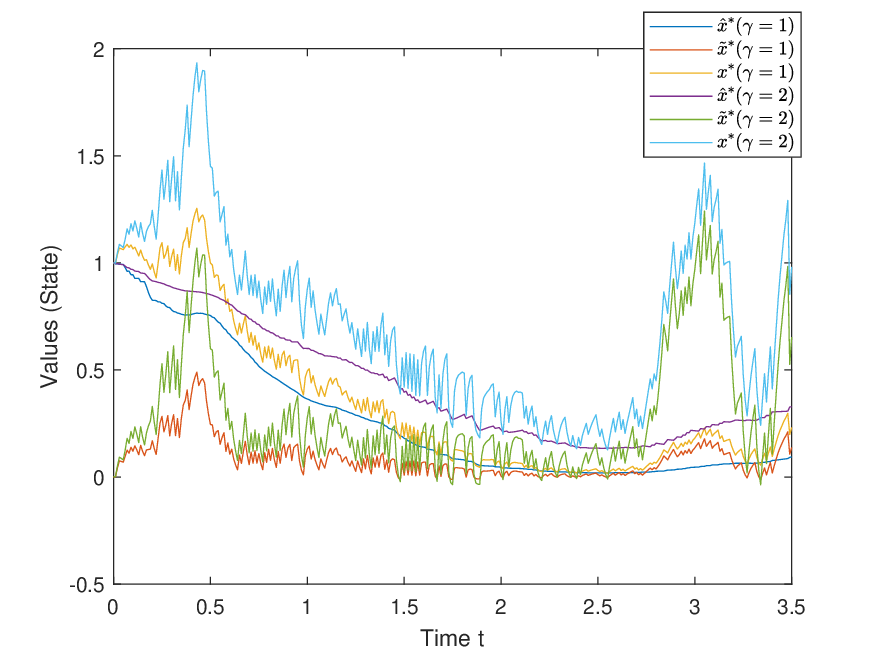}
	\caption{The impact of $\gamma$ on the state processes}
	\label{fig:9}
\end{figure}
It can be seen from Figure \ref{fig:9} that the higher the disturbance attenuation level, the greater the intensities of the corresponding state process $x^*(\cdot)$, filtering state process $\hat{x}^*(\cdot)$, and the difference process $\tilde{x}^*(\cdot)$ of the closed system under the closed-loop saddle point.

\begin{figure}[H]
	\centering
	\includegraphics[width=0.65\linewidth]{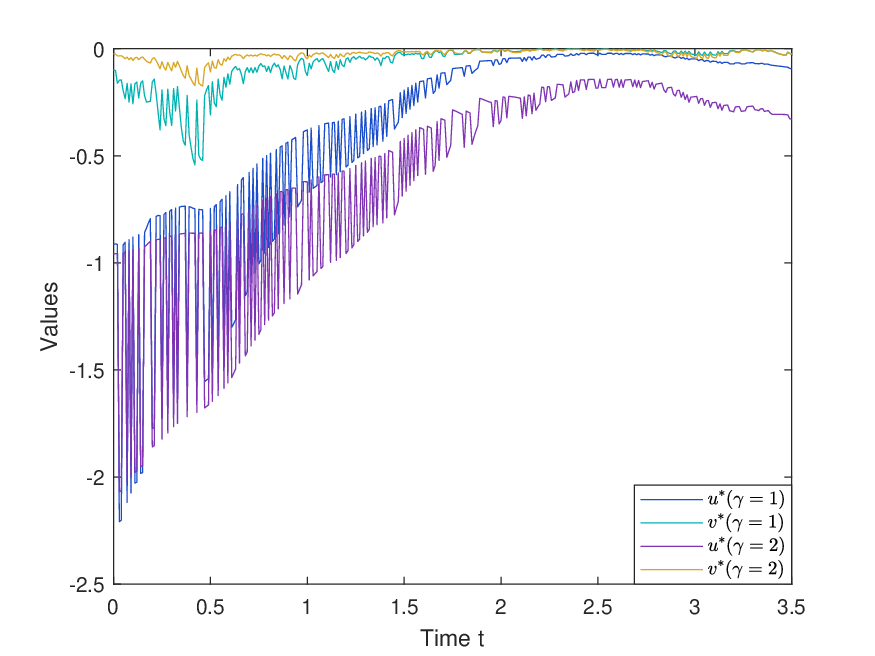}
	\caption{The impact of $\gamma$ on the $H_\infty$ control $u^*(\cdot)$ and the worst-case disturbance $v^*(\cdot)$}
	\label{fig:10}
\end{figure}
As can be seen from Figure \ref{fig:10}, when the market remains in the same state, the higher the disturbance attenuation level, the greater the intensity of the corresponding optimal control strategy; conversely, the intensity of the worst-case disturbance decreases under this condition. This phenomenon is highly consistent with practical market scenarios. A higher disturbance attenuation level indicates that institutional investors can tolerate a greater degree of impact of external disturbances on the system output. Under such circumstances, institutional investors no longer maintain a risk-averse stance; instead, they become highly optimistic about the stock market, believing that the intensity of the worst-case disturbance is relatively low, and thus naturally increase their investment intensity.

\section{Conclusions}

In this paper, we have studied an SLQ optimal control problem with Markov chain and model uncertainty under partial information, where both the drift and diffusion terms of the state equation, as well as the cost functional contain the control and the external unknown disturbance, and the information available to the control is based on a sub-$\sigma$-algebra of the filtration generated by the underlying Brownian motion and the Markov chain. By $H_\infty$ control theory and the zero-sum game approach, a soft-constrained zero-sum LQ stochastic differential game with Markov chain and partial information has been considered. By the filtering technique, the Riccati equation approach, the method of orthogonal decomposition, and the completion-of-squares method, the closed-loop saddle point of the zero-sum game has been derived by means of the optimal feedback control-strategy pair. Then, we have also demonstrated that the corresponding outcome of the closed-loop saddle point ensures the $H_\infty$ performance holds. At last, we have given a numerical example as further illustrations of theoretical results.

In the future, it is interesting to consider partially observed SLQ optimal control problems with Markovian regime switching and model uncertainty.

\end{document}